\documentclass[a4paper, 11pt]{amsart}
\usepackage[british]{babel}
\usepackage[utf8x]{inputenc}
\usepackage[T1]{fontenc}
\usepackage[
    backend=biber,
    style=alphabetic,
    maxnames=99,
    citestyle=alphabetic,
    maxalphanames=99,
    ]{biblatex}
\renewbibmacro*{volume+number+eid}{
    \printfield{volume}
    \setunit*{\addnbthinspace}
    \printfield{number}
    \setunit{\addcomma\space}
    \printfield{eid}}
  \DeclareFieldFormat[article]{number}{\mkbibparens{#1}}
\renewbibmacro{in:}{}

\usepackage[a4paper,top=3cm,bottom=3cm,left=3cm,right=3cm,marginparwidth=1.75cm]{geometry}

\usepackage{amsmath}
\numberwithin{equation}{section}
\usepackage{amsfonts}
\usepackage{bbm}
\usepackage{amssymb}
\usepackage{graphicx}
\usepackage{dsfont}
\usepackage[colorinlistoftodos]{todonotes}
\usepackage[colorlinks=true, allcolors=blue]{hyperref}
\usepackage{enumitem}
\usepackage{amsthm}
\usepackage{tikz-cd}
\usepackage{quiver}
\usepackage{mathrsfs}
\usepackage[dvipsnames]{xcolor}
\usepackage{hyperref}
\hypersetup{
    linkcolor=purple,
    urlcolor=purple,
    citecolor=purple,
    }

\usepackage{tikz}
\usetikzlibrary{knots}
\usetikzlibrary{matrix}
\usetikzlibrary{decorations.pathreplacing}
\usetikzlibrary{decorations.markings}
\usetikzlibrary{arrows}
\usetikzlibrary{calc}
\usetikzlibrary{shapes.misc}
\usetikzlibrary{fit}
\usepgflibrary{decorations.pathmorphing}
\usepgflibrary{shapes.geometric}
\usepackage{yhmath}
\usepackage{Cobordism}

\newcommand{\pa}[1]{\left( #1 \right)}

\newcommand{\set}[1]{\left\{ #1 \right\}}
\newcommand{\ol}[1]{\overline{#1}}

\newcommand{\Bord}[1]{\mathbf{Bord}_{1,2,3}^{\mathrm{#1}}}
\newcommand{\bord}[1]{\mathbf{Bord}_{2,3}^{\mathrm{#1}}}
\newcommand{\vect}{\mathbf{Vect}_k}
\newcommand{\Vect}{2\mathbf{Vect}_k}
\newcommand{\KV}{2\mathbf{Vect}_k^{\mathrm{KV}}}
\newcommand{\hatC}{\widehat{\mathbf{C}}}
\newcommand{\hatD}{\widehat{\mathbf{D}}}
\newcommand{\parin}{\partial_{\mathrm{in}}}
\newcommand{\parout}{\partial_{\mathrm{out}}}
\newcommand{\parIn}{\partial_{\mathrm{in}}'}
\newcommand{\parOut}{\partial_{\mathrm{out}}'}
\newcommand{\twovect}[2]{\begin{pmatrix}#1 \\ #2 \end{pmatrix}}
\newcommand{\ssig}{_{\mathrm{sig}}}
\newcommand{\ssigh}{_{\mathrm{sig/2}}}

\newcommand{\textcolour}[2]{\textcolor{#1}{#2}}

\newcommand\II{\ensuremath{\mathrm{II}}}
\newcommand\III{\ensuremath{\mathrm{III}}}

\newcommand\fixboundingbox{\path [use as bounding box, draw=none] (current bounding box.north west) rectangle (current bounding box.south east);}
\newcommand\selectpart[2][\selectcolour]{\fixboundingbox\begin{pgfonlayer}{selectionbox}\node [draw=red, fit=#2, inner sep=0.8*\cobordismlinewidth, #1, line width=\cobordismlinewidth] {};\end{pgfonlayer}}

\include{arrows}

\newenvironment{tz}[1][]{\begin{tikzpicture}[baseline={([yshift=-.8ex]current bounding box.center)},#1]}{\end{tikzpicture}}

\usepackage{etoolbox}

\makeatletter
\def\calign@preamble{%
   &\hfil\strut@
    \setboxz@h{\@lign$\m@th\displaystyle{##}$}%
    \ifmeasuring@\savefieldlength@\fi
    \set@field
    \hfil
    \tabskip\alignsep@
}
\let\cmeasure@\measure@
\patchcmd\cmeasure@{\divide\@tempcntb\tw@}{}{}{}
\patchcmd\cmeasure@{\divide\@tempcntb\tw@}{}{}{}
\patchcmd\cmeasure@{\ifodd\maxfields@
  \global\advance\maxfields@\@ne
  \fi}{}{}{}    
\newenvironment{calign}
{%
  \let\align@preamble\calign@preamble
  \let\measure@\cmeasure@
  \align
}
{%
  \endalign
}  
\makeatother

\def\smallbordisms{\scalecobordisms{0.5}\setlength\obscurewidth{0pt}}

\DeclareMathOperator{\Hom}{Hom}

\DeclareMathOperator{\id}{id}
\DeclareMathOperator{\Span}{span}
\DeclareMathOperator{\Ext}{Ext}
\DeclareMathOperator{\tr}{tr}

\newtheorem{thm}{Theorem}[section]
\newtheorem{lem}[thm]{Lemma}
\newtheorem{prop}[thm]{Proposition}
\newtheorem{cor}[thm]{Corollary}
\newtheorem{conj}[thm]{Conjecture}

\newtheorem{thmA}{Theorem}

\newtheorem{propA}[thmA]{Proposition}

\theoremstyle{definition}
\newtheorem{defn}[thm]{Definition}
\newtheorem{eg}[thm]{Example}
\newtheorem{constr}[thm]{Construction}

\theoremstyle{remark}
\newtheorem*{rk}{Remark}

\newcommand*{\newproofname}{Proof}
\newenvironment{proof*}[1][\newproofname]{\begin{proof}[#1]}{\end{proof}}

\title[The half-signature extension of $\Bord{or}$ and its representations]{The half-signature extension of the 3D bordism bicategory and its representations}
\author{Glen Lim}
\address{Mathematical Institute, University of Oxford}
\email{glen.lim@maths.ox.ac.uk}

\begin{document}

\begin{abstract}
    We construct the half-signature bordism bicategory $\Bord{sig/2}$ as a central extension of the 3-dimensional oriented bordism bicategory by $\mathbb{Z}$. This is an index $2$ subextension of the signature bordism bicategory. Following methods of Bartlett--Douglas--Schommer-Pries--Vicary, we construct a presentation of this bicategory and hence show that its linear representations are classified exactly by modular tensor categories.
\end{abstract}

\maketitle

\tableofcontents

\section{Introduction}

A major question in the study of topological quantum field theories (TQFTs) is their classification. Following axiomatic definitions of TQFTs due to Atiyah \cite{Atiyah-TQFT} and Segal \cite{Segal} and of extended TQFTs due to Lawrence \cite{Lawrence}, this is the classification of symmetric monoidal functors $$\mathbf{Bord}_{d-n,\ldots,d} \rightarrow n\mathbf{Vect}_k$$ from some bordism $n$-category of $(d-n)$-dimensional manifolds, $(d-n+1)$-dimensional bordisms, etc. (possibly with some extra topological structure) to some $n$-categorical analogue of the category of vector spaces over some field $k$. Classifications of (extended) TQFTs are then results of the following form: equivalence classes of symmetric monoidal functors $\mathbf{Bord}_{d-n,\ldots,d} \rightarrow n\mathbf{Vect}_k$ are in bijection with equivalence classes of some algebraic object.

In lower dimensions, various classification results of this flavour are known. When $n=1$, such classification results are known for $d=1,2,3$. 1-dimensional oriented TQFTs are in bijection with finite-dimensional vector spaces. 2-dimensional oriented TQFTs are in bijection with commutative Frobenius algebras \cite{Abrams,Kock}. 3-dimensional oriented TQFTs are in bijection with an algebraic structure known as $J$-algebras \cite{Juhasz}. In the $n=2,d=2$ case, once-extended 2-dimensional oriented TQFTs are in bijection with separable symmetric Frobenius algebras (up to a notion of Morita equivalence), and once-extended 2-dimensional unoriented TQFTs are in bijection with stellar separable Frobenius algebras (up to a notion of Morita equivalence) \cite{csp-phd}.

In the $n=1,d=3$ case, many examples of 3-dimensional oriented TQFTs are projective: the compatibility with composition only holds up to multiplication by (invertible) scalars, known as the \emph{anomaly}. In order to include these examples, the common approach is to replace the bordism category $\bord{or}$ with a suitable \emph{central extension} by $\mathbb{Z}$, equipping morphisms with extra data given by an element of a certain $\mathbb{Z}$-torsor. Examples include the $p_1$ bordism category $\mathbf{Bord}_{2,3}^{p_1}$ considered by Blanchet, Habegger, Masbaum and Vogel \cite{BHMV}, whose morphisms consist of bordisms with a $p_1$ structure; and Walker's \cite{Walker} signature bordism category $\bord{sig}$ whose morphisms are bordisms with the extra data of the signature of a bounding 4-manifold.

In their classification of once-extended 3-dimensional oriented TQFTs (i.e. the $n=2,d=3$ case), Bartlett, Douglas, Schommer-Pries and Vicary \cite{BDSV4} thus also classified symmetric monoidal functors out of the bicategorical analogues of the signature and $p_1$ bordism categories. In doing so, they found that simple representations of the oriented bordism bicategory $\Bord{or}$ and its two central extensions $\Bord{sig}$ and $\mathbf{Bord}_{1,2,3}^{p_1}$ are, respectively, classified by modular tensor categories with an extra condition (that the anomaly is trivial), or with extra data (a choice of square root or $6^\mathrm{th}$ root of the anomaly).

In this paper, we construct another symmetric monoidal extension of $\Bord{or}$ by $\mathbb{Z}$, the \emph{half-signature bordism bicategory} $\Bord{sig/2}$, whose simple representations are exactly classified by modular tensor categories. Our main results may be summarised as follows. First, we identify the half-signature bordism bicategory $\Bord{sig/2}$ as an index $2$ symmetric monoidal subbicategory of $\Bord{sig}$.
\begin{thmA}[{Construction \ref{constr:sig/2-out} \& Lemma \ref{lem:sig/2-mon}}] \label{thmA:exists-bordsig/2}
    There exists an index $2$ symmetric monoidal subextension of $\Bord{sig}$.
\end{thmA}
Then, following the approach of \cite{BDSV4}, we construct the \emph{global half-signature presentation} $\mathcal{G}$ and show that it presents the symmetric monoidal bicategory $\Bord{sig/2}$.
\begin{thmA}[{Theorem \ref{conj:pres-sig/2}}] \label{thmA:pres-sig/2}
    There exists an equivalence of symmetric monoidal bicategories $|-|_\mathcal{G}: \mathbf{F}(\mathcal{G}) \rightarrow \Bord{sig/2}$ between the symmetric monoidal bicategory generated by the global half-signature presentation and the half-signature bordism bicategory.
\end{thmA}
Finally, we present a modified version of the arguments in \cite{BDSV4} in order to classify the linear representations of $\Bord{sig/2}$.
\begin{thmA}[{Theorem \ref{thm:main-sig/2}}] \label{thmA:main-sig/2}
    Symmetric monoidal functors $\Bord{sig/2}\rightarrow \Vect$ are classified by finite direct sums of modular tensor categories whose anomalies are equal.
\end{thmA}

Additionally, we also construct the \emph{componentwise half-signature bordism bicategory} in analogy to the componentwise signature bordism bicategory $\Bord{csig}$ considered in \cite{BDSV4}. The corresponding results for this bicategory are given below.
\begin{thmA}[{Theorem \ref{conj:pres-csig/2}}] \label{thmA:pres-csig/2}
    There exists an equivalence of symmetric monoidal bicategories $|-|_\mathcal{H}: \mathbf{F}(\mathcal{H}) \rightarrow \Bord{csig/2}$ between the symmetric monoidal bicategory generated by the half-signature presentation and the componentwise half-signature bordism bicategory.
\end{thmA}
\begin{thmA}[{Theorem \ref{thm:main-csig/2}}] \label{thmA:main-csig/2}
    Symmetric monoidal functors $\Bord{csig/2}\rightarrow \Vect$ are classified by finite direct sums of modular tensor categories.
\end{thmA}

It should be noted that Theorems \ref{thmA:pres-sig/2} and \ref{thmA:pres-csig/2} (and so Theorems \ref{thmA:main-sig/2} and \ref{thmA:main-csig/2}, along with the corresponding results of \cite{BDSV4}) depend on the analogous result on the presentation of oriented bordism bicategory $\Bord{or}$, which we state as Theorem \ref{conj:pres-or}. The Cerf-theoretic work in proving this has largely been completed by Sytilidis \cite{filippos-thesis} and Haïoun \cite{haioun}; upcoming work by Bartlett, Douglas and Sytilidis \cite{pres-bord} will put these together to finish the proof. In a separate paper \cite{partB}, we present an alternate way of proving Theorems \ref{thmA:main-sig/2} and \ref{thmA:main-csig/2} (as well as the corresponding results of \cite{BDSV4}) independently of Theorem \ref{conj:pres-or}.

\subsection{Central extensions of the 3-dimensional bordism (bi)category}

We provide a deeper overview of the objects of study in this paper. When relevant, we provide a sketch of the methods used to prove our main results.

The main object in this paper is the half-signature bordism bicategory $\Bord{sig/2}$. This is a symmetric monoidal extension by $\mathbb{Z}$ of the oriented bordism bicategory $\Bord{or}$. In order to put this into context, we shall discuss the study of extensions of $\Bord{or}$.

An apt starting point for this would be the central extensions of mapping class groups. Diffeomorphisms of surfaces naturally give rise to bordisms by considering their mapping cylinders, and isotopic diffeomorphisms give bordisms which are diffeomorphic relative to the boundary. Mapping class groups may thus be viewed as automorphism groups of objects of $\bord{or}$ (which are in turn 1-morphisms of $\Bord{or}$).

The central extensions by $\mathbb{Z}$ of the mapping class group $\Gamma_g$ of a closed genus $g$ surface are classified by the second group cohomology $H^2(\Gamma_g;\mathbb{Z})$. This may be computed from the first and second homologies via the universal coefficient theorem. The computation of the first homology may be attributed to Mumford for $g=2$ \cite{Mumford} and Powell for $g\ge 3$ \cite{Powell}, though a quick proof based on a group presentation of $\Gamma_g$ can be found in Korkmaz's survey \cite{Korkmaz}. On the other hand, the second homology was computed for $g\ge 5$ by Harer \cite{Harer} and for lower genus (up to a torsion term in the case $g=3$) by Korkmaz and Stipsicz \cite{korkstip}. Overall, one concludes that $H^2(\Gamma_g;\mathbb{Z})$ is the cyclic group $\mathbb{Z}_{12}$ for $g=1$,  $\mathbb{Z}_{10}$ for $g=2$, and $\mathbb{Z}$ for all $g\ge3$.

An explicit construction of such a central extension of $\Gamma_g$ by $\mathbb{Z}$ is the \emph{signature central extension} given by Meyer \cite{Meyer} in terms of signatures of certain surface bundles. Meyer additionally showed that for $g\ge 3$, the pairing of the corresponding cocycle against $H_2(\Gamma_g;\mathbb{Z})$ has image $4\mathbb{Z}$. Together with the fact that $H^2(\Gamma_g;\mathbb{Z})\cong\mathbb{Z}$, this shows that the signature central extension corresponds to $4$ times a generator. This was later extended by Masbaum and Roberts \cite{MR} to arbitrary genus by constructing a group presentation for the signature central extension. As an element of $H^2(\Gamma_g;\mathbb{Z})$, the signature central extension is represented by the Shale--Weil cocycle \cite{Walker,LV}.

The relationship between the signature central extension and TQFTs was observed by Walker \cite{Walker}, building upon work of Atiyah \cite{Atiyah-2framings} and Witten \cite{Witten}. By paring down the dependence of Witten's framed 3-manifold invariants on the framing, Walker constructed the \emph{signature bordism category} $\bord{sig}$ which, at the level of mapping class groups, restricts to Meyer's signature central extension. Walker's model of $\bord{sig}$ was constructed in such a way so as to minimise the amount of extra information attached to each surface $\Sigma$; this turned out to be a choice of Lagrangian subspace of the symplectic vector space $H_1(\Sigma;\mathbb{R})$. We will instead be using a model of a more geometric flavour, which Walker constructed as an intermediate step. This associates to each surface $\Sigma$ a 3-manifold $M$ with $\partial M = \Sigma$, and to each 3-dimensional bordism an integer, which may be thought of as the signature of a bounding 4-manifold.

This was extended even further in \cite{BDSV3} to construct the \emph{signature bordism bicategory} $\Bord{sig}$, which we reproduce in Definition \ref{defn:bordsig}. An alternate model more similar in spirit to Walker's construction with Lagrangian subspaces can be found in \cite{dr-thesis}.

The construction and classification of subextensions of the signature central extension of mapping class groups and the signature bordism category has been well-studied in the literature. At the level of mapping class groups, the fact that the signature central extension corresponds to $4$ times a generator in group cohomology implies the existence of index $2$ and $4$ subextensions. Group presentations of both of these are given in \cite{MR}.

An index $2$ subextension $\bord{sig/2}$ of the signature bordism category $\bord{sig}$ was constructed explicitly by Gilmer \cite{Gilmer}. Under this construction, $\bord{sig/2}$ is the symmetric monoidal subcategory of $\bord{sig}$ whose 2-morphisms satisfy a certain condition relating the parity of the signature term to the Betti numbers of the underlying bordism. Gilmer showed that the composition of such morphisms also satisfies this parity condition, and hence this forms an index $2$ subextension of $\bord{sig}$. (Moreover, the compatibility with disjoint unions is evident from the parity condition, and hence this is in fact a symmetric monoidal subextension.) Upon restricting to mapping class groups, this also provides an explicit description of an index $2$ subextension of the signature central extension.

Independently of Gilmer, the problem of constructing categorical and bicategorical subextensions was also studied by Bartlett, Douglas, Schommer-Pries and Vicary in the early 2010s via homotopy-theoretic methods. Their approach studied symmetric monoidal extensions of the $\infty$- and $(\infty,2)$-categorical analogues of $\bord{or}$ and $\Bord{or}$ respectively; these are the bordism $\infty$-category $\mathscr{B}\mathit{ord}_{2,3}^\mathrm{or}$ and $(\infty,2)$-category $\mathscr{B}\mathit{ord}_{1,2,3}^\mathrm{or}$ whose morphisms and 2-morphisms respectively consist of moduli spaces of bordisms. Via a straightening-unstraightening argument sketched in \cite{csp-inv-tft}, symmetric monoidal extensions of these correspond to certain cohomology groups of Madsen--Tillmann spectra. On the other hand, by pulling back along the symmetric monoidal functors $\tau:\mathscr{B}\mathit{ord}_{2,3}^\mathrm{or}\rightarrow \bord{or}$ and $\tau: \mathscr{B}\mathit{ord}_{1,2,3}^\mathrm{or}\rightarrow \Bord{or}$ which take connected components on the spaces of morphisms/2-morphisms, symmetric monoidal extensions of $\bord{or}$ and $\Bord{or}$ give rise to their symmetric monoidal extensions of $\infty$- and $(\infty,2)$-categorical analogues.

Following the homotopy-theoretic computations in \cite{BDSV3}, which later appear in \cite{csp-inv-tft}, the signature $\infty$- and $(\infty,2)$-categories $\mathscr{B}\mathit{ord}_{2,3}^\mathrm{or}$ and $\mathscr{B}\mathit{ord}_{1,2,3}^\mathrm{or}$ both correspond to $2$ times a generator in their respective cohomology groups. In particular, this implies the existence of a symmetric monoidal $(\infty,2)$-subextension $\mathscr{B}\mathit{ord}_{1,2,3}^\mathrm{sig/2}$. The linear representations of this were then stated, but not proven, in private communication to correspond exactly to modular tensor categories (i.e. Theorem \ref{thmA:main-sig/2}).

Additionally, \cite{BDSV3} contains a proposed construction which they conjectured to provide an isomorphism between the groups of symmetric monoidal extensions of $\mathscr{B}\mathit{ord}_{1,2,3}^\mathrm{or}$ and $\Bord{or}$.
\begin{conj}[{\cite{BDSV3}}] \label{conj:inf-2}
    The symmetric monoidal extensions of $\mathscr{B}\mathit{ord}_{1,2,3}^\mathrm{or}$ and $\Bord{or}$ are related in the following way:
    \begin{enumerate}
        \item Given a symmetric monoidal extension of $\mathscr{B}\mathit{ord}_{1,2,3}^\mathrm{or}$, its truncation is a symmetric monoidal extension of $\Bord{or}$.
        \item This construction induces an isomorphism between groups of symmetric monoidal extensions of $\mathscr{B}\mathit{ord}_{1,2,3}^\mathrm{or}$ and $\Bord{or}$ by $\mathbb{Z}$.
    \end{enumerate}
\end{conj}
It was then expected that the image of $\mathscr{B}\mathit{ord}_{1,2,3}^\mathrm{sig/2}$ under this conjectural isomorphism would also satisfy the statement of Theorem \ref{thmA:main-sig/2}.

More recently, half-signature structures were studied by Freed, Scheimbauer and Teleman. In the initial version of \cite{FST}, these were defined in terms of the twisted Anderson dual cohomologies of the classifying space $BSO(3)$. We understand that future work of the same authors will address this approach in greater detail.

In this paper, our construction of the half-signature bordism bicategory $\Bord{sig/2}$ is an explicit geometric construction, similar in approach to Gilmer's construction of $\bord{sig/2}$. Our proof of Theorem \ref{thmA:exists-bordsig/2} identifies $\Bord{sig/2}$ as the symmetric monoidal subbicategory of $\Bord{sig}$ consisting of 2-morphisms whose signature terms have the same parity as some expression in terms of the Betti numbers of the underlying bordism. Thus, this directly extends Gilmer's construction to bicategories.

Additionally, in providing an explicit geometric construction, we are then able to modify the presentations in \cite{BDSV4} to obtain a presentation of $\Bord{sig/2}$. In particular, the construction of the equivalence $|-|_\mathcal{G}: \mathbf{F}(\mathcal{G})\rightarrow \Bord{sig/2}$ in Construction \ref{constr:modg} relies on the parity condition. We then check that it is compatible with the corresponding equivalence relating $\Bord{or}$ to its presentation, and use this to deduce Theorem~\ref{thmA:pres-sig/2} from Theorem~\ref{conj:pres-or}.

Finally, this construction also proves the first half of Conjecture \ref{conj:inf-2}. Indeed, $\mathscr{B}\mathit{ord}_{1,2,3}^\mathrm{sig/2}$ naturally arises as the pullback of $\Bord{sig/2}$, so truncating it yields $\Bord{sig/2}$ again. Since all symmetric monoidal extensions of $\mathscr{B}\mathit{ord}_{1,2,3}^\mathrm{or}$ are multiples of $\mathscr{B}\mathit{ord}_{1,2,3}^\mathrm{sig/2}$, the same must hold for them as well.

As for the index $4$ subextension of the signature central extension of mapping class groups, an explicit cocycle for it was constructed by Turaev \cite{Turaev-Maslov}. Based on this, Gilmer and Masbaum \cite{GM} constructed the index $4$ subextension of the signature central extension by writing down a condition relating the signature term modulo $4$ to the image of the underlying mapping class group element under the standard symplectic representation. They then asked if such a construction could be extended to $\bord{sig}$.

This was answered by Schommer-Pries \cite{csp-inv-tft} in the negative through the abovementioned homotopy-theoretic methods. Following the computations of \cite{csp-inv-tft} (which in turn are based on those in \cite{BDSV3}), the index $2$ subextension of $\mathscr{B}\mathit{ord}_{2,3}^\mathrm{sig}$ corresponds to a generator of the relevant cohomology group. Thus, it cannot admit a symmetric monoidal subextension at the level of symmetric monoidal $\infty$-categories, and hence at the level of symmetric monoidal categories as well.

We provide an alternate proof of this fact, working purely at the level of cocycles. In fact, this proof makes no reference to the monoidal structure, so it shows that $\bord{sig}$ does not admit an index $4$ subextension at the level of ordinary categories:

\begin{propA}[{Proposition \ref{prop:indiv}}]
    There does not exist an index $4$ subextension of $\bord{sig}$.
\end{propA}

\subsection{Overview of this paper}

The structure of this paper is as follows. In \textsection\ref{section:ext-bicat}, we provide a description of (central) extensions of bicategories via cocycles, which serves as the framework on which $\Bord{sig/2}$ will be defined. This endows the set of equivalence classes of extensions with a group structure; we show via direct computation in Appendix \ref{section:equiv-ext} that an equivalence of bicategories induces an isomorphism of these groups.

The next three sections are devoted to $\Bord{sig/2}$. In \textsection\ref{section:sig}, we recall the definition of $\Bord{sig}$ from \cite{BDSV3} and rewrite it in terms of cocycles. In \textsection\ref{section:sig/2}, we then construct an explicit cocycle for $\Bord{sig/2}$ and realise the half-signature bordism bicategory as an index $2$ symmetric monoidal subextension of $\Bord{sig}$, hence giving it a symmetric monoidal structure. We also construct a componentwise variant $\Bord{csig/2}$, which has representation-theoretic significance. In \textsection\ref{section:indiv}, we show that $\bord{sig/2}$ is indivisible as an extension of $\bord{or}$, and hence $\Bord{sig/2}$ is indivisible as well.

In the final two sections, we adapt the methods in \cite{BDSV4} to classify the linear representations of $\Bord{sig/2}$ and $\Bord{csig/2}$. In \textsection\ref{section:pres}, we construct presentations for these two symmetric monoidal bicategories. Then, in \textsection\ref{section:rep-mtc}, we use these presentations to classify the linear representations of $\Bord{sig/2}$ and $\Bord{csig/2}$.

\subsection{Notations and conventions} Throughout this paper, we adopt the following conventions.

\subsubsection*{Base field} Throughout, $k$ denotes an algebraically closed field of arbitrary characteristic. All notions of linearity are over $k$.

\subsubsection*{Cardinality}  All categories and bicategories considered here are essentially small. Throughout, we implicitly take small models of these. In particular, this allows us to define extensions of bicategories in terms of cocycles. 

\subsubsection*{Equivalent models of bicategories} At various points, it will be useful to consider alternate models of the oriented bordism bicategory and its extensions. We use slightly different notations to make clear which model we are working with, but it will be helpful to bear in mind that $\Bord{or},\Bord{or,exp},\Bord{or,lag}$ all refer to equivalent symmetric monoidal bicategories. The same comment holds for $\Bord{sig}$ and for $\Bord{sig/2}$.

\subsubsection*{The bicategorical analogue of $\vect$} As in \cite{BDSV4}, we let $\Vect$ denote the symmetric monoidal bicategory of Cauchy-complete $k$-linear categories, functors, and natural transformations. There are other bicategorical analogues in the literature, but as shown in \cite[Appendix A]{BDSV4}, these all admit symmetric monoidal functors to $\Vect$ which are fully faithful on 2-morphisms. Using one of these other symmetric monoidal bicategories instead would not change our results.

\subsection{Acknowledgements}

We are grateful to André Henriques for proposing this project, as well as continued guidance and support throughout this process. Thanks are also due to Christopher Douglas for many helpful discussions and useful insights, as well as his feedback on the historical context of our results and the structure of this paper. We also thank Filippos Sytilidis for helpful discussions as well as Luciana Basualdo Bonatto and Thomas Wasserman for their comments on a previous iteration of this paper.

\section{Extensions of bicategories} \label{section:ext-bicat}

We describe a notion of extensions of a bicategory $\mathbf{C}$ by an abelian group $A$. Roughly speaking, this is a bicategory $\hatC$ with the same objects and 1-morphisms as $\mathbf{C}$, but now the 2-morphisms are instead replaced by $A$-families of $2$-morphisms with compatible $A$-actions. This may be viewed as a generalisation of central extensions of groups. As with the case for groups, we may classify such extensions with a notion of cocycles, which endows the set of isomorphism classes of extensions $\Ext(\mathbf{C},A)$ with a group structure and relates it to the cohomology of $\mathbf{C}$.

\subsection{Extensions of categories}

We first provide a quick overview of extensions of ordinary categories, of which central extensions of groups are a special case.

\begin{defn}
    Let $\mathbf{C}$ be a category and $A$ an abelian group. An \emph{extension of $\mathbf{C}$ by $A$} is a category $\hatC$ with a functor $q: \hatC \rightarrow \mathbf{C}$ such that:
    \begin{itemize}
        \item For every two objects $x,y$, there is a free $A$-action on $\Hom_{\hatC}(x,y)$.
        \item This $A$-action is compatible with composition, i.e. $ (a\cdot g)\circ f = g\circ(a\cdot f) = a\cdot(g\circ f)$ for all composable $f,g$ and for all $a\in A$.
        \item $q: \hatC \rightarrow \mathbf{C}$ is a quotient of $\hatC$ by the action of $A$. In particular, $q$ is a bijection on objects.
    \end{itemize}

    Given two extensions $\hatC, \hatC'$ of $\mathbf{C}$ with quotient functors $q, q'$, an \emph{isomorphism} between them is an equivalence of categories $F: \hatC \rightarrow \hatC'$ such that $q = q'\circ F$ and $F(a\cdot f) = a\cdot Ff$ for each morphism $f$ and each $a\in A$. Let $\Ext(\mathbf{C},A)$ denote the set of isomorphism classes of extensions of $\mathbf{C}$ by $A$.
\end{defn}

This is indeed a generalisation of central extensions of groups:
\begin{eg}
    Let $0\rightarrow A\rightarrow \widehat{G} \overset{q}\rightarrow G \rightarrow 0$ be a central extension of $G$ by $A$. Then, $Bq: B\widehat{G}\rightarrow BG$ is an extension of $BG$ by $A$.
\end{eg}

The data of a central extension $q:\widehat{G}\rightarrow G$ of groups may be captured by a \emph{cocycle} $c:G\times G \rightarrow A$: for each $g\in G$, pick a lift $\widehat{g}\in \widehat{G}$, and for each $g,h\in G$, let $c(g,h)$ be the unique element $a\in A$ such that $\widehat{g}\widehat{h} = a\cdot\widehat{gh}$. Associativity of multiplication in $\widehat{G}$ is equivalent to the cocycle condition
\begin{equation}\label{eq:cocyc-cat}
    c(g,f) + c(h,gf) = c(h,g) + c(hg,f).
\end{equation}
Two cocycles determine isomorphic extensions if and only if they differ by a \emph{coboundary}: these are functions of the form
\begin{equation}\label{eq:cobound-cat}
    c(g,h) = a(g) - a(h)
\end{equation}
where $a:G\rightarrow A$ is a function corresponding to a change in choice of lift of each $g\in G$. The set of extensions of central extensions of $G$ by $A$ (up to isomorphism) is thus in bijection with the quotient of the abelian group of cocycles by the subgroup of coboundaries, and so obtains an abelian group structure. This is in fact isomorphic to the group cohomology $H^2(G;A)$.

An analogous result holds for categories, as laid out in \cite{cohom-cat}. Given a category $\mathbf{C}$, the data of an extension is captured by a cocycle $c$, which takes as input two composable morphisms in $\mathbf{C}$ and outputs an element of $A$. This has to satisfy the cocycle condition (\ref{eq:cocyc-cat}) for each composable $f,g,h$. Given any function $a$ assigning to each morphism of $\mathbf{C}$ an element of $A$, its coboundary is given as in (\ref{eq:cobound-cat}). Then, $\Ext(\mathbf{C},A)$ is in bijection with the quotient of the group of cocycles by the subgroup of coboundaries, and so obtains an abelian group structure. This is exactly $H^2(N(\mathbf{C});A)$, the second cohomology of the nerve of $\mathbf{C}$.

In the rest of this section, we will be extending these ideas to bicategories.

\subsection{Extensions of bicategories}

We define extensions of bicategories analogously to extensions of categories. We also define a related concept, monoidal extensions of monoidal bicategories. For the rest of this section, we will use the same notation and terminology for bicategories and lax functors as in \cite[\textsection 2.1 \& \textsection 4.1]{JY}.

In the case for ordinary categories, an extension $\hatC$ of $\mathbf{C}$ has the same objects, but the morphisms are replaced by $A$-families of morphisms. For bicategories, we instead want $\hatC$ to have the same objects and 1-morphisms as $\mathbf{C}$, and $A$-families of 2-morphisms. In particular, we require that both bicategories have the same identity 1-morphisms and that their 1-morphisms compose in the same way. Hence, we want the quotient map $q:\hatC\rightarrow\mathbf{C}$ to be a \emph{strict} functor, i.e. the natural transformations associated to lax functoriality and lax unitality are identity 2-morphisms. 

\begin{defn}
    Let $\mathbf{C}$ be a bicategory and $A$ be an abelian group. An \emph{extension of $\mathbf{C}$ by $A$} is a bicategory $\hatC$ with a strict functor $q: \hatC \rightarrow \mathbf{C}$ such that:
    \begin{itemize}
        \item For every pair of 1-morphisms $f,g$ with the same source and target, there is a free $A$-action on $\Hom_{\hatC}(f,g)$.
        \item This $A$-action is compatible with both horizontal and vertical composition, i.e. $(g \cdot \beta) \circ \alpha = g \cdot (\beta \circ \alpha) = \beta \circ (g \cdot \alpha)$ for all vertically composable $\beta,\alpha$ and for all $g\in A$, and similarly, $(g \cdot \beta) \star \alpha = g \cdot (\beta \star \alpha) = \beta \star (g \cdot \alpha)$ for all horizontally composable $\beta,\alpha$ and for all $g\in A$.
        \item $q: \hatC \rightarrow \mathbf{C}$ is a quotient of $\hatC$ by the action of $A$.
    \end{itemize}

    Given two extensions $\hatC, \hatC'$ of $\mathbf{C}$ with quotient functors $q, q'$, an \emph{equivalence} between them is an equivalence of bicategories $F: \hatC \rightarrow \hatC'$ such that $q = q'\circ F$ and $F(g\cdot\alpha) = g\cdot F\alpha$ for each 2-morphism $\alpha$ and each $g\in A$. In particular, this must map the corresponding objects and 1-morphisms of $\hatC$ and $\hatC'$ to each other. However, we allow the component 2-cells $F^2_{g,f}, F^0_x$ associated to the laxity constraints not to be identity 2-morphisms. Instead, they may be any lifts of $\id_{gf}, \id_{\id_x}$ respectively. Let $\Ext(\mathbf{C},A)$ denote the set of equivalence classes of extensions of $\mathbf{C}$ by $A$.
\end{defn}

The oriented bordism bicategory (as defined in the next section) also has a symmetric monoidal structure. For such bicategories, we may also require the $A$-action to be compatible with the symmetric monoidal structure:

\begin{defn} \label{defn:ext}
    Let $\mathbf{C}$ be a (symmetric) monoidal bicategory and $A$ an abelian group. A \emph{(symmetric) monoidal extension of $\mathbf{C}$ by $A$} is a (symmetric) monoidal bicategory $\hatC$ with a (symmetric) monoidal functor $q: \hatC \rightarrow \mathbf{C}$ such that $\hatC$ is an extension of $\mathbf{C}$ and the $A$-action on $\hatC$ is compatible with the (symmetric) monoidal structure, i.e. $(g \cdot \alpha) \otimes \beta = g \cdot (\alpha \otimes \beta) = \alpha \otimes (g \cdot \beta)$ for all 2-morphisms $\alpha, \beta$ and for all $g\in A$.
\end{defn}

\subsection{Extensions as cocycles} As with groups and ordinary categories, we may describe an extension of a bicategory $\mathbf{C}$ (non-uniquely) via a cocycle. We describe the data of a cocycle, and the conditions it must satisfy.

Let $\mathbf{C}$ be a small bicategory and $q: \hatC \rightarrow \mathbf{C}$ be a extension by $A$. Choose a section of this, i.e. for each 2-morphism $\alpha$ of $\mathbf{C}$, choose some $\widehat{\alpha} \in q^{-1}(\alpha)$. Then, we may write the 2-morphisms of $\hatC$ as pairs $(\alpha, g) = g \cdot \widehat{\alpha}$, where $\alpha$ is a 2-morphism in $\mathbf{C}$ and $g\in A$. Additionally, for each 1-morphism $f$ in $\mathbf{C}$, we set $\widehat{\id}_f$ to be the identity 2-morphism of $f$ in $\hatC$. (We may choose to omit this requirement, but this will result in the conditions for cocycles to be even more complicated than they already are.)

To describe $\hatC$, one has to specify the data of horizontal and vertical composition of 2-morphisms, as well as the associators and unitors. Note that composition is determined purely by the 2-morphisms $(\alpha,0) = \widehat{\alpha}$: Suppose $$(\beta, 0) \circ (\alpha, 0) = (\beta \circ \alpha, c(\beta, \alpha)),$$ where $c(\beta, \alpha) \in A$. Then for all $a, b \in A$, we have $$(\beta, b) \circ (\alpha, a) = (\beta \circ \alpha, a + b + c(\beta, \alpha)).$$ The analogous result for horizontal composition holds.

In other words, it suffices to specify the following data:
\begin{itemize}
    \item For each pair of vertically composable 2-morphisms $\alpha,\beta$ in $\mathbf{C}$, some $c(\beta, \alpha) \in A$ such that $$(\beta, 0) \circ (\alpha, 0) = (\beta \circ \alpha, c(\beta, \alpha)).$$
    \item For each pair of horizontally composable 2-morphisms $\alpha,\beta$ in $\mathbf{C}$, some $\tilde{c}(\beta, \alpha) \in A$ such that $$(\beta, 0) \star (\alpha, 0) = (\beta \star \alpha, \tilde{c}(\beta, \alpha)).$$
    \item For each composable 1-morphisms $f,g,h$ in $\mathbf{C}$ (and so $\hatC$), some $a(f,g,h) \in A$ such that $(\alpha_{h,g,f}, a(h,g,f))$ is the associator of $f,g,h$ in $\hatC$, where $\alpha_{h,g,f}$ is the associator of $f,g,h$ in $\mathbf{C}$.
    \item For each 1-morphism $f$ in $\mathbf{C}$ (and so $\hatC$), some $\ell(f), r(f) \in A$ such that $(\lambda_f, \ell(f))$, $(\rho_f, r(f))$ are the left and right unitors of $f$ in $\hatC$, where $\lambda_f, \rho_f$ are the left and right unitors of $f$ in $\mathbf{C}$. 
\end{itemize}

As $\hatC$ is a bicategory, these have to satisfy the following properties:

    \begin{itemize}
    \item (Associativity of vertical composition) Let $\alpha, \beta, \gamma$ be vertically composable 2-morphisms in $\mathbf{C}$. Then the associativity of $\widehat{\alpha}, \widehat{\beta}, \widehat{\gamma}$ is equivalent to
    \begin{equation}\label{eq:cocyc-1}
        c(\gamma, \beta) + c(\gamma \circ \beta, \alpha) = c(\beta, \alpha) + c(\gamma, \beta \circ \alpha).
    \end{equation}
    For ease of writing, we denote both sides by $c(\gamma, \beta, \alpha)$.
    \item (Unitality of vertical composition) Let $\alpha: f \rightarrow g$ be a 2-morphism in $\mathbf{C}$. Then as $\widehat{\id}_f, \widehat{\id}_g$ are identities in $\hatC$, we have
    \begin{equation}\label{eq:cocyc-2}
        c(\id_g,\alpha) = c(\alpha, \id_f) = 0.
    \end{equation}
    \item (Middle four exchange) Let $f,f',f'': x \rightarrow y$ and $g,g',g'': y \rightarrow z$ be 1-morphisms and $\alpha:f\rightarrow f', \alpha': f'\rightarrow f'', \beta: g \rightarrow g', \beta': g' \rightarrow g''$ be 2-morphisms in $\mathbf{C}$. Then the equality $$\pa{\widehat{\beta}'\circ\widehat{\beta}}\star\pa{\widehat{\alpha}'\circ\widehat{\alpha}} = \pa{\widehat{\beta}'\star\widehat{\alpha}'}\circ\pa{\widehat{\beta}\star\widehat{\alpha}}$$ is equivalent to
    \begin{equation} \label{eq:middle-four}
        c(\alpha',\alpha) + c(\beta',\beta) + \tilde{c}(\beta'\circ\beta,\alpha'\circ\alpha) = \tilde{c}(\beta,\alpha) + \tilde{c}(\beta',\alpha') + c(\beta'\star\alpha',\beta\star\alpha).
    \end{equation}
    In particular, putting $\alpha = \alpha' = \id_f$ and $\beta = \beta' = \id_g$ gives $\tilde{c}(\id_g, \id_f) = 0$.
    \item (Associativity of horizontal composition) Let $\alpha: f \rightarrow f'$, $\beta: g \rightarrow g'$, $\gamma: h \rightarrow h'$ be horizontally composable 2-morphisms in $\mathbf{C}$. Then the equality $$\pa{\widehat{\gamma}\star\pa{\widehat{\beta}\star\widehat{\alpha}}} \circ (\alpha_{h,g,f}, a(h,g,f)) = (\alpha_{h',g',f'}, a(h',g',f')) \circ \pa{\pa{\widehat{\gamma}\star\widehat{\beta}}\star\widehat{\alpha}}$$ is equivalent to 
    \begin{equation} \label{eq:horiz-assoc}
        \begin{aligned}
            &\tilde{c}(\beta,\alpha) + \tilde{c}(\gamma,\beta\star\alpha) + a(h,g,f) + c(\gamma\star(\beta\star\alpha), \alpha_{h,g,f}) \\
            &\quad= \tilde{c}(\gamma,\beta) + \tilde{c}(\gamma\star\beta,\alpha) + a(h',g',f') + c(\alpha_{h',g',f'}, (\gamma\star\beta)\star\alpha).
        \end{aligned}
    \end{equation}
    \item (Naturality of unitors) Let $\alpha: f \rightarrow f'$ be a 2-morphism in $\mathbf{C}$ where $f,f': x \rightarrow y$. Then the equality $$(\lambda_f, \ell(f)) \circ \widehat{\alpha} = \pa{\widehat{\id}_{\id_y} \star \widehat{\alpha}} \circ (\lambda_{f'}, \ell(f'))$$ is equivalent to
    \begin{equation}\label{eq:cocyc-5}
        c(\lambda_f, \alpha) + \ell(f) = \tilde{c}(\id_{\id_y}, \alpha) + c(\id_{\id_y} \star \alpha, \lambda_{f'}) + \ell(f').
    \end{equation}
    Likewise, we must have for the right unitors
    \begin{equation}
        c(\alpha, \rho_f) + r(f) = \tilde{c}(\alpha,\id_{\id_y}) + c(\rho_{f'}, \alpha \star \id_{\id_y}) + r(f').
    \end{equation}
    \item (Unity axiom) Let $f: x \rightarrow y$ and $g: y \rightarrow z$ be 1-morphisms in $\mathbf{C}$. Then the unity axiom $$(\rho_g, r(g)) \star \widehat{\id}_f = \pa{\widehat{\id}_g\star(\lambda_f, \ell(f))}\circ (\alpha_{g,\id_y,f}, a(g,\id_y,f))$$ is equivalent to
    \begin{equation}\label{eq:cocyc-unity}
        \tilde{c}(\rho_g, \id_f) + r(g) = \tilde{c}(\id_g, \lambda_f) + c(\id_g\star\lambda_f, \alpha_{g,\id_y,f}) + a(g,\id_y,f) + \ell(f).
    \end{equation}
    \item (Pentagon axiom) Let $f,g,h,i$ be composable 1-morphisms in $\mathbf{C}$. Then the pentagon axiom is equivalent to
    \begin{equation}\label{eq:cocyc-0}
        \begin{aligned}
            &a(ih,g,f) + a(i,h,gf) + c(\alpha_{ih,g,f},\alpha_{i,h,gf}) \\
            &\quad= a(i,h,g) + a(i,hg,f) + a(h,g,f) \\
            &\quad\qquad + \tilde{c}(\alpha_{i,h,g},\id_f) + \tilde{c}(\id_i, \alpha_{h,g,f}) + c(\id_i\star\alpha_{h,g,f}, \alpha_{i,hg,f}, \alpha_{i,h,g}\star\id_f).
        \end{aligned}
    \end{equation}
\end{itemize}

Conversely, any tuple of data $(c,\tilde{c},a,\ell,r)$ satisfying the conditions (\ref{eq:cocyc-1})-(\ref{eq:cocyc-0}) produces an extension $\hatC$ of $\mathbf{C}$. Hence, we take this to be our definition of cocycle:

\begin{defn}
    Let $\mathbf{C}$ be a small bicategory and $A$ an abelian group. A \emph{cocycle} is a tuple $(c,\tilde{c},a,\ell,r)$ of functions with values in $A$ and whose domains are, respectively, pairs of vertically composable 2-morphisms in $\mathbf{C}$, pairs of horizontally composable 2-morphisms in $\mathbf{C}$, triples of composable 1-morphisms in $\mathbf{C}$, 1-morphisms in $\mathbf{C}$ and 1-morphisms in $\mathbf{C}$, satisfying equations (\ref{eq:cocyc-1})-(\ref{eq:cocyc-0}).

    Let $Z(\mathbf{C},A)$ denote the set of cocycles. Then this is naturally an abelian group by componentwise addition.
\end{defn}

\begin{rk}
    This definition of a cocycle may seem unwieldy, but the conditions simplify greatly when $a=\ell=r=0$. For example, this happens when $\mathbf{C}$ and $\hatC$ are 2-categories, i.e. their associators and unitors are identity 2-morphisms. In such cases, the conditions reduce to requiring that $c,\tilde{c}$ satisfy associativity of vertical composition (\ref{eq:cocyc-1}), middle four exchange (\ref{eq:middle-four}) and a simplified version of assocativity of horizontal composition:
    \begin{equation}
        \tilde{c}(\beta,\alpha) + \tilde{c}(\gamma,\beta\star\alpha) = \tilde{c}(\gamma,\beta) + \tilde{c}(\gamma\star\beta,\alpha).
    \end{equation}
\end{rk}

We have now defined a cocycle, which captures the data of an extension. However, multiple cocycles may produce equivalent extensions. To get a group structure on $\Ext(\mathbf{C},A)$, we need to take the quotient of $Z(\mathbf{C},A)$ by the subgroup of coboundaries, described next.

\subsection{Equivalences of extensions as coboundaries}

We describe a notion of a coboundary, such that two cocycles produce equivalent extensions if and only if they differ by a coboundary.

Let $(F, F^2, F^0): \hatC \rightarrow \hatC'$ be an equivalence of extensions of $\mathbf{C}$ by $A$. As in the previous subsection, we may let the 2-morphisms of $\hatC$ and $\hatC'$ be of the form $(\alpha, g)$, where $\alpha$ is a 2-morphism of $\mathbf{C}$ and $g\in A$. Let the resultant cocycles be $(c,\tilde{c},a,\ell,r)$ and $(c',\tilde{c}',a',\ell',r')$ respectively.

For each 2-morphism $\alpha$ in $\mathbf{C}$, $F$ sends $(\alpha,0)$ in $\hatC$ to some $(\alpha, m(\alpha))$ in $\hatC'$. As identity 2-morphisms are sent to identity 2-morphisms, we have $m(\id_f) = 0$ for every 1-morphism $f$. Now, the functoriality of $F$ on each hom space of $\hatC$ is equivalent to
\begin{equation}\label{eq:cob-1}
    c'(\beta,\alpha) - c(\beta,\alpha) = m(\beta\alpha) - m(\alpha) - m(\beta).
\end{equation}

For each composable 1-morphisms $f,g$, $F^2_{g,f}$ must be of the form $(\id_{gf}, n(g,f))$, some $n(g,f) \in A$. Let $\alpha:f\rightarrow f', \beta:g\rightarrow g'$ be 2-morphisms in $\mathbf{C}$. Then the naturality of $F^2$ is equivalent to
\begin{equation}
    \tilde{c}'(\beta,\alpha) - \tilde{c}(\beta,\alpha) = m(\beta\alpha) - m(\alpha) - m(\beta) + n(g,f) - n(g',f').
\end{equation}

Let $f,g,h$ be composable 1-morphisms. The lax associativity condition on $F^2$ is equivalent to
\begin{equation}
    a'(h,g,f) - a(h,g,f) = n(h,g) + n(hg,f) - n(g,f) - n(h,gf) + m(\alpha_{h,g,f}).
\end{equation}

Finally, for each object $x$ in $\mathbf{C}$, $F^0_x$ must be of form $(\id_{\id_x}, k(x))$, some $k(x)\in A$. For a 1-morphism $f:x\rightarrow y$ in $\mathbf{C}$, the lax unity conditions are equivalent to 
\begin{equation}
    \ell'(f) - \ell(f) = k(y) + n(\id_y,f) - m(\lambda_f)
\end{equation}
and
\begin{equation}\label{eq:cob-0}
    r'(f) - r(f) = k(x) + n(f, \id_x) - m(\rho_f).
\end{equation}

Conversely, given two cocycles $(c,\tilde{c},a,\ell,r)$ and $(c',\tilde{c}',a',\ell',r')$, if there exist $m,n,k$ such that equations (\ref{eq:cob-1}) to (\ref{eq:cob-0}) are satisfied, then we may construct an equivalence $(F,F^2,F^0)$ between the two extensions that these induce.

Hence, we have the following definition of coboundary:

\begin{defn}\label{defn:coboundary}
    Let $\mathbf{C}$ be a small bicategory and $A$ an abelian group. Let $(m,n,k)$ be a tuple of functions with values in $A$ and whose domains are, respectively, 2-morphisms in $\mathbf{C}$, pairs of composable 1-morphisms in $\mathbf{C}$, objects in $\mathbf{C}$, satisfying $m(\id_f) = 0$ for every 1-morphism $f$. The \emph{coboundary} of $(m,n,k)$ is the tuple $(c,\tilde{c},a,\ell,r)$ whose components $c(\beta,\alpha), \tilde{c}(\beta,\alpha), a(h,g,f), \ell(f), r(f)$ are the right-hand sides of the equations (\ref{eq:cob-1}) to (\ref{eq:cob-0}).

    The coboundaries form a group under addition; denote it by $B(\mathbf{C},A)$.
\end{defn}

From our arguments above, two cocycles induce equivalent extensions if and only if they differ by a coboundary. Hence, we have shown:

\begin{prop} \label{prop:ext-gp}
    $\Ext(\mathbf{C},A)$ is in bijection with $Z(\mathbf{C},A)/B(\mathbf{C},A)$.
\end{prop}

This endows $\Ext(\mathbf{C},A)$ with the structure of an abelian group.

\section{The signature bordism bicategory} \label{section:sig}

\subsection{The oriented bordism bicategory}

The main object whose extensions we will be considering is the oriented bordism bicategory $\Bord{or}$. Roughly, this is a bicategory with oriented 1-manifolds as objects, bordisms between them as 1-morphisms, and bordisms between bordisms as 2-morphisms. This is described rigorously in \cite[\textsection 3]{csp-phd}, and we provide an overview below.

\begin{defn}
    Let $M, M'$ be closed oriented $(d-1)$-dimensional manifolds. An \emph{oriented bordism} from $M$ to $M'$ is an oriented $d$-manifold $X$ with an orientation-preserving diffeomorphism $\partial X \cong (-M) \sqcup M'$, where $-M$ is $M$ with the opposite orientation.

    Given such a bordism $X$, write $\parin X = M$ and $\parout X = M'$.
\end{defn}

For any $d$, the closed oriented $d$-dimensional manifolds and the bordisms between them form a symmetric monoidal category:

\begin{defn}
    The \emph{oriented bordism category} $\mathbf{Bord}_{d-1,d}^{\mathrm{or}}$ is the category with the following objects and morphisms:
    \begin{itemize}
        \item Objects are closed oriented $(d-1)$-dimensional manifolds.
        \item Morphisms from $M$ to $M'$ are oriented bordisms from $M$ to $M'$, taken up to diffeomorphism relative to the boundary.
    \end{itemize}
    Given bordisms $X, X'$ with $\parout X = \parin X' = M$, their composition is defined to be $X \cup_M X'$.

    The monoidal structure is given by taking disjoint unions.
\end{defn}

For composition to be well-defined, we need to equip the boundaries of bordisms with collars, so that we may get a smooth structure on the glued bordism $X\cup_M X'$. Different choices of collars give bordisms which are diffeomorphic with respect to the boundary, and so this is indeed a well-defined category.

To extend this to a bicategory, we need to add $2$-morphisms, which requires a notion of bordisms between bordisms. These are \emph{manifolds with corners}, which are topological manifolds with boundary equipped with a maximal smooth atlas of charts to $[0,\infty)^d \subseteq \mathbb{R}^d$. For a manifold with corners $M$, the \emph{index} of a point $x\in M$ is the number of coordinates in $\varphi(x)$ which are $0$, where $\varphi$ is a chart at $x$. In the case of bordisms between bordisms, points will only have index $0$ (corresponding to the interior), $1$ (corresponding to the boundary minus the corners) or $2$ (corresponding to the corners). We write $\partial M$ for the union of the index $1$ and $2$ points.

\begin{defn}\label{defn:bordism-corners}
    Let $X, X'$ be oriented bordisms between closed oriented $(d-2)$-dimensional manifolds $M$ and $M'$. An \emph{oriented bordism with corners} from $X$ to $X'$ is an oriented $d$-dimensional manifold with corners $Y$ with orientation-preserving diffeomorphism $$\partial Y \cong \pa{(-X) \sqcup X'} \cup_{(-M) \sqcup M \sqcup (-M') \sqcup M'} \pa{I \times (-M) \sqcup I \times M'}.$$

    Here, the two copies of $M$ and the two copies of $M'$ are the corners (index $2$) of $Y$, while the rest of $\partial Y$ has index $1$.

    This has a \emph{horizontal boundary} which consists of the manifolds with boundary $\parin Y = X$ and $\parout Y = X'$, as well as a \emph{vertical boundary} which consists of the manifolds with boundary $\parIn Y = I \times M$ and $\parOut Y = I \times M'$.
\end{defn}

By considering bordisms with corners, we may extend $\mathbf{Bord}_{d-1,d}^{\mathrm{or}}$ to a bicategory:

\begin{defn}\label{defn:bordor}
    The \emph{oriented bordism bicategory} $\mathbf{Bord}_{d-2,d-1,d}^{\mathrm{or}}$ is the category with the following objects and morphisms:
    \begin{itemize}
        \item Objects are closed oriented $(d-2)$-dimensional manifolds.
        \item 1-morphisms from $M$ to $M'$ are oriented bordisms from $M$ to $M'$.
        \item 2-morphisms from $X$ to $X'$ with $\parin X = \parin X'$ and $\parout X = \parout X'$ are bordisms with corners from $X$ to $X'$, taken up to diffeomorphism relative to the boundary.
    \end{itemize}
    Composition can be done either vertically or horizontally, by gluing corresponding components of the horizontal or vertical boundaries respectively.

    As in the ordinary category case, the monoidal structure is given by taking disjoint unions.
\end{defn}

As with the case of the bordism category, we need to equip our bordisms with collars to have a smooth structure upon gluing. However, there is now some subtlety that arises from the fact that 1-morphisms are not diffeomorphism classes of bordisms. When gluing two 1-morphisms, different choices of collars result in different (though diffeomorphic) smooth structures, which are now different (though isomorphic) 1-morphisms. This is resolved in \cite{csp-phd} by showing that any two such choices result in a bordisms which admit a diffeomorphism that is canonical up to isotopy, and so despite some arbitrary choices being made, $\mathbf{Bord}_{d-2,d-1,d}^\mathrm{or}$ as defined is indeed a bicategory.

Henceforth, we shall focus on the case $d=3$, i.e. the bicategory $\Bord{or}$. The arguments made will be topological in nature, and so we will not concern ourselves with the subtleties which arise from choosing collars.

We wish to study the symmetric monoidal extensions of $\Bord{or}$ by $\mathbb{Z}$. First, we describe one such extension, the signature bordism bicategory.

\subsection{The signature bordism bicategory}

The signature bordism bicategory $\Bord{sig}$, as first described in \cite{BDSV3}, is not a symmetric monoidal extension of the model of $\Bord{or}$ presented in Definition \ref{defn:bordor}. Rather, it is a symmetric monoidal extension of $\Bord{or,exp}$, which is equivalent as a symmetric monoidal bicategory to $\Bord{or}$. One may then obtain a symmetric monoidal extension of $\Bord{or}$ by pulling $\Bord{sig}$ back along an equivalence between $\Bord{or}$ and $\Bord{or,exp}$. We shall present a geometric description of $\Bord{or,exp}$ and $\Bord{sig}$ as well as a characterisation in terms of cocycles.

\begin{defn}\label{defn:bordexp}
    The \emph{oriented bordism bicategory with expanded manifolds} $\Bord{or,exp}$ is the bicategory with:
    \begin{itemize}
        \item Objects: Pairs $(S, D_S)$ where $S$ is a closed oriented 1-dimensional manifold and $D_S$ a 2-manifold with boundary $\partial D_S = S$. $D_S$ is oriented such that the orientation induced on $\partial D_S$ is the same as the orientation on $S$.
        \item 1-morphisms: Pairs $(\Sigma, H_\Sigma)$ where $\Sigma$ is a oriented 2-dimensional bordism, and $H_\Sigma$ is a 3-manifold with boundary $$\partial H_{\Sigma} = \ol{\Sigma}:= D_{\parin \Sigma} \cup_{\parin \Sigma} \Sigma \cup_{\parout \Sigma} \pa{-D_{\parout \Sigma}}.$$ $H_\Sigma$ is oriented such that the orientation induced on $\partial H_\Sigma$ is the same as the orientation on $\ol{\Sigma}$.
        \item 2-morphisms: oriented bordisms with corners, taken up to diffeomorphism relative to the boundary.
    \end{itemize}
    The monoidal structure is given by taking disjoint unions.
\end{defn}
From the definition, it is clear that the symmetric monoidal functor $\Bord{or,exp}\rightarrow\Bord{or}$ given by forgetting the extra structure on the objects and 1-morphisms is an equivalence of symmetric monoidal bicategories.

\begin{defn}\label{defn:bordsig}
    The \emph{signature bordism bicategory} $\Bord{sig}$ is the symmetric monoidal bicategory with:
    \begin{itemize}
        \item Objects and 1-morphisms: as in $\Bord{or,exp}$ in Definition \ref{defn:bordexp}.
        \item 2-morphisms: Equivalence classes of pairs $(M, X)$ where $M$ is a 3-dimensional oriented bordism with corners, and $X$ is a 4-manifold whose boundary $\partial X = \ol{M}$ is obtained from $M$ by gluing $H_{\parin M}$ and $-H_{\parout M}$ onto its horizontal boundaries and $V_{\parIn M}:= I\times D_{\parin\parin M}$ and $-V_{\parOut M}:= -(I\times D_{\parout\parin M})$ onto its vertical boundaries.
        \begin{figure}[hbt!]
            \includegraphics[scale=0.5]{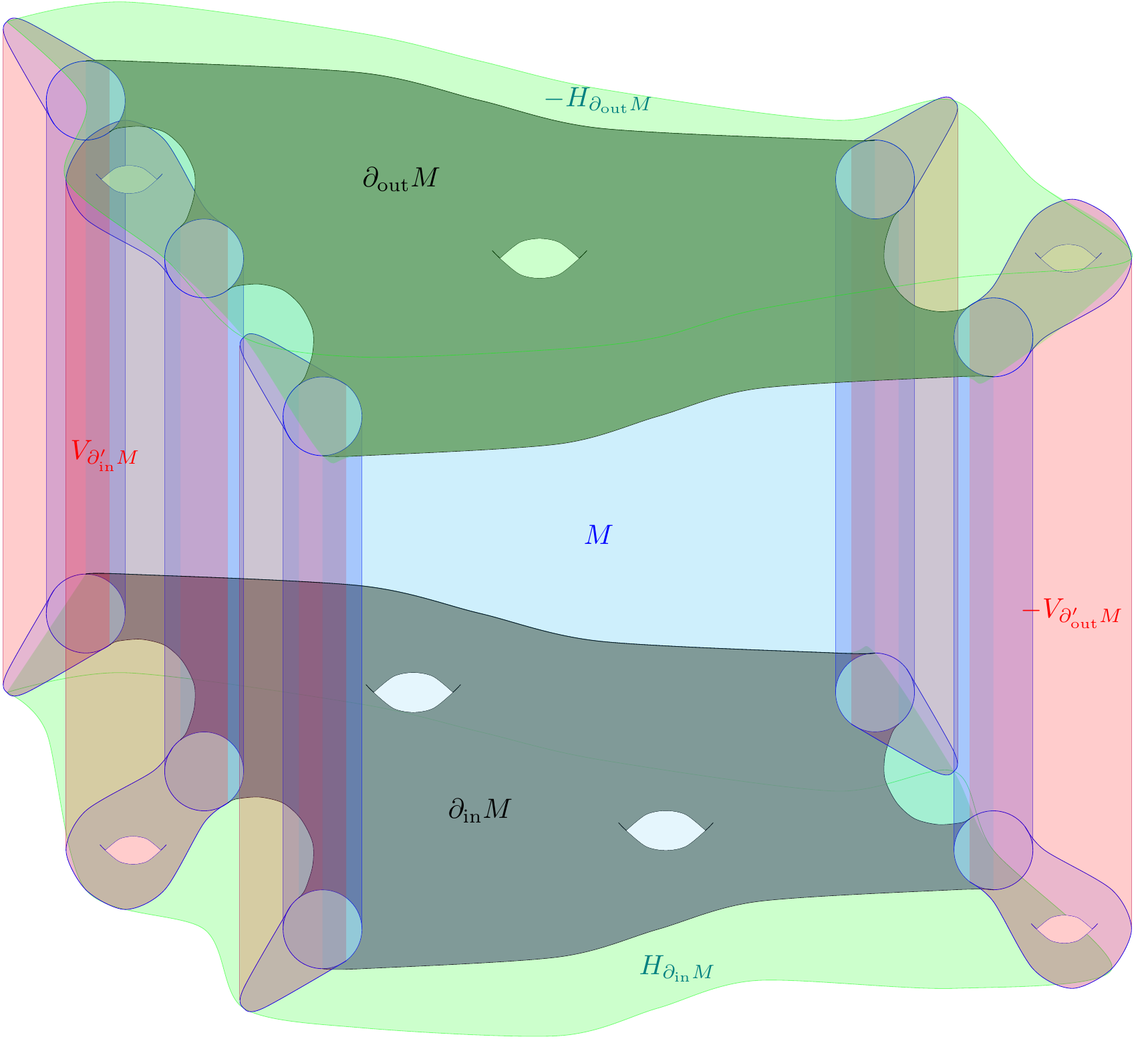}
            \caption{Schematic for $\ol{M}$.} \label{fig:bordsig-2mor}
        \end{figure}
        Figure \ref{fig:bordsig-2mor} shows a schematic for $\ol{M}$.

        The pairs $(M,X)$ and $(M',X')$ are equivalent if: $M$ and $M'$ are diffeomorphic relative to the boundary, which induces a diffeomorphism $\ol{M} \cong \ol{M'}$; and the closed 4-manifold $(-X) \cup_{\ol{M}} X'$ is cobordant to $\emptyset$.
    \end{itemize}
    The composition of 1-morphisms $(\Sigma, H_{\Sigma})$ and $(\Sigma', H_{\Sigma'})$ with  $\parout \Sigma = \parin \Sigma' = S$ is given by $$(\Sigma \cup_S \Sigma', H_{\Sigma} \cup_{D_S} H_{\Sigma'}).$$
    Similarly, the vertical composition of two 2-morphisms $(M, X)$ and $(M', X)$ with, respectively, target and source $(\Sigma, H_\Sigma)$ is given by $$\pa{M \cup_\Sigma M', X\cup_{H_\Sigma}X'}$$
    and similarly, the horizontal composition of two 2-morphisms $(M, X)$ and $(M', X')$ with $\parOut M \cong \parIn M' \cong \Sigma$ is $$\pa{M \cup_\Sigma M', X\cup_{V_\Sigma}X'}$$

    The monoidal structure is given by taking disjoint unions.
\end{defn}

Based on this definition, it is not immediately obvious how $\Bord{sig}$ might be an extension of $\Bord{or,exp}$ by $\mathbb{Z}$, or even that composition respects the equivalence relations on the $(M,X)$. This can be seen via the relation between cobordant 4-manifolds and the signature.

\begin{defn}
    Let $X$ be an oriented $4n$-manifold with boundary. Consider the map on cohomology $q^*: H^{2n}(X,\partial X; \mathbb{R}) \rightarrow H^{2n}(X; \mathbb{R})$. The cup product induces a symmetric, non-degenerate bilinear form on $\mathrm{Im}(q^*)$. The \emph{signature} of $X$, written $\sigma(X)$, is the signature of this bilinear form.
\end{defn}

For 4-manifolds, this exactly classifies when 4-manifolds are cobordant. Specifically, two closed 4-manifolds are cobordant if and only if they have the same signature \cite{Rohlin}. Moreover, by Novikov's additivity of the signature \cite[pp. 587--589]{AS-Novikov}, the signature of a closed 4-manifold which arises from gluing two 4-manifolds along their common boundary is given by the sum of the signatures of the two 4-manifolds. Thus, we have$$\sigma((-X) \cup_{\ol{M}} X') = \sigma(-X) + \sigma(X') = -\sigma(X) + \sigma(X').$$ Hence, returning to the equivalence relation on 2-morphisms, $(M,X)$ and $(M',X')$ are equivalent if and only if $M, M'$ are diffeomorphic relative to the boundary and $\sigma(X) = \sigma(X')$.

Let $\mathbb{Z}$ act on the 2-morphisms $\Bord{sig}$ as follows:
\begin{equation*}
    n \cdot (M, X) = 
    \begin{cases}
        \pa{M, X \sqcup \coprod^n \mathbb{CP}^2}, &n \ge 0 \\
        \pa{M, X \sqcup \coprod^{n} (-\mathbb{CP}^2)}, &n < 0 
    \end{cases}
\end{equation*}
This is compatible with composition of 2-morphisms in both directions. Thus, the quotient by the $\mathbb{Z}$-action $\Bord{sig}\rightarrow\Bord{or,exp}$ defines an extension by $\mathbb{Z}$. Since the signature is additive along disjoint unions, this is in fact a symmetric monoidal extension.

Henceforth, we regard the 2-morphisms from $\Sigma$ to $\Sigma'$ in $\Bord{sig}$ as pairs $(M, n)$, where $M$ is a bordism with corners from $\Sigma$ to $\Sigma'$, taken up to diffeomorphism relative to the boundary, and $n \in \mathbb{Z}$. Here, $n$ represents the signature of the 4-manifold $X$ with boundary $\ol{M}$.

We now move from this geometric definition of $\Bord{sig}$ to an algebraic definition given by cocycles. 

\subsection{$\Bord{sig}$ in terms of cocycles}

Considering $\Bord{sig}$ as an extension of $\Bord{or,exp}$ (so ignoring the symmetric monoidal structure), its bicategorical structure may be captured with the data of a cocycle $$\mathbf{c}\ssig = (c\ssig,\tilde{c}\ssig,a\ssig,\ell\ssig,r\ssig)$$ as described in \textsection\ref{section:ext-bicat}. We now compute this cocycle, which will provide an algebraic description for the signature extension. In \textsection\ref{section:sig/2} and \textsection\ref{section:indiv}, we will then use this to study the subextensions of $\Bord{sig}$.

As explained in \textsection\ref{section:ext-bicat}, to compute the cocycle for an extension, we have to first pick a choice of lift for each 2-morphism. There is an obvious choice for this: for each 2-morphism $M$, choose its lift to be $(M,0)$. (It's not immediately clear a priori that identity 2-morphisms lift to identity 2-morphisms; this will be checked later in Lemma \ref{lem:c-norm}.)

To compute $c\ssig$ and $\tilde{c}\ssig$, we need to determine how the signature term $n$ is affected by composition of 2-morphisms. Composition of 2-morphisms in $\Bord{sig}$ involves gluing two $X,X'$ 4-manifolds along some portion of their boundaries $H$. We wish to determine how the signature of this new 4-manifold $X \cup_H X'$ relates to the signatures of the original two 4-manifolds $X, X'$. This is an application of Wall's non-additivity formula, which we outline below.

Wall \cite{Wall} showed that $\sigma(X \cup_H X')$ may be written in terms of $\sigma(X)$, $\sigma(X')$ and an invariant based on the vector space $V := H_1(\partial H; \mathbb{R})$, which has a symplectic bilinear form given by the intersection form, along with three specific Lagrangian subspaces of $V$.

In general, let $(V, \omega)$ be a finite-dimensional symplectic vector space and $A, B, C$ three Lagrangian subspaces. $\omega$ induces a bilinear form $\langle\cdot,\cdot\rangle$ on $W = \frac{A \cap (B+C)}{(A \cap B) + (A \cap C)}$ as follows:
for $a, a' \in A \cap (B+C)$, let $a + b + c = a' + b' + c' = 0$, where $b,b' \in B$ and $c, c' \in C$. Then define $\langle a,a'\rangle := \omega(a, b')$.

Wall showed that $\langle\cdot,\cdot\rangle$ is a well-defined non-degenerate symmetric bilinear form, and hence it makes sense to define:

\begin{defn}
    Let $(V, \omega)$ be a finite-dimensional symplectic vector space and $A, B, C$ three Lagrangian subspaces. \emph{Wall's invariant} $\sigma(V; A, B, C)$ is the signature of $\langle\cdot,\cdot\rangle$.
\end{defn}
\begin{rk}
    There is another invariant of Lagrangian subspaces of a symplectic vector space, called the Maslov index, which first appears in the literature in \cite[\textsection 1.5]{LV}. As Wall's invariant coincides with this \cite[\textsection 12]{CLM}, it is often referred to as the Maslov index.
\end{rk}
\begin{rk}
    Interchanging the roles of any two of $A,B,C$ flips the sign of $\langle\cdot,\cdot\rangle$. Indeed, $\omega(a,b') = \omega(a,a'+b') = -\omega(a,c')$ and so interchanging the roles of $A,C$ flips the sign. A similar argument can be made for the other two swaps.
\end{rk}

We now note a simple condition for $\sigma(V;A,B,C)$ to be zero:
\begin{lem} \label{lem:coinc-zero}
    If any two of $A,B,C$ coincide, then $\sigma(V; A, B, C) = 0$.
\end{lem}
\begin{proof}
    If $B=C$, then $A\cap(B+C) = A \cap B = (A\cap B)+(A\cap C)$, so $W = \frac{A\cap(B+C)}{(A\cap B)+(A\cap C)}$ is the zero vector space, and thus has signature $0$.

    If $A=B$, then $A \cap (B+C) = A$ while $(A\cap B) + (A\cap C) = A + (A\cap C) = A$, so again $W$ is the zero vector space, and so has signature $0$.
\end{proof}

Wall's invariant measures the failure of the signature to be additive when two $4k$-manifolds are glued along subspaces of their boundary.

\begin{thm}[\cite{Wall}]\label{thm:wall}
    Let $M_-,H,M_+$ be $(4k-1)$-manifolds with common boundary $-\partial M_- = \partial H = \partial M_+ = \Sigma$ and $X_-, X_+$ be $4k$-manifolds with $\partial X_+ = M_+ \cup_\Sigma (-H), \partial X_- = X_0 \cup_\Sigma M_-$, and $X_0 = X_+ \cup_{H} X_-$, as in Figure \ref{fig:wall-diagram}. Then the signatures of $X_0$ and $X_\pm$ are related by $$\sigma(X_0) = \sigma(X_-) + \sigma(X_+) - \sigma(V; A, B, C)$$
    where $V = H_{2k-1}(\Sigma; \mathbb{R})$ and $A, B, C$ are the Lagrangian subspaces of $V$ associated to the inclusions $\Sigma \hookrightarrow M_-$, $\Sigma \hookrightarrow H$, $\Sigma \hookrightarrow M_+$ respectively.
    \begin{figure}[hbt!]
        \includegraphics[scale=0.5]{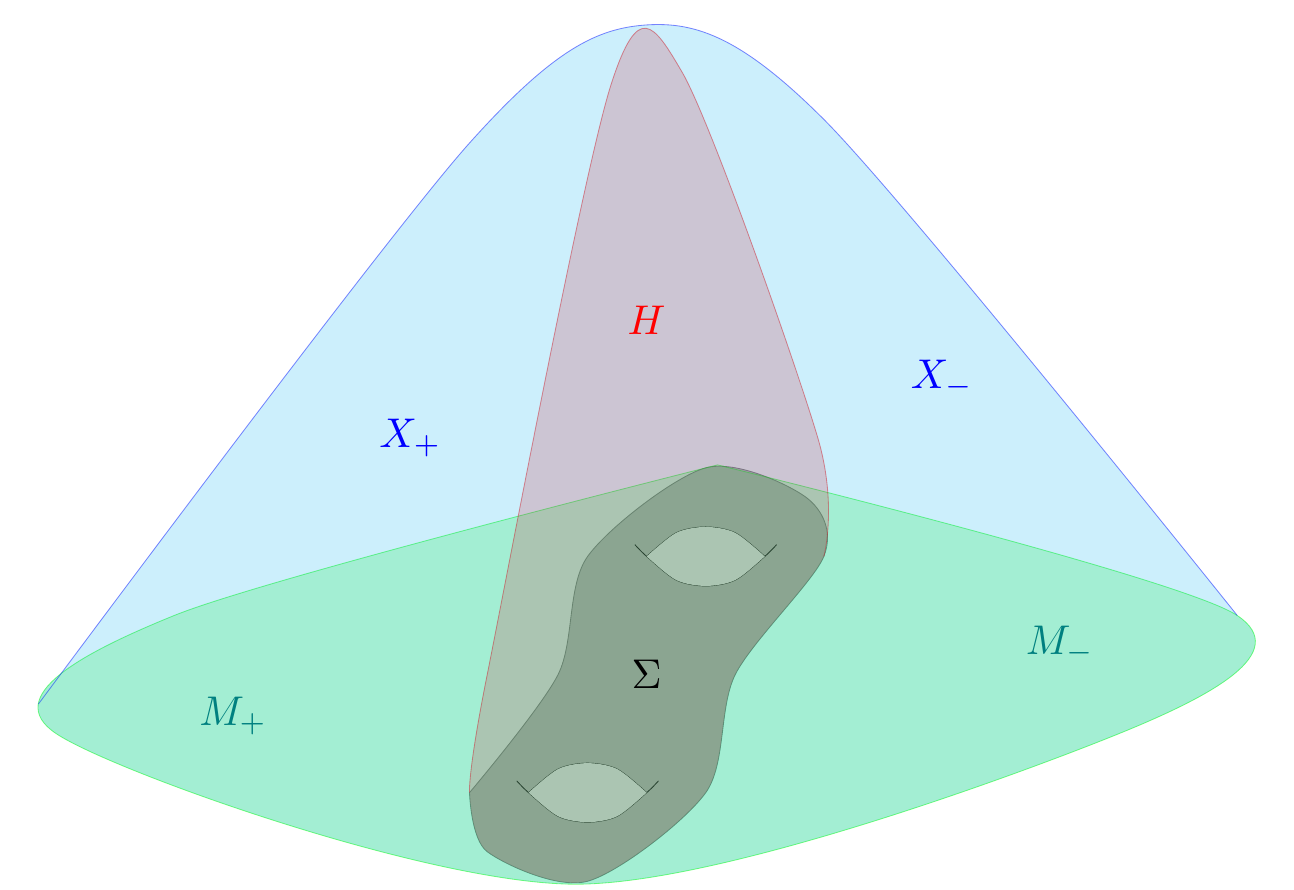}
        \caption{Gluing $X_\pm$ to form $X_0$ in Wall's Theorem.} \label{fig:wall-diagram}
    \end{figure}
\end{thm}

This allows us to compute the terms $c\ssig$ and $\tilde{c}\ssig$.

\subsubsection*{Vertical composition}

We use the notation of Definition \ref{defn:bordsig}. We want to determine the signature of $X \cup_{H_\Sigma}X'$. This is an application of Wall's theorem where $X_+ = X, X_- = X', M_+ = \ol{M} \setminus \mathring{H}_{\Sigma'}, X_0 = H_{\Sigma'}, X = \ol{M'} \setminus \mathring{H}_{\Sigma'}$. The vertical composition of $(M, n)$ and $(M', n')$ is thus $$\pa{M \cup_\Sigma M', n + n' - \sigma(V; A, B, C)}$$ where
\begin{equation} \label{eq:csig}
    \begin{split}
        V &= H_1(\ol{\Sigma}; \mathbb{R}) \\
        A &= \ker(H_1(\ol{\Sigma}; \mathbb{R}) \rightarrow H_1(\ol{M'}\setminus \mathring{H}_\Sigma; \mathbb{R})) \\
        B &= \ker(H_1(\ol{\Sigma}; \mathbb{R}) \rightarrow H_1(H_\Sigma; \mathbb{R})) \\
        C &= \ker(H_1(\ol{\Sigma}; \mathbb{R}) \rightarrow H_1(\ol{M}\setminus \mathring{H}_\Sigma; \mathbb{R}))
    \end{split}
\end{equation}
Hence, we have
\begin{equation}\label{eq:csig-def}
    c\ssig(M',M) = -\sigma(V;A,B,C).
\end{equation}

\subsubsection*{Horizontal composition}

We again use the notation of Definition \ref{defn:bordsig}. This time, we instead apply Wall's theorem to $X\cup_{V_\Sigma}X'$, where $\Sigma = \parOut M = \parIn M'$, and the computation is similar to the previous case. Let $\ol{\Sigma} = D_{\parout\parin M} \cup_{\parout\parin M} \Sigma \cup_{\parout\parout M} (-D_{\parout\parout M})$, and let
\begin{equation} \label{eq:ctildesig}
    \begin{split}
        V &= H_1(\ol{\Sigma}; \mathbb{R}) \\
        A &= \ker(H_1(\ol{\Sigma}; \mathbb{R}) \rightarrow H_1(\ol{M'}\setminus \mathring{V}_\Sigma; \mathbb{R})) \\
        B &= \ker(H_1(\ol{\Sigma}; \mathbb{R}) \rightarrow H_1(V_\Sigma; \mathbb{R})) \\
        C &= \ker(H_1(\ol{\Sigma}; \mathbb{R}) \rightarrow H_1(\ol{M}\setminus \mathring{V}_\Sigma; \mathbb{R}))
    \end{split}
\end{equation}
Then we have
\begin{equation}\label{eq:ctildesig-def}
    \tilde{c}\ssig(M',M) =  -\sigma(V;A,B,C).
\end{equation}

Earlier, we made the assumption that our chosen section lifts identity 2-morphisms to identity 2-morphisms. We may now verify this by proving a slightly stronger lemma.

\begin{lem}\label{lem:c-norm}
    Let $(\Sigma, H_\Sigma)$ and $(\Sigma',H_{\Sigma'})$ be 1-morphisms in $\Bord{or,exp}$. Let $M$ be a mapping cylinder of an orientation-preserving diffeomorphism from $\Sigma$ to $\Sigma'$ which sends $\ker(H_1(\ol{\Sigma};\mathbb{R})\rightarrow H_1(H_\Sigma;\mathbb{R}))$ to $\ker(H_1(\ol{\Sigma'};\mathbb{R})\rightarrow H_1(H_{\Sigma'};\mathbb{R}))$. Then:
    \begin{enumerate}
        \item For any 1-morphism $M'$ in $\Bord{or,exp}$ with target $(\Sigma, H_\Sigma)$, we have $$c\ssig(M,M') = 0.$$
        \item For any 1-morphism $M''$ in $\Bord{or,exp}$ with source $(\Sigma', H_{\Sigma'})$, we have $$c\ssig(M'',M) = 0.$$
    \end{enumerate}    
    In particular, if $M = I\times\Sigma$ is the mapping cylinder of the identity diffeomorphism $\id:\Sigma\rightarrow\Sigma$, then $(I\times\Sigma, 0)$ is the identity 2-morphism of $(\Sigma, H_\Sigma)$ in $\Bord{sig}$.
\end{lem}
\begin{proof}
    In the first case, the Lagrangian subspaces $A$ and $B$ in (\ref{eq:csig}) are the same, and so we may apply Lemma \ref{lem:coinc-zero}. Likewise, in the latter case, $B$ and $C$ are the same Lagrangian subspace.

    Finally, applying (1) and (2) to the case where $M$ is the mapping cylinder of the identity diffeomorphism, we have verified that $(\Sigma \times I, 0)$ is the identity 2-morphism of $(\Sigma, H_\Sigma)$ in $\Bord{sig}$.
\end{proof}

Now, notice that the associators and unitors in $\Bord{or,exp}$ are mapping cylinders of this form, and the associators and unitors in $\Bord{sig}$ do not change the signature of the $4$-manifold associated to the 2-morphisms. Hence, we have $$a\ssig=\ell\ssig=r\ssig=0.$$

In conclusion, the cocycle $\mathbf{c}\ssig=(c\ssig,\tilde{c}\ssig,a\ssig,\ell\ssig,r\ssig)$ we have computed has $c\ssig$ as in (\ref{eq:csig-def}), $\tilde{c}\ssig$ as in (\ref{eq:ctildesig-def}) and all other terms $0$.

\begin{rk}
    There is another choice of model for $\Bord{or}$ which results in a slightly nicer expression for $\tilde{c}\ssig$. Consider the full subbicategory of $\Bord{or,exp}$ whose objects are $(S,D_S)$ where $D_S$ is a disjoint union of discs, and whose 1-morphisms are $(\Sigma,H_\Sigma)$ where $H_\Sigma$ is a disjoint union of handlebodies. The inclusion of this into $\Bord{or,exp}$ is an equivalence of symmetric monoidal bicategories.

    Upon pulling back $\mathbf{c}\ssig$ to this subbicategory, $\tilde{c}\ssig$ is now $0$, since all the $\parIn M$ and $\parOut M$ are disjoint unions of $S^2$, and so have trivial $H_1$.
\end{rk}

\subsection{$\bord{sig}$ and Walker's category of extended surfaces and $3$-manifolds} In \cite[\textsection 1]{Walker}, Walker defined a category of extended surfaces and $3$-manifolds which is a symmetric monoidal extension of a symmetric monoidal category $\bord{or,lag}$ equivalent to $\bord{or}$. This is closely related to the restriction of $\Bord{sig}$ to closed surfaces and the bordisms between them. The aim of this subsection is to clarify the relation between these categories.

\begin{defn}
    The \emph{oriented bordism category with expanded manifolds} $\bord{or,exp}$ has
    \begin{itemize}
        \item Objects: Pairs $(\Sigma, H_{\Sigma})$ where $\Sigma$ is a closed oriented surface and $H_\Sigma$ is a 3-manifold with boundary with $\partial H_\Sigma = \Sigma$. This is oriented such that the orientation induced on $\partial H_\Sigma$ is the same as the orientation on $\Sigma$.
        \item Morphisms: Bordisms $M$ between closed surfaces, taken up to diffeomorphism relative to the boundary.
    \end{itemize}
\end{defn}

In other words, considering the object $(\emptyset,\emptyset)$ in $\Bord{or,exp}$ given by the empty 1-manifold bounding the empty surface, $\bord{or,exp}$ is the hom category $$\bord{or,exp} = \Hom_{\Bord{or,exp}}((\emptyset,\emptyset),(\emptyset,\emptyset)).$$

Now, for any (symmetric monoidal) extension $\Bord{ext}\rightarrow\Bord{or,exp}$, we may restrict it to $\bord{or,exp}$: let $\bord{ext}$ be the hom category $$\bord{ext} := \Hom_{\bord{ext}}((\emptyset,\emptyset),(\emptyset,\emptyset)),$$ and then $\bord{ext}\rightarrow\bord{or,exp}$ defines a (symmetric monoidal) extension. As in the bicategorical case, the forgetful functor $\bord{or,exp}\rightarrow\bord{or}$ is an equivalence of symmetric monoidal categories, and so this is equivalent to defining an (symmetric monoidal) extension of $\bord{or}$.

Ignoring the symmetric monoidal structure, we may describe this at the level of cocycles, in terms of a homomorphism between $\Ext(\Bord{or,exp},\mathbb{Z})$ and $\Ext(\bord{or,exp},\mathbb{Z})$.
\begin{lem}\label{lem:res-cocyc}
    For any abelian group $A$, the map $\Ext(\Bord{or,exp},A)\rightarrow\Ext(\bord{or,exp},A)$ sending $[\mathbf{c}] = [(c,\tilde{c},a,l,r)]$ to $[c]$ defines a group homomorphism.
\end{lem}
\begin{proof}
    The cocycle condition (\ref{eq:cocyc-cat}) for $c\ssig$ is satisfied, by the condition (\ref{eq:cocyc-1}) on $\mathbf{c}\ssig$. Coboundaries for $\Bord{or,exp}$ are sent to coboundaries for $\bord{or,exp}$; refer to the computation of a coboundary in (\ref{eq:cob-1}). Thus, the map is well-defined. Moreover, this map respects the group operation on the two $\Ext$ groups, and so we have a well-defined homomorphism $\Ext(\Bord{or,exp})\rightarrow\Ext(\bord{or,exp})$.
\end{proof}
Under this homomorphism, $[\mathbf{c}\ssig]$ corresponding to the signature bordism bicategory is sent to $[c\ssig]$, where $c\ssig$ is the cocycle given by (\ref{eq:csig-def}). This defines an extension $\bord{sig}$ of $\bord{or,exp}$ which is in fact a symmetric monoidal extension.

In fact, this cocycle may be rewritten purely in terms of the Lagrangian subspaces $\ker(H_1(\ol{\Sigma};\mathbb{R})) \rightarrow H_1(H_{\Sigma};\mathbb{R})$ for each $\Sigma$. Indeed, $B$ is exactly such a subspace while $A,C$ may be obtained from the Lagrangian subspaces corresponding to $\Sigma'',\Sigma$ via some Mayer-Vietoris computations. This then motivates one to consider another model of $\bord{sig}$ whose objects are equipped with a choice of Lagrangian subspace.
\begin{defn}
    The \emph{oriented bordism category with Lagrangian subspaces} $\bord{or,lag}$ has
    \begin{itemize}
        \item Objects: Pairs $(\Sigma, L)$ where $\Sigma$ is a closed oriented surface and $L$ is a Lagrangian subspace of $H_1(\Sigma;\mathbb{R})$.
        \item Morphisms: Bordisms $M$ between closed surfaces, taken up to diffeomorphism relative to the boundary.
    \end{itemize}
\end{defn}
\begin{rk}
    The pair $(\Sigma,L)$ is called an \emph{extended surface} in \cite{Walker}.
\end{rk}

Then, the forgetful functors $\bord{or,exp}\rightarrow\bord{or,lag}\rightarrow\bord{or}$ are equivalences of symmetric monoidal categories. These define isomorphisms between their groups of extensions, as the cohomology of nerves of equivalent categories are isomorphic. The corresponding extension $\bord{sig,lag}$ of $\bord{or,lag}$ is exactly the category of extended surfaces and $3$-manifolds as described in \cite{Walker}.

Analogously, there is another model $\Bord{or,lag}$ of $\Bord{or}$ whose 1-morphisms are equipped with Lagrangian subspaces, which correspondingly has an extension $\Bord{sig,lag}$. This $\Bord{sig,lag}$ is another model of $\Bord{sig}$ which is used, for example, in \cite{dr-thesis}.

\section{Index $2$ subextensions of $\Bord{sig}$} \label{section:sig/2}

In the previous section, we described a symmetric monoidal extension $\Bord{sig}$ of $\Bord{or,exp}$. As a bicategory, its data is captured by an element $[\mathbf{c}\ssig]$ of the abelian group $\Ext(\Bord{or,exp},\mathbb{Z})$. In this section, we show that $[\mathbf{c}\ssig]$ may be divided by $2$, giving rise to an extension $\Bord{sig/2}$ of $\Bord{or,exp}$. This may be viewed (non-uniquely) as a subbicategory of $\Bord{sig}$, forming an index $2$ subextension. $\Bord{sig/2}$ then inherits a symmetric monoidal structure from $\Bord{sig}$, making it a symmetric monoidal extension of $\Bord{or,exp}$.

In the $(\infty,2)$-categorical setting, the existence of such a symmetric monoidal subextension is implied by the results of \cite[\textsection 7]{csp-inv-tft}, which is in turn based on computations in \cite{BDSV3}. Symmetric monoidal extensions of the bordism $(\infty,2)$-category $\mathscr{B}\mathit{ord}_{1,2,3}^\mathrm{or}$ by $\mathbb{Z}$ correspond (via a straightening-unstraightening argument sketched in \cite[\textsection 7.5]{csp-inv-tft}) to certain cohomology groups of Madsen--Tillmann spectra. The relevant cohomology group for $\mathscr{B}\mathit{ord}_{1,2,3}^\mathrm{or}$ is $\mathbb{Z}$, with the signature extension corresponding to twice a generator. A preliminary version of this argument is also in \cite{BDSV3}, where it is additionally conjectured that symmetric monoidal extensions of $\mathscr{B}\mathit{ord}_{1,2,3}^\mathrm{or}$ by $\mathbb{Z}$ correspond to symmetric monoidal extensions of $\Bord{or}$.

More recently, half-signature structures were defined as twisted tangential structures in the initial draft of \cite[\textsection 8]{FST}. These are defined in terms of the twisted Anderson dual cohomologies of the classifying space $BSO(3)$. We understand that future work of the same authors will address this in more detail.

Here, our definition of $\Bord{sig/2}$ is written purely in terms of the Betti numbers of the bordisms, providing an explicit description of the half-signature bordism bicategory. This description will in particular be useful in proving later results such as Proposition  \ref{prop:fg-sig/2} and \cite[Proposition 4.18]{partB}.

\subsection{Subextensions of $\bord{sig}$ and $\Bord{sig}$} \label{subsection:sig/2-overview}

As described in the previous section, we may define an extension $\bord{sig}$ of $\bord{or,exp}$ by taking the image of $[\mathbf{c}\ssig]$ under the group homomorphism defined in Lemma \ref{lem:res-cocyc}. This element $[c\ssig]\in\Ext(\bord{or,exp},\mathbb{Z}) \cong \Ext(\bord{or,lag},\mathbb{Z})$ is well-studied.

Gilmer \cite[\textsection 7]{Gilmer}, while working with $\bord{or,lag}$, constructed an index $2$ subextension of $\bord{sig}$ by defining certain morphisms in $\bord{sig}$ to be even morphisms, and showing that the composition of even morphisms is even. For each morphism in $\bord{or,lag}$, the corresponding $\mathbb{Z}$-family of morphisms in $\bord{sig}$ has morphisms which alternate between even and odd, and so the category given by only the even morphisms is an index $2$ subextension of $\bord{sig}$.

At the level of cocycles, this argument is equivalent to showing that $[c\ssig]$ is divisible by $2$ in $\Ext(\bord{or,exp},\mathbb{Z})$. Indeed, let $m(M)$ be $0$ if $(M,0)$ is even and $1$ otherwise. Then, the composition of even morphisms being even is equivalent to the coboundary of $m$ being equal to $c\ssig$ modulo $2$. Hence, subtracting this coboundary from $c\ssig$ and dividing by $2$ yields a cocycle $c\ssigh$ which satisfies $2[c\ssigh] = [c\ssig] \in \Ext(\bord{or,exp},\mathbb{Z})$. By construction, this cocycle $c\ssigh$ exactly determines the index $2$ subextension of $\bord{sig}$ consisting only of the even morphisms.

We will generalise this to the bicategory case by constructing a cocycle $\mathbf{c}\ssigh$ such that $2[\mathbf{c}\ssigh] = [\mathbf{c}\ssig] \in \Ext(\Bord{or,exp},\mathbb{Z})$, producing an extension $\Bord{sig/2}$. This new extension will also turn out to be a symmetric monoidal and can be viewed as an index 2 subextension of $\Bord{sig}$.

\subsection{Constructing $\Bord{sig/2}$}

We will provide an explicit expression for $\mathbf{c}\ssigh$, a cocycle which satisfies $2[\mathbf{c}\ssigh] = [\mathbf{c}\ssig] \in \Ext(\Bord{or,exp},\mathbb{Z})$. This will be done by constructing a function $m$ defined on $\mathbb{M}$, the set of 2-morphisms of $\Bord{or,exp}$, with values in $\mathbb{Z}$, such that $m$ sends identity 2-morphisms to $0$ and subtracting the coboundary of $(m,0,0)$ from $\mathbf{c}\ssig$ results in a cocycle with values in $2\mathbb{Z}$. We can then divide it by $2$ to obtain $\mathbf{c}\ssigh$.

To find a valid $m$, it suffices to work modulo $2$. Hence, we seek $m:\mathbb{M}\rightarrow\set{0,1}$ sending identity 2-morphisms to $0$ such that if $M,M'$ are vertically composable 2-morphisms in $\Bord{or,exp}$, then
\begin{equation}\label{eq:condA}
    c\ssig(M',M) \equiv m(M \cup_\Sigma M') - m(M) - m(M') \pmod 2
\end{equation}
and for horizontally composable 2-morphisms,
\begin{equation}\label{eq:condB}
    \tilde{c}\ssig(M',M) \equiv m(M \cup_\Sigma M') - m(M) - m(M') \pmod 2.
\end{equation}

We present an explicit $m$ which we will check to satisfy the conditions:
\begin{defn}\label{defn:sig/2}
    Let $\mathbb{M}$ be the set of 2-morphisms of $\Bord{or,exp}$. Define the function $m:\mathbb{M}\rightarrow\mathbb{Z}/2$ as
    \begin{equation*}
        m(M) = b_1(\ol{M}) + b_0(\ol{M}) + \frac12 b_1(\ol{\parout M}) + b_0(\ol{\parout M}),
    \end{equation*}
    where $b_i$ denotes the $i^\text{th}$ Betti number, $b_i(X) := \dim H_i(X;\mathbb{R})$.

    For ease of writing, we will let $m_1(M)$ denote the sum of the first two terms on the right-hand side and $m_2(\parout M)$ the sum of the last two terms.
\end{defn}

The next few results will be devoted to checking that this $m$ satisfies (\ref{eq:condA}) and (\ref{eq:condB}). All computations in proofs are taken modulo $2$ and all homologies are taken with coefficients in $\mathbb{R}$. We first present a useful computational lemma on the parity of Wall's invariant.
\begin{lem} \label{lem:maslow}
    Let $V$ be a symplectic vector space of dimension $2g$ and $A, B, C$ Lagrangian subspaces.
    Then
    \begin{equation*}
        \sigma(V; A, B, C) \equiv g + \dim(A \cap B) + \dim(B \cap C) + \dim(C \cap A) \pmod2.
    \end{equation*}
\end{lem}
\begin{proof}
    As $\sigma$ is non-degenerate, the signature is the same parity of the dimension of the vector space, i.e.
    \begin{equation*}
        \sigma(V; A, B, C) = \dim\pa{\frac{A \cap (B+C)}{(A\cap B) + (A\cap C)}}.
    \end{equation*}
    Applying the fact that $\dim(X+Y) + \dim(X\cap Y) = \dim X + \dim Y$ repeatedly, we obtain
    \begin{equation*}
        \begin{split}
            \sigma(V; A, B, C) 
            &= \dim\pa{A \cap (B+C)} + \dim\pa{(A\cap B) + (A\cap C)} \\
            &= g + \dim(B+C) + \dim(A \cap B) + \dim(A \cap C) \\
            &\qquad + \dim(A+B+C) +  \dim(A \cap B \cap C) \\
            &= g + \dim(B \cap C) + \dim(A \cap B) + \dim(A \cap C) \\
            &\qquad + \dim(A+B+C) +  \dim(A \cap B \cap C). \\
        \end{split}
    \end{equation*}
    Finally, notice that 
    \begin{equation*}
        (A+B+C)^\perp = A^\perp \cap B^\perp \cap C^\perp = A\cap B\cap C
    \end{equation*}
    and hence
    \begin{equation*}
        \dim(A+B+C) +  \dim(A \cap B \cap C) = 0
    \end{equation*}
    and so the two extra terms cancel out.
\end{proof}

Hence, by working modulo $2$, the $\sigma(V; A, B, C)$ terms may be written in terms of dimensions of vector subspaces of $H_1(\ol{\Sigma})$ and we may disregard the bilinear form $\langle\cdot,\cdot\rangle$. We now proceed to check the two conditions on $m$. The arguments here do not depend on orientations, and so we drop the minus signs which correspond to orientation reversal.

\begin{prop} \label{prop:condA}
    $m$ satisfies (\ref{eq:condA}).
\end{prop}
\begin{proof}
    We use the notation in (\ref{eq:csig}). Let $N = \ol{M} \setminus \mathring{H}_\Sigma$, $N' = \ol{M'} \setminus \mathring{H}_\Sigma$ and $M'' = M \cup_\Sigma M'$, so we have $$\ol{M''} = \ol{M \cup_\Sigma M'} = N \cup_{\ol{\Sigma}} N'.$$

    Consider the Mayer-Vietoris sequence of this union:
    \begin{equation*}
        \begin{split}
            \cdots &\rightarrow H_1(\ol{\Sigma}) \overset{\alpha}\rightarrow H_1(N) \oplus H_1(N') \rightarrow H_1(\ol{M''}) \\
            &\rightarrow H_0(\ol{\Sigma}) \rightarrow H_0(N) \oplus H_0(N') \rightarrow H_0(\ol{M''}) \rightarrow 0
        \end{split}
    \end{equation*}
    Note that $\ker \alpha$ is exactly $A \cap C$, and hence we have
    \begin{equation*}
        b_1(\ol{M''}) + b_0(\ol{M''})= \dim(A\cap C) + b_1(N) + b_1(N') 
        + b_0(\ol{\Sigma}) + b_0(N) + b_0(N').
    \end{equation*}
    Replacing $N'$ with $H_\Sigma$, since $\ol{M} = N \cup_{\ol{\Sigma}} H_{\Sigma}$, we instead obtain
    \begin{equation*}
        b_1(\ol{M}) + b_0(\ol{M})= \dim(B\cap C) + b_1(N) + b_1(H_\Sigma) 
        + b_0(\ol{\Sigma}) + b_0(N) + b_0(H_\Sigma).
    \end{equation*}
    Similarly, we have 
    \begin{equation*}
        b_1(\ol{M'}) + b_0(\ol{M'})= \dim(A\cap B) + b_1(N') + b_1(H_\Sigma) 
        + b_0(\ol{\Sigma}) + b_0(N') + b_0(H_\Sigma).
    \end{equation*}
    Notice that, upon adding these three equations, most terms on the right-hand sides cancel, leaving us with
    \begin{equation*}
        m_1(M) + m_1(M') + m_1(M'') = \dim(A\cap B) + \dim(A\cap C) + \dim(B\cap C) + b_0(\ol{\Sigma}).
    \end{equation*}
    Additionally, using the fact that $\parout M'' = \parout M'$ and  applying Lemma \ref{lem:maslow}, we get
    \begin{equation*}
        \begin{split}
            m(M) + m(M') + m(M'')
            &= m_1(M) + m_1(M') + m_1(M'') + m_2(\parout M) \\
            &= b_0 (\ol{\Sigma}) + \frac12 b_1(\ol{\Sigma}) + \sigma(V;A,B,C) + b_0(\ol{\Sigma}) + \frac12 b_1(\ol{\Sigma}) \\
            &= \sigma(V;A,B,C)
        \end{split}
    \end{equation*}
    which by (\ref{eq:csig-def}) gives us the desired (\ref{eq:condA}).
\end{proof}

This automatically lets us deduce that $m$ sends identity 2-morphisms to $0$. We prove a slightly stronger version of this:
\begin{lem} \label{lem:m-norm}
    Let $(\Sigma, H_{\Sigma})$ and $(\Sigma', H_{\Sigma'})$ be 1-morphisms in $\Bord{or,exp}$. Let $M$ be a mapping cylinder of an orientation-preserving diffeomorphism from $\Sigma$ to $\Sigma'$ which extends to a diffeomorphism between $H_{\Sigma}$ and $H_{\Sigma'}$. Then, $m(M) = 0$.

    In particular, $m$ sends identity 2-morphisms to $0$.
\end{lem}
\begin{proof}
    Let $M'$ be any 2-morphism in $\Bord{or,exp}$ whose source is $(\Sigma', H_{\Sigma'})$. Then, the composition $M''$ of $M,M'$ has $\ol{M''}$ diffeomorphic to $\ol{M'}$ and $\parout M'' = \parout M'$, and hence $m(M'') = m(M')$. By (\ref{eq:condA}), $m(M)$ has the same parity as $c\ssig(M',M)$, which is $0$ by Lemma \ref{lem:c-norm}.
\end{proof}

Before we proceed to the other condition, we first prove a similar result for surfaces glued along their common boundary.

\begin{lem} \label{lem:surface-gluing}
    Let $\Sigma,\Sigma'$ be orientable surfaces with common boundary $\partial\Sigma = \partial\Sigma' = S$. Let $\Sigma'' := \Sigma\cup_S\Sigma'$. Then 
    \begin{equation*}
        g(\Sigma'') - b_0(\Sigma'') = g(\Sigma) + g(\Sigma') - b_0(\Sigma) - b_0(\Sigma') + b_0(S)
    \end{equation*}
    where $g(\Sigma)$ denotes the sum of the genus over the components of $\Sigma$.
\end{lem}
\begin{proof}
    Let $D$ be a disjoint union of discs with boundary $S$, and $H,H'$ be disjoint unions of handlebodies with boundaries $\Sigma\cup_S D$ and $\Sigma'\cup_S D$. Then, $H''=H\cup_D H'$ is a handlebody with boundary $\Sigma''$. Consider the Mayer-Vietoris sequence for this union:
    \begin{equation*}
        \begin{split}
            0 &\rightarrow H_1(H) \oplus H_1(H') \rightarrow H_1(H'') \\
            \rightarrow H_0(D) &\rightarrow H_0(H) \oplus H_0(H') \rightarrow H_0(H'') \rightarrow 0
        \end{split}
    \end{equation*}
    The alternating sum of dimensions of terms in this long exact sequence is $0$. Noting that $b_0(D) = b_0(S)$, $b_0(H) = b_0(\Sigma)$ and  $b_1(H) = g(\Sigma)$ and that the analogous results hold for $H',H''$, our result follows.
\end{proof}

\begin{prop} \label{prop:condB}
    $m$ satisfies (\ref{eq:condB}).
\end{prop}
\begin{proof}
    We use the notation in (\ref{eq:ctildesig}). Let $N = \ol{M} \setminus \mathring{V}_\Sigma$, $N' = \ol{M'} \setminus \mathring{V}_\Sigma$ and $M'' = M \cup_\Sigma M'$, so we have $$\ol{M''} = \ol{M \cup_\Sigma M'} = N \cup_{\ol{\Sigma}} N'.$$

    The same computation as in Proposition \ref{prop:condA} gives us
    \begin{equation} \label{eq:condB-intermediate}
        m_1(M) + m_1(M') + m_1(M'') = \sigma(V;A,B,C) + b_0(\ol{\Sigma}) + \frac12 b_1(\ol{\Sigma}).
    \end{equation}

    Let $S = \parout\parout M = \parin\parout M'$. Then, $\ol{\Sigma}$ is diffeomorphic to $D_S \cup_{S} D_S$. Applying Lemma \ref{lem:surface-gluing} to this union, we obtain 
    \begin{equation*}
        b_0(\ol{\Sigma}) + \frac12 b_1(\ol{\Sigma}) = b_0(S).
    \end{equation*}

    On the other hand, let $$K = \parout M \cup_{\parin\parout M} D_{\parin\parout M}, \quad K' = \parout M' \cup_{\parout\parout M'} D_{\parout\parout M'},$$ so we have $$\ol{\parout M} = K\cup_S D_S, \quad \ol{\parout M'} = K\cup_S D_S, \quad \ol{\parout M''} = K \cup_S K'.$$
    Applying Lemma \ref{lem:surface-gluing} to these three unions, we obtain:
    \begin{equation*}
        \begin{split}
            m_2(\parout M) &= g(K) + g(D_S) + b_0(K) + b_0(D_S) + b_0(S) \\
            m_2(\parout M') &= g(K') + g(D_S) + b_0(K') + b_0(D_S) + b_0(S) \\
            m_2(\parout M') &= g(K) + g(K') + b_0(K) + b_0(K') + b_0(S)
        \end{split}
    \end{equation*}
    Adding these three equations to (\ref{eq:condB-intermediate}), we obtain
    \begin{equation*}
        m(M) + m(M') + m(M'') = \sigma(V;A,B,C) + 4b_0(S) = \sigma(V;A,B,C)
    \end{equation*}
    which by (\ref{eq:ctildesig-def}) gives us the desired (\ref{eq:condB}).
\end{proof}

Hence, we may indeed define $\mathbf{c}\ssigh$ as outlined in \textsection\ref{subsection:sig/2-overview}: 
\begin{defn}
    Let $\partial\mathbf{m}$ denote the coboundary of $(m,0,0)$, viewing $m$ as a function with values in $\set{0,1} \subset \mathbb{Z}$. Then, the cocycle $\mathbf{c}\ssig - \partial\mathbf{m}$ takes values in $2\mathbb{Z}$, and so we may define
    \begin{equation*}
        \mathbf{c}\ssigh := \frac{\mathbf{c}\ssig - \partial\mathbf{m}}2.
    \end{equation*}
    The \emph{half-signature bordism bicategory} $\Bord{sig/2}$ is the extension of $\Bord{or,exp}$ corresponding to this cocycle $\mathbf{c}\ssigh$.
\end{defn}

Thus, we have proven by explicit construction the following:
\begin{thm} \label{thm:exists-sig/2}
    The class $[\mathbf{c}\ssig]$ is divisible by $2$ in $\Ext(\Bord{or,exp},\mathbb{Z})$.
\end{thm}
\begin{proof}
    By construction, $2[\mathbf{c}\ssigh] = [\mathbf{c}\ssig]$.
\end{proof}
\begin{rk}
    By Theorem \ref{thm:equiv-ext}, an equivalence of bicategories induces an isomorphism on the groups of extensions. Thus, this result is independent of the choice of model of $\Bord{or}$.
\end{rk}

Based on the explicit formula of $m$ given in Definition \ref{defn:sig/2}, we may identify $\Bord{sig/2}$ as a subbicategory of $\Bord{sig}$.
\begin{constr} \label{constr:sig/2-out}
    Call a 2-morphism $(M,n)$ of $\Bord{sig}$ \emph{even} if $n \equiv m(M) \pmod{2}$, and \emph{odd} otherwise. Then, in each $\mathbb{Z}$-family corresponding to the preimage of a 2-morphism under the quotient $\Bord{sig}\rightarrow\Bord{or,exp}$, 2-morphisms alternate between even and odd. By (\ref{eq:condA}) and (\ref{eq:condB}), the horizontal and vertical compositions of even 2-morphisms are even. Hence, we have a subbicategory $\Bord{sig/2,out}$ of $\Bord{sig}$ consisting of all its objects and 1-morphisms, but only the even 2-morphisms.

    There is an action of $\mathbb{Z}$ on the 2-morphisms if $\Bord{sig/2,out}$, given by $a\cdot (M,n) = (M,n+2a)$ for each $a\in\mathbb{Z}$. Then, the composition $\Bord{sig/2,out} \rightarrow \Bord{sig} \rightarrow \Bord{or,exp}$ is a quotient by this action, and so $\Bord{sig/2,out}$ is an extension of $\Bord{or,exp}$. By construction, it is an index $2$ subextension of $\Bord{sig}$.
\end{constr}
\begin{rk}
    The subbicategory $\Bord{sig/2,out}$, upon restriction to the hom category between $(\emptyset,\emptyset)$ and itself, corresponds to Gilmer's subextension of even morphisms, cf. \cite[Proposition 7.3]{Gilmer}.
\end{rk}

\begin{prop}
    $\Bord{sig/2,out}$ corresponds to $[\mathbf{c}\ssigh]\in\Ext(\Bord{or,exp},\mathbb{Z})$.
\end{prop}
\begin{proof}
    For each 2-morphism $M$ in $\Bord{or,exp}$, we choose its lift in $\Bord{sig/2,out}$ to be $(M,m(M))$. Then, if $M,M'$ are vertically composable 2-morphisms in $\Bord{or,exp}$ with composition $M''$, the composition of their lifts is 
    \begin{equation*}
        \begin{split}
            (M'', m(M) + m(M') + c\ssig(M',M)) &= (M'', m(M'') + 2c\ssigh(M',M)) \\
            &= c\ssigh(M',M) \cdot (M'', m(M'')).
        \end{split}
    \end{equation*}
    A similar computation for horizontal computation holds.

    For the $a,\ell,r$ terms, note that $m$ sends associators and unitors to $0$, by Lemma \ref{lem:m-norm}. Hence, we have $a\ssigh = \ell\ssigh = r\ssigh = 0$. Correspondingly, the associators and unitors of $\Bord{sig/2,out}$ are those of $\Bord{sig}$, which all have signature term $0$.
\end{proof}

Thus, we have identified $\Bord{sig/2}$ with an index $2$ subextension of $\Bord{sig}$.

\begin{lem} \label{lem:sig/2-mon}
    In Construction \ref{constr:sig/2-out}, even 2-morphisms are closed under the monoidal product.
\end{lem}
\begin{proof}
    This is equivalent to showing that
    \begin{equation}
        m(M \sqcup M') = m(M) + m(M') \pmod2,
    \end{equation}
    which is true as the homology of a disjoint union is a direct sum of homologies.
\end{proof}

Thus, $\Bord{sig/2,out}$ is in fact a symmetric monoidal subbicategory of $\Bord{sig}$, and hence $\Bord{sig/2}$ is a symmetric monoidal extension of $\Bord{or,exp}$.
\begin{rk}
    In the formula for $m$, we could have replaced all instances of $\parout M$ with $\parin M$ and obtained a different $m$ for which all the results above still hold. However, we will show later that this choice ends up producing an equivalent symmetric monoidal extension.
\end{rk}

\subsection{An automorphism of $\Bord{sig}$} \label{subsection:aut-bordsig}

In the previous subsection, we constructed a symmetric monoidal extension $\Bord{sig/2}$ of $\Bord{or,exp}$. This depended on a function $m$, which is explicitly defined in Definition \ref{defn:sig/2}. However, there is another obvious choice for $m$, given by replacing the instances of $\parout M$ in its definition with $\parin M$. This also satisfies (\ref{eq:condA}) and (\ref{eq:condB}), by analogous arguments to Propositions \ref{prop:condA} and \ref{prop:condB}.

Then, by following Construction \ref{constr:sig/2-out} with this new choice of $m$, we may define a symmetric monoidal extension $\Bord{sig/2,in}$ of $\Bord{or,exp}$. We shall show that this is isomorphic to $\Bord{sig/2,out}$, and so this alternate choice of $m$ results in the same extension. In fact, there exists a symmetric monoidal automorphism of $\Bord{sig}$ which sends $\Bord{sig/2,in}$ and $\Bord{sig/2,out}$ to each other.

Let $\hatC\rightarrow\mathbf{C}$ be an extension of bicategories by an abelian group $A$. We want to construct strict isomorphisms $\hatC\rightarrow\hatC$ respecting the quotient to $\mathbf{C}$ as well as the action by $A$. This is equivalent to assigning to each 2-morphism $\alpha$ of $\mathbf{C}$ an element $\delta(\alpha)$ of $A$, and then mapping each lift $(\alpha,a)$ of $\alpha$ in $\hatC$ to $(\alpha,a+\delta(a))$. This is an isomorphism if and only if $\delta:\mathbf{C}\rightarrow B^2A$ defines a strict functor. If $\mathbf{C}$ is a symmetric monoidal extension and we require the isomorphism to also preserve the symmetric monoidal structure, then this is equivalent to $\delta$ being strictly monoidal in the sense that $\delta(\alpha\otimes\beta) = \delta(\alpha) + \delta(\beta)$ for all 2-morphisms $\alpha,\beta$.

In the case $\mathbf{C} = \Bord{or,exp}$, in order to construct our desired automorphism of $\Bord{sig}$, we have to check the following:
\begin{lem}
    Let $\mathbb{M}$ be the set of 2-morphisms of $\Bord{or,exp}$. The function $\delta:\mathbb{M}\rightarrow\mathbb{Z}$ given by
    \begin{equation*}
            \delta(M) = \pa{\frac12 b_1(\ol{\parin M};\mathbb{R}) - b_0(\ol{\parin M};\mathbb{R})} 
            - \pa{\frac12 b_1(\ol{\parout M};\mathbb{R}) - b_0(\ol{\parout M};\mathbb{R})}
    \end{equation*}
    defines a strict monoidal functor $\Bord{or,exp}\rightarrow B^2\mathbb{Z}$.
\end{lem}
\begin{proof}
    This is clear for vertical composition and for disjoint unions. It suffices to consider horizontal composition. We use the same notation as Proposition \ref{prop:condB}.
    
    Applying Lemma \ref{lem:surface-gluing} and noting that $b_0(D_S) = b_0(S)$, we obtain
    \begin{equation*}
        \begin{split}
            \frac12 b_1(\ol{\parin M}) - b_0(\ol{\parin M}) &= g(K) + g(D_S) - b_0(K) \\
            \frac12 b_1(\ol{\parin M'}) - b_0(\ol{\parin M'}) &= g(K') + g(D_S) - b_0(K') \\
            \frac12 b_1(\ol{\parin M''}) - b_0(\ol{\parin M''}) &= g(K) + g(K') - b_0(K) - b_0(K') + b_0(S)
        \end{split}
    \end{equation*}
    A similar computation holds for the $\parout$ case. Thus, we obtain
    \begin{equation*}
        \delta(M'') - \delta(M) - \delta(M') = \pa{b_0(S) - 2g(D_S)} - \pa{b_0(S) - 2g(D_S)} = 0.
    \end{equation*}
\end{proof}

Hence, for any (symmetric monoidal) extension $\Bord{ext}\rightarrow\Bord{or,exp}$, this $\delta$ induces an isomorphism $\Delta_{\mathrm{ext}}:\Bord{ext}\rightarrow\Bord{ext}$ of (symmetric monoidal) bicategories respecting the quotient to $\Bord{or,exp}$ and the $\mathbb{Z}$-action:
\[\begin{tikzcd}
	\Bord{ext} & {} & \Bord{ext} \\
	& \Bord{or,exp}
	\arrow["\Delta_{\mathrm{ext}}", from=1-1, to=1-3]
	\arrow[from=1-1, to=2-2]
	\arrow[from=1-3, to=2-2]
\end{tikzcd}\]
Moreover, in the case where the extension is $\Bord{sig}$, the automorphism $\Delta\ssig$ sends $\Bord{sig/2,in}$ to $\Bord{sig/2,out}$ and vice versa. These two are thus isomorphic as symmetric monoidal extensions, and it thus makes sense to denote both of them as $\Bord{sig/2}$.

Consider an inclusion $\iota:\Bord{sig/2}\hookrightarrow\Bord{sig}$ as an index $2$ subextension. The $\mathbb{Z}$-action on $\Bord{sig/2}$ corresponds to twice the $\mathbb{Z}$-action on $\Bord{sig}$, i.e. for a 2-morphism $\hat{\alpha}$ in $\Bord{sig/2}$, $\iota(n\cdot\hat{\alpha}) = 2n\cdot\iota(\hat{\alpha})$. Hence, we have a way of comparing $\Delta\ssig$ and $\Delta\ssigh$; the following diagram commutes.
\[\begin{tikzcd}
    \Bord{sig/2} & \Bord{sig} \\
    \Bord{sig/2} & \Bord{sig}
    \arrow["{\iota}", from=1-1, to=1-2]
    \arrow["{\Delta_{\mathrm{sig}/2}}"', from=1-1, to=2-1]
    \arrow["{\Delta\ssig^2}", from=1-2, to=2-2]
    \arrow["{\iota}"', from=2-1, to=2-2]
\end{tikzcd}\]

In fact, letting $\iota_n := \Delta_{\mathrm{sig}}^n\circ\iota$ for each $n\in\mathbb{Z}$, what we really have is a $\mathbb{Z}$-family of inclusions $\Bord{sig/2}\hookrightarrow\Bord{sig}$. In the case where $\iota$ is as defined in Construction \ref{constr:sig/2-out}, the image of $\iota_n$ in $\Bord{sig}$ is $\Bord{sig/2,in}$ if $n$ is odd and $\Bord{sig/2,out}$ if $n$ is even.

\subsection{Componentwise versions} In \cite{BDSV4}, there is also a componentwise version of $\Bord{sig}$ for which a 2-morphism with $c$ connected components is given the extra data of a $c$-tuple of integers, thought of as signatures of 4-manifolds bounded by each of the connected components. While this is not an extension of $\Bord{or}$ in the sense of Definition~\ref{defn:ext}, it has some representation-theoretic significance. We record a half-signature analogue of this symmetric monoidal bicategory; as will be shown in \textsection\ref{section:rep-mtc}, its linear representations have a classification in terms of modular tensor categories.
\begin{defn}\label{defn:bordcsig}
    The \emph{componentwise signature bordism bicategory} $\Bord{csig}$ is the symmetric monoidal bicategory with 
    \begin{itemize}
        \item Objects and 1-morphisms: same as $\Bord{or,exp}$
        \item 2-morphisms: equivalence classes of $(M,\mathbf{n})$, where $M = \sqcup_{i=1}^c M_i$, each $M_i$ connected, and $\mathbf{n} = (n_1,\ldots,n_c)$ is an $c$-tuple of integers.
    \end{itemize}
    When a vertical or horizontal composition involves multiple connected components, the signature term assigned to the combined connected component is given by the sum of all the signature terms, plus the $c\ssig$ term given by Wall's invariant.

    The \emph{componentwise half-signature bordism bicategory} $\Bord{csig/2}$ is the symmetric monoidal subbicategory of $\Bord{csig}$ consisting only of the 2-morphisms $(M,\mathbf{n})$ such that $n_i \equiv m(M_i) \pmod{2}$ for each $i$.
\end{defn}

These admit symmetric monoidal functors to $\Bord{sig}$ and $\Bord{sig/2}$ respectively, given by summing up the entries of the tuples $\mathbf{n}$.

\section{Minimality of $\bord{sig/2}$} \label{section:indiv}

Starting from the extension $\Bord{sig}$ of $\Bord{or,exp}$, we have shown that it is divisible by $2$ as an element in $\Ext(\Bord{or,exp},\mathbb{Z})$. A natural question to ask then is whether it can be further divided, i.e. if $\Bord{sig/2}$ is minimal.

This question has been well-studied at the level of mapping class groups, which arise as automorphism groups of 1-morphisms in $\Bord{or}$. Upon restriction to these, $\Bord{sig}$ yields a central extension of mapping class groups by $\mathbb{Z}$, commonly known as the \emph{signature central extension}. It turns out that this signature central extension not only admits an index $2$ subextension, but also an index $4$ one.

In the light of this fact, Gilmer and Masbaum \cite[Remark 7.5]{GM} asked if the same would be true for the signature bordism category $\bord{sig}$. Via homotopy-theoretic methods, this was answered by Schommer-Pries \cite[\textsection 7.5]{csp-inv-tft} in the negative for symmetric monoidal extensions: Gilmer's index two symmetric monoidal subextension $\bord{sig/2}$ of $\bord{sig}$ does not admit any further symmetric monoidal subextensions.

In this section, we present a proof that $\bord{sig/2}$ in fact does not admit any subextensions (so without any reference to the monoidal structure), purely via cocycle computations.

\subsection{Restriction to mapping class groups}

First, we review the signature central extension of mapping class groups. Let $\Gamma_g$ be the mapping class group for a closed oriented surface $\Gamma_g$ of genus $g$. Then, there is an inclusion functor $B\Gamma_g \rightarrow \bord{or,exp}$ sending the one object to $(\Gamma_g,H_{\Gamma_g})$ for some choice of handlebody $H_{\Gamma_g}$, and each $[f]\in\Gamma_g$ to the mapping cylinder $M_f$ of $f$. Different choices of $f$ result in diffeomorphic choices of $M_f$, i.e. the same morphism in $\bord{or,exp}$.

This inclusion induces a homomorphism $\Ext(\bord{or,exp},\mathbb{Z}) \rightarrow H^2(\Gamma_g;\mathbb{Z})$, given by restricting cocycles. Composing this with the homomorphism in Lemma \ref{lem:res-cocyc}, we obtain a restriction homomorphism $\Ext(\Bord{or,exp},\mathbb{Z}) \rightarrow H^2(\Gamma_g;\mathbb{Z})$. Under this, $[\mathbf{c}\ssig]$ and $[\mathbf{c}\ssigh]$ are sent to some $[c\ssig], [c\ssigh]\in H^2(\Gamma_g;\mathbb{Z})$.

The central extension of $\Gamma_g$ represented by $[c\ssig]$ is commonly known as the \emph{signature central extension}. This was first defined by Meyer \cite{Meyer} using the signatures of certain surface bundles and later reformulated by Walker \cite{Walker} in the more familiar terms of Lagrangian subspaces:

Let $L = \ker(H_1(\Sigma_g)\rightarrow H_1(H_{\Sigma_g}))$ be the Lagrangian subspace associated to the handlebody bounded by $\Sigma_g$. Then, for $[f],[f']$ in the mapping class group of $\Sigma_g$, applying (\ref{eq:csig-def}) to $M = M_f$ and $M' = M_{f'}$ gives
\begin{equation}\label{eq:sw}
    c\ssig([f'],[f]) = -\sigma(V;A,B,C) = \sigma(\mathbb{R}^{2g}, f_*L,L,(f')_*^{-1}L) = \sigma(\mathbb{R}^{2g};(f'f)_*L,f_*'L,L).
\end{equation}
In the second step, we swapped the positions of $A,C$ in order to remove the minus sign. The expression on the right-hand side is commonly known as the \emph{Shale--Weil cocycle} \cite{Walker,LV}.

By finding a group presentation for the signature central extension, Masbaum and Roberts \cite{MR} identified $[c\ssig]$ as an element of $H^2(\Gamma_g)$, which is a cyclic group for all $g$:
\begin{thm}[{\cite[Theorem 3.10]{MR}}]\label{thm:csig-4}
    The cocycle $[c\ssig]$ is $4$ times a generator in $H^2(\Gamma_g)$.
\end{thm}

In particular, for $g\ge3$, we have $H^2(\Gamma_g)\cong\mathbb{Z}$ \cite{Harer,korkstip}. Then, under this isomorphism, $[c\ssigh]$ corresponds to $2$, so it can be divided by $2$ and no further. An explicit construction of its index $2$ subextension is given in \cite[\textsection 5]{GM}.

\subsection{Indivisibility of $[c\ssigh]$ by $2$.}

In the light of the fact that $[c\ssigh]$ can indeed be further subdivided by $2$ at the level of mapping class groups, Gilmer and Masbaum \cite[Remark 7.5]{GM} asked if the same is true over $\bord{or,lag}$ (and so equivalently, over $\bord{or,exp}$ or $\bord{or}$). The following proposition answers this in the negative:
\begin{prop} \label{prop:indiv}
    $[c\ssig]\in\Ext(\bord{or,exp},\mathbb{Z})$ is not divisible by $4$.
\end{prop}
\begin{proof}
    Let $\mathbb{M}'$ denote the morphisms of $\bord{or,exp}$. It suffices to show that there does not exist a function $m: \mathbb{M}' \rightarrow \mathbb{Z}/4$ satisfying
    \begin{equation}\label{eq:div-4}
        c\ssig(M',M) \equiv m(M\cup_\Sigma M') - m(M) - m(M') \pmod4
    \end{equation}
    for all composable morphisms $M,M'$.

    Suppose there is such a function $m$. Applying (\ref{eq:div-4}) to the case where $M=M'$ is the identity morphism, we see that $m$ evaluates to $0$ on all identity morphisms.

    Now, we restrict our attention to the objects $(\emptyset, \emptyset)$ and $(\Sigma, H_\Sigma) = (S^1 \times S^1, D^2 \times S^1)$ (with the natural inclusion $S^1\times S^1 \hookrightarrow D^2 \times S^1$). Identify $H_1(\Sigma)$ with $\mathbb{R}^2$ in the obvious way, so $\ker(H_1(\Sigma) \rightarrow H_1(H_\Sigma)) = \Span\set{\twovect10}$. Let this Lagrangian subspace be $L$. Then, we can apply (\ref{eq:sw}) to the automorphisms of $(\Sigma,H_\Sigma)$.

    Viewing $\Sigma$ as $\mathbb{R}^2/\mathbb{Z}^2$, let $f:\Sigma\rightarrow\Sigma$ be the diffeomorphism induced by $T = \begin{pmatrix}-1 & -1 \\ 1 & 0\end{pmatrix} \in SL_2(\mathbb{Z})$. Then $f_* = T$, which has order $3$. Specifically,
    \begin{equation*}
        T^{-1} = T^2 = \begin{pmatrix}0 & 1 \\-1 & -1\end{pmatrix}.
    \end{equation*}
    Applying (\ref{eq:sw}) to $f, f^2$, by Lemma \ref{lem:coinc-zero}, we have
    \begin{equation*}
        c\ssig([f],[f^2]) = \sigma(\mathbb{R}^2; L, f_*L, L) = 0.
    \end{equation*}
    Hence, we have from (\ref{eq:div-4}) on $M_f,M_{f^2}$
    \begin{equation}\label{eq:f-f2}
        m(M_f) + m(M_{f^2}) = m(M_{\id_\Sigma}) = 0 \in \mathbb{Z}/4.
    \end{equation}

    On the other hand, applying (\ref{eq:sw}) to $f, f$, we instead have
    \begin{equation*}
        \sigma(V; A, B, C) = \sigma\pa{\mathbb{R}^2; \Span\set{\twovect0{-1}}, \Span\set{\twovect{-1}1}, \Span\set{\twovect10}}.
    \end{equation*}
    As the three vectors are pairwise linearly independent, this evaluates to $\pm 1$ (depending on the sign convention of the intersection form on $H_1(\Sigma)$.)
    In any case, applying (\ref{eq:div-4}) to the composition $M_{f^2} = M_f \cup_\Sigma M_f$ gives
    \begin{equation}\label{eq:f-f}
        2m(M_f) - m(M_{f^2}) = \pm1 \in\mathbb{Z}/4.
    \end{equation}
    Solving (\ref{eq:f-f2}) and (\ref{eq:f-f}) simultaneously, we see that $m(M_f), m(M_{f^2})$ are $1,-1$ in some order. In particular, they are not equal.

    Now, view $H_\Sigma$ as a morphism from $(\emptyset,\emptyset)$ to $(\Sigma, H_\Sigma)$. Let $M$ be any morphism with source $(\Sigma, H_\Sigma)$. Then, in (\ref{eq:csig}), we have $A=B$, and so by Lemma \ref{lem:coinc-zero},
    \begin{equation*}
        c\ssig(M,H_\Sigma) = 0.
    \end{equation*}
    Likewise, if $H_\Sigma^\dag$ is the same handlebody (with orientation reversed) viewed as a morphism from $(\Sigma, H_\Sigma)$ to $(\emptyset, \emptyset)$,
    then for any morphism $M$ with target $(\Sigma, H_\Sigma)$, we have
    \begin{equation*}
        c\ssig(H_\Sigma^\dag,M) = 0.
    \end{equation*}

    Hence, applying (\ref{eq:div-4}) and working in $\mathbb{Z}/4$, we have
    \begin{equation*}
        \begin{split}
            m(H_\Sigma \cup_\Sigma M_f \cup_\Sigma H_\Sigma^\dag) &= m(H_\Sigma) + m(M_f) +m(H_\Sigma^\dag) \\
            &\ne m(H_\Sigma) + m(M_{f^2}) + m(H_\Sigma^\dag) = m(H_\Sigma \cup_\Sigma M_{f^2} \cup_\Sigma H_\Sigma^\dag).
        \end{split}
    \end{equation*}
    However, the closed 3-manifolds on the left and right are both $S^3$, and so have an orientation-preserving diffeomorphism between them, i.e. they are the same morphism in $\bord{or,exp}$. This is a contradiction.
\end{proof}

\subsection{Indivisibility of $[\mathbf{c}\ssigh]$}

Finally, this allows us to conclude that the group element $[c\ssigh] \in \Ext(\bord{or,exp},\mathbb{Z})$ is indivisible, and hence that $[\mathbf{c}\ssigh] \in \Ext(\Bord{or,exp},\mathbb{Z})$, is indivisible as well.

\begin{thm}\label{thm:indiv}
    For all $n \ge 2$, $[c\ssigh] \in \Ext(\bord{or,exp},\mathbb{Z})$ is not divisible by $n$.
\end{thm}
\begin{proof}
    Suppose $[c\ssigh]\in\Ext(\bord{or,exp},\mathbb{Z})$ is divisible by some $n \ge 3$, then for any $g\ge 3$, $[c\ssig]\in H^2(\Gamma_g)$ is divisible by $2n$. However, this contradicts Theorem \ref{thm:csig-4}. On the other hand, Proposition \ref{prop:indiv} rules out the $n=2$ case.
\end{proof}
\begin{cor}
    For all $n \ge 2$, $[\mathbf{c}\ssigh]\in\Ext(\Bord{or,exp},\mathbb{Z})$ is not divisible by $n$.
\end{cor}
\begin{proof}
    The homomorphism in Lemma \ref{lem:res-cocyc} sends $[\mathbf{c}\ssigh]$ to an element which, by Theorem \ref{thm:indiv}, is indivisble by $n$.
\end{proof}
\begin{rk}
    By Theorem \ref{thm:equiv-ext}, an equivalence of bicategories induces an isomorphism on the groups of extensions. Thus, this result is independent of the choice of model of $\Bord{or}$.
\end{rk}

\section{A finite presentation of $\Bord{sig/2}$}
\label{section:pres}

In the approach of \cite{BDSV4}, the classification of linear representations of $\Bord{or}$ and $\Bord{sig}$ was done by the construction of certain presentations for these symmetric monoidal bicategories. In this section, we provide a modification of these which presents the symmetric monoidal bicategory $\Bord{sig/2}$. Then, in the next section, we will use this, again following the approach of \cite{BDSV4}, to classify the linear representations of $\Bord{sig/2}$.

Throughout this section and the next, for a presentation $\mathcal{X}$, we let $\mathbf{F}(\mathcal{X})$ denote the free symmetric monoidal bicategory generated by $\mathcal{X}$. This is constructed explicitly in \cite[\textsection 2.10]{csp-phd}. It is easy to define symmetric monoidal functors out of $\mathbf{F}(\mathcal{X})$, as well as transformations between them: Let $\mathbf{C}$ be a symmetric monoidal bicategory. Then, by \cite[Theorem 2.78]{csp-phd}, the bicategory of symmetric monoidal functors $\mathbf{F}(\mathcal{X}) \rightarrow \mathbf{C}$ is equivalent to a certain bicategory $\mathcal{X}(\mathbf{C})$ of ``$\mathcal{X}$-data'' in $\mathbf{C}$. In particular, the objects of $\mathcal{X}(\mathbf{C})$ are assignments of each of the generating objects, 1-morphisms and 2-morphisms of $\mathcal{X}$ to objects, 1-morphisms and 2-morphisms of $\mathbf{C}$ such that the relations of $\mathcal{X}$ are satisfied. The 1-morphisms of $\mathcal{X}(\mathbf{C})$ are assignments of the generating objects and 1-morphisms to 1-morphisms and 2-morphisms in $\mathbf{C}$ satisfying a naturality relation for each generating 2-morphism of $\mathcal{X}$. We refer the reader to \cite[Definition 2.76]{csp-phd} for more details.

\subsection{The global half-signature presentation} In this subsection, we construct a presentation $\mathcal{G}$ which presents $\Bord{sig/2}$. This is a modification of the presentations given in \cite{BDSV4}. As we will be referring to the specific generators and relations of these presentations, we shall first reproduce them here.

The two new presentations in this subsection are the half-signature presentation $\mathcal{H}$ and the global half-signature presentation $\mathcal{G}$. These present $\Bord{csig/2}$ and $\Bord{sig/2}$ respectively, but the proof of this is delayed to \textsection\ref{subsection:constr-modg}.

\begin{defn}[{\cite[Definition 3.7]{BDSV4}}] \smallbordisms
    The \emph{ribbon presentation} $\mathcal{R}$ is the presentation with 
    \begin{itemize}
        \item Generating object:
            \begin{equation*}
                \begin{tz}
                    \node[Cyl, top, height scale=0]  at (0,0) {};
                \end{tz}
            \end{equation*}
        \item Generating 1-morphisms:
            \begin{equation*}
                \begin{tz}
                    \node[Pants, top, bot] (A) at (0,0) {};
                    \node[Copants, top, bot] (B) at (2,0) {};
                    \node[Cup, top] (C) at (4,0.1) {};
                    \node[Cap, bot] (D) at (6,-0.1) {};
                \end{tz}
            \end{equation*}
        \item Generating 2-morphisms:
            \begin{calign}
                \nonumber
                \begin{tz}
                    \node[Pants, top, bot, wide] (A) at (0,0) {};
                    \node[Pants,  bot, anchor=belt] (B) at (A.leftleg) {};    
                    \node[Cyl, bot, anchor=top] at (A.rightleg) {}; 
                \end{tz}    
                \rightleftdoublearrow{\alpha}{\alpha^{-1}}
                \begin{tz}
                    \node[Pants, top, bot, wide] (A) at (0,0) {};
                    \node[Pants,  bot, anchor=belt] (B) at (A.rightleg) {};    
                    \node[Cyl, bot, anchor=top] at (A.leftleg) {}; 
                \end{tz} 
                &
                \begin{tz}
                    \node[Pants, top, bot] (A) at (0,0) {};
                    \node[Cyl, bot, anchor=top] at (A.leftleg) {};
                    \node[Cup] at (A.rightleg) {};  
                \end{tz}
                \rightleftdoublearrow{\rho}{\rho^{-1}}
                \begin{tz}
                    \node[Cyl, bot, top, tall] at (0,0) {};
                \end{tz}
                \rightleftdoublearrow{\lambda^{-1}}{\lambda}
                \begin{tz}
                    \node[Pants, top, bot] (A) at (0,0) {};
                    \node[Cyl, bot, anchor=top] at (A.rightleg) {};
                    \node[Cup] at (A.leftleg) {};   
                \end{tz}
            \\ \nonumber
                \begin{tz}
                    \node[Pants, top, bot] (A) at (0,0) {};
                \end{tz}
                \rightleftdoublearrow{\beta}{\beta^{-1}}
                \begin{tz}
                    \node[Pants, top, bot] (A) at (0,0) {};
                    \node[BraidB, anchor=topleft, bot] at (A.leftleg) {};
                \end{tz}
                &
                \begin{tz}
                    \node[Cyl, top, bot] (A) at (0,0) {};
                \end{tz}
                \rightleftdoublearrow{\theta}{\theta^{-1}}
                \begin{tz}
                    \node[Cyl, top, bot] (A) at (0,0) {};
                \end{tz}
            \\ \nonumber
                \begin{tz} 
                    \node[Cyl, tall, top, bot] (A) at (0,0) {};
                    \node[Cyl, tall, top, bot] (B) at (2*\cobwidth, 0) {};
                \end{tz}
                \longxdoubleto{\eta}
                \begin{tz} 
                    \node[Pants, bot] (A) at (0,0) {};
                    \node[Copants, top, bot, anchor=belt] at (A.belt) {};
                \end{tz}
                &
                \begin{tz} 
                    \node[Pants, top, bot] (A) at (0,0) {};
                    \node[Copants, bot, anchor=leftleg] at (A.leftleg) {};
                \end{tz}
                \longxdoubleto{\epsilon}
                \begin{tz} 
                    \node[Cyl, top, bot, tall] (A) at (0,0) {};
                \end{tz}
            \\ \nonumber
                \begin{tz}
                    \draw[green] (0,0) rectangle (0.6, -0.6);  
                \end{tz}
                \longxdoubleto{\nu}
                \begin{tz}
                    \node[Cap, bot] (A) at (0,0) {};
                    \node[Cup] at (0,0) {};
                \end{tz}
                &
                \begin{tz}
                    \node[Cup, top] (A) at (0,0) {};
                    \node[Cap, bot] (B) at (0,-2*\cobheight) {};
                \end{tz}
                \longxdoubleto{\mu}
                \begin{tz}
                    \node[Cyl, top, bot, tall] (A) at (0,0) {};
                \end{tz}
            \\ \nonumber
                \begin{tz}
                    \node[Copants, top, bot] (A) at (0,0) {};
                    \node[Pants, bot, anchor=belt] (B) at (A.belt) {};
                \end{tz} 
                \longxdoubleto{\phi_1^{-1}}
                \begin{tz}
                    \node[Pants, top, bot] (A) at (0,0) {};
                    \node[Cyl, bot, anchor=top] (B) at (A.leftleg) {};
                    \node[Copants, bot, anchor=leftleg] (C) at (A.rightleg) {};
                    \node[Cyl, top, bot, anchor=bottom] (D) at (C.rightleg) {}; 
                \end{tz}
                &
                \begin{tz}
                    \node[Copants, top, bot] (A) at (0,0) {};
                    \node[Pants, bot, anchor=belt] (B) at (A.belt) {};
                \end{tz} 
                \longxdoubleto{\phi_2^{-1}}
                \begin{tz}
                    \node[Pants, top, bot] (A) at (0,0) {};
                    \node[Cyl, bot, anchor=top] (B) at (A.rightleg) {};
                    \node[Copants, bot, anchor=rightleg] (C) at (A.leftleg) {};
                    \node[Cyl, top, bot, anchor=bottom] (D) at (C.leftleg) {}; 
                \end{tz}
            \end{calign}
    \end{itemize}
    subject to the following relations:
    \begin{itemize}
        \item (Inverses) Each of $\omega=\alpha,\rho,\lambda,\beta,\theta$ satisfies $\omega\omega^{-1} = \id$ and $\omega^{-1}\omega=\id$.
        \item (Monoidal) $\alpha,\rho,\lambda$ satisfy the pentagon and triangle equations:
        \begin{equation*}
            \begin{tz}[xscale=2.1, yscale=1.2]
                \node (1) at (0,0)
                {
                $\begin{tikzpicture}
                    \node [Pants, wider, top, bot] (A) at (0,0) {};
                    \node [Pants, wide, bot, anchor=belt] (B) at (A.leftleg) {};
                    \node [Cyl, tall, bot, anchor=top] (C) at (A.rightleg) {};
                    \node[Pants, bot, anchor=belt] (D) at (B.leftleg) {};
                    \node[Cyl, bot, anchor=top] (E) at (B.rightleg) {};
                    \selectpart[green] {(A-belt) (B-leftleg) (A-rightleg)};
                    \selectpart[red] {(A-leftleg) (D-leftleg) (B-rightleg)};
                \end{tikzpicture}$
                };
                \node (2) at (1,1)
                {
                $\begin{tikzpicture}
                    \node [Pants, verywide, top, bot] (A) at (0,0) {};
                    \node [Pants, bot, anchor=belt] (B) at (A.rightleg) {};
                    \node [Cyl, bot, anchor=top] (C) at (A.leftleg) {};
                    \node[Pants, bot, anchor=belt] (D) at (C.bottom) {};
                    \node[Cyl, bot, anchor=top] (E) at (B.leftleg) {};
                    \node[Cyl, bot, anchor=top] (F) at (B.rightleg) {};
                    \selectpart[green] {(A-leftleg) (A-rightleg) (D-leftleg) (F-bottom)};
                \end{tikzpicture}$
                };
                \node (3) at (2,1)
                {
                $\begin{tikzpicture}
                    \node [Pants, verywide, top, bot] (A) at (0,0) {};
                    \node [Pants, bot, anchor=belt] (B) at (A.leftleg) {};
                    \node[Cyl, bot, anchor=top] (C) at (A.rightleg) {};
                    \node[Cyl, bot, anchor=top] (D) at (B.leftleg) {};
                    \node[Cyl, bot, anchor=top] (E) at (B.rightleg) {};
                    \node[Pants, bot, anchor=belt] (F) at (C.bottom) {};
                    \selectpart[green] {(A-belt) (B-leftleg) (A-rightleg)};
                \end{tikzpicture}$
                };
                \node (4) at (3,0)
                {
                $\begin{tikzpicture}
                    \node [Pants, wider, top, bot] (A) at (0,0) {};
                    \node [Cyl, tall, bot, anchor=top] (B) at (A.leftleg) {};
                    \node[Pants, bot, anchor=belt, wide] (C) at (A.rightleg) {};
                    \node[Cyl, bot, anchor=top] (D) at (C.leftleg) {};
                    \node[Pants, bot, anchor=belt] (E) at (C.rightleg) {};
                \end{tikzpicture}$
                };
                \node (5) at (1,-1)
                {
                $\begin{tikzpicture}
                    \node [Pants, veryverywide, top, bot] (A) at (0,0) {};
                    \node [Pants, wide, bot, anchor=belt] (B) at (A.leftleg) {};
                    \node [Cyl, tall, bot, anchor=top] (C) at (A.rightleg) {};
                    \node[Pants, bot, anchor=belt] (D) at (B.rightleg) {};
                    \node[Cyl, bot, anchor=top] (E) at (B.leftleg) {};
                    \selectpart[green] {(A-belt) (B-leftleg) (A-rightleg)};
                \end{tikzpicture}$
                };
                \node (6) at (2,-1)
                {
                $\begin{tikzpicture}
                    \node [Pants, veryverywide, top, bot] (A) at (0,0) {};
                    \node [Pants, wide, bot, anchor=belt] (B) at (A.rightleg) {};
                    \node [Cyl, tall, bot, anchor=top] (C) at (A.leftleg) {};
                    \node[Pants, bot, anchor=belt] (D) at (B.leftleg) {};
                    \node[Cyl, bot, anchor=top] (E) at (B.rightleg) {};
                    \selectpart[green] {(A-rightleg) (D-leftleg) (B-rightleg)};
                \end{tikzpicture}$
                };
                \begin{scope}[double arrow scope]
                    \draw (1) --  node[above left, green, pos=0.65]{$\alpha$} (2);
                    \draw (2) --  node[above]{$\varphi$} (3);
                    \draw (3) --  node[above right, pos=0.35]{$\alpha$} (4);
                    \draw (1) --  node[below left, red, pos=0.65]{$\textcolour{red}{\alpha}$} (5);
                    \draw (5) --  node[below]{$\alpha$} (6);
                    \draw (6) --  node[below right, pos=0.35]{$\alpha$} (4);
                \end{scope}
            \end{tz}
        \end{equation*}
        \begin{equation*}
            \begin{tz}[xscale=1, yscale=2, every to/.style={out=down,in=up}]
                \node (1) at (-1,1)
                {
                $\begin{tikzpicture}
                    \node (A) [Pants, bot] at (0,0) {};
                    \node (B) [Pants, wide, top, bot, anchor=leftleg] at (A.belt) {};
                    \node (C) [Cup] at (A.rightleg) {};
                    \node [Cyl, bot, anchor=top] (D) at (A.leftleg) {};
                    \node [Cyl, tall, bot, anchor=top] (E) at (B.rightleg) {};
                    \selectpart[green] {(B-belt) (A-leftleg) (B-rightleg)};
                    \selectpart[red] {(B-leftleg) (D-bottom) (A-rightleg)};
                \end{tikzpicture}$
                };
                \node (2) at (1,1)
                {
                $\begin{tikzpicture}
                    \node (A) [Pants, bot] at (0,0) {};
                    \node (B) [Pants, wide, top, bot, anchor=rightleg] at (A.belt) {};
                    \node (C) [Cup] at (A.leftleg) {};
                    \node (D) [Cyl, bot, anchor=top] at (A.rightleg) {};
                    \node [Cyl, tall, bot, anchor=top] at (B.leftleg) {};
                    \selectpart[green] {(B-rightleg) (A-leftleg) (D-bottom)};
                \end{tikzpicture}$
                };
                \node (3) at (0,0)
                {
                $\begin{tikzpicture}
                    \node (A) [Pants, top, bot] at (0,0) {};
                \end{tikzpicture}$
                };
                \begin{scope}[double arrow scope]
                    \draw (1) --  node[above, green]{$\alpha$} (2);
                    \draw (2) --  node[below right, pos=0.35]{$\lambda$} (3);
                    \draw (1) --  node[below left, red, pos=0.35]{$\textcolour{red}{\rho}$} (3);
                \end{scope}
            \end{tz}
        \end{equation*}
        In the pentagon relation, $\varphi$ denotes the canonical interchanger.
        \item (Balanced) $\beta$ provides a braided monoidal structure compatible with the twist $\theta$:
        \begin{equation*}
            \begin{tz}[xscale=2.4, yscale=1.2]
                \node (1) at (0,0)
                {
                $\begin{tikzpicture}
                    \node [Pants, wide, top, bot] (A) at (0,0) {};
                    \node [Pants, bot, anchor=belt] (B) at (A.leftleg) {};
                    \node [Cyl, bot, anchor=top] (C) at (A.rightleg) {};
                    \selectpart[green]{(A-leftleg) (B-leftleg) (B-rightleg)};
                \end{tikzpicture}$
                };

                \node (2) at (0.75,1)
                {
                $\begin{tikzpicture}
                    \node [Pants, wide, top, bot] (A) at (0,0) {};
                    \node [Pants, bot, anchor=belt] (B) at (A.rightleg) {};
                    \node [Cyl, bot, anchor=top] (C) at (A.leftleg) {};
                    \selectpart[green]{(A-belt) (A-leftleg) (A-rightleg)};
                \end{tikzpicture}$
                };

                \node (3) at (1.5,1)
                {
                $\begin{tikzpicture}
                    \node [Pants, wide, top, bot] (A) at (0,0) {};
                    \node[BraidB, wide, bot, anchor=topleft] (B) at (A.leftleg) {};
                    \node [Pants, bot, anchor=belt] (C) at (B.bottomright) {};
                    \node [Cyl, bot, anchor=top] (D) at (B.bottomleft) {};
                \end{tikzpicture}$
                };

                \node (4) at (2.25,1)
                {
                $\begin{tikzpicture}
                    \node [Pants, wide, top, bot] (A) at (0,0) {};
                    \node [Pants, bot, anchor=belt] (C) at (A.leftleg) {};
                    \node [Cyl, bot, anchor=top] (D) at (A.rightleg) {};
                    \node [BraidB, bot, anchor=topleft] (E) at (C.rightleg) {};
                    \node[Cyl, bot, anchor=top] (F) at (C.leftleg) {};
                    \node[BraidB, bot, anchor=topleft] (G) at (F.bottom) {};
                    \node[Cyl, bot, anchor=top] (H) at (G.bottomright) {};
                    \node[Cyl, bot, anchor=top] (I) at (G.bottomleft) {};
                    \node[Cyl, bot, anchor=top, tall] (J) at (E.bottomright) {};
                    \selectpart[green]{(A-belt) (C-leftleg) (D-bottom)};
                \end{tikzpicture}$
                };

                \node (5) at (3,0)
                {
                $\begin{tikzpicture}
                    \node [Pants, wide, top, bot] (A) at (0,0) {};
                    \node [Pants, bot, anchor=belt] (C) at (A.rightleg) {};
                    \node [Cyl, bot, anchor=top] (D) at (A.leftleg) {};
                    \node [BraidB, bot, anchor=topleft] (E) at (C.leftleg) {};
                    \node[Cyl, tall, bot, anchor=top] (F) at (A.leftleg) {};
                    \node[BraidB, bot, anchor=topleft] (G) at (F.bottom) {};
                    \node[Cyl, bot, anchor=top] (H) at (E.bottomright) {};
                \end{tikzpicture}$
                };

                \node (6) at (1,-1)
                {
                $\begin{tikzpicture}
                    \node [Pants, wide, top, bot] (A) at (0,0) {};
                    \node [Pants, bot, anchor=belt] (B) at (A.leftleg) {};
                    \node [Cyl, bot, anchor=top, tall] (C) at (A.rightleg) {};
                    \node [BraidB, bot, anchor=topleft] (D) at (B.leftleg) {};
                    \selectpart[green]{(A-belt) (B-leftleg) (A-rightleg)};
                \end{tikzpicture}$
                };

                \node (7) at (2.0,-1)
                {
                $\begin{tikzpicture}
                    \node [Pants, wide, top, bot] (A) at (0,0) {};
                    \node [Pants, bot, anchor=belt] (B) at (A.rightleg) {};
                    \node [Cyl, bot, anchor=top] (C) at (A.leftleg) {};
                    \node[BraidB, bot, anchor=topleft] (D) at (C.bottom) {};
                    \node[Cyl, bot, anchor=top] (E) at (B.rightleg) {};
                    \selectpart[green]{(A-rightleg) (B-leftleg) (B-rightleg)};
                \end{tikzpicture}$
                };

                \begin{scope}[double arrow scope]
                    \draw (1) --  node[above left]{$\alpha$} (2);
                    \draw (2) --  node[above]{$\beta$} (3);
                    \draw[-=, shorten <=-4pt] (3) --  (4);
                    \draw (4) --  node[above right]{$\alpha$} (5);
                    \draw (1) --  node[below left]{$\beta$} (6);
                    \draw (6) --  node[below]{$\alpha$} (7);
                    \draw (7) --  node[below right]{$\beta$} (5);
                \end{scope}
            \end{tz}
        \end{equation*}
        \begin{equation*}
            \begin{tz}[xscale=1.4, yscale=1.5]

                \node (1) at (0,0)
                {
                $\begin{tikzpicture}
                        \node[Pants, top, bot] (A) at (0,0) {};
                        \selectpart[green, inner sep=1pt]{(A-belt)};
                        \selectpart[red] {(A-leftleg) (A-rightleg) (A-belt)};
                \end{tikzpicture}$
                };
                \node (2) at (1,0)
                {
                $\begin{tikzpicture}
                        \node[Pants, top, bot] (A) at (0,0) {};
                \end{tikzpicture}$
                };

                \node (3) at (0,-1)
                {
                $\begin{tikzpicture}
                        \node[Pants, top, bot] (A) at (0,0) {};
                        \selectpart[green, inner sep=1pt]{(A-leftleg)};
                \end{tikzpicture}$
                };

                \node (4) at (1,-1)
                {
                $\begin{tikzpicture}
                        \node[Pants, top, bot] (A) at (0,0) {};
                        \selectpart[green, inner sep=1pt]{(A-rightleg)};
                \end{tikzpicture}$
                };

                \begin{scope}[double arrow scope]
                    \draw (1) -- node[above, green] {$\theta$} (2);
                    \draw (1) -- node[left, red] {$\textcolour{red}{\beta}^2$} (3);
                    \draw (3) -- node[below] {$\theta$} (4);
                    \draw (4) -- node[right] {$\theta$} (2);
                \end{scope}
            \end{tz}
        \end{equation*}
        \begin{equation*}
            \begin{tz}[xscale=1.6, yscale=2]
                \node (1) at (0,0) {$
                \begin{tikzpicture}
                        \node[Cup, top] (C) at (0,0) {};
                        \selectpart[green, inner sep=1pt]{(C-center)};
                \end{tikzpicture}$};
                \node (2) at (1,0) {$
                \begin{tikzpicture}
                        \node[Cup, top] (C) at (0,0) {};
                        \selectpart[green, inner sep=1pt]{(C-center)};
                \end{tikzpicture}$};
                \begin{scope}[double arrow scope]
                    \draw (1) --  node[above]{$\theta$} (2);
                \end{scope}
            \end{tz}
            \quad = \quad
            \id
        \end{equation*}
        \item (Adjunction) $\eta,\epsilon$ witness $\tikztinypants \dashv \tikztinycopants$ while $\nu,\mu$ witness $\tikztinycup \dashv \tikztinycap$, i.e. the following are the identity:
        \begin{calign} \nonumber
            \begin{tz}
                \node (1) [Cup, top] at (0,0) {};
                \node (2) [Cup, invisible] at (0,-1.5\cobheight) {};
                \node (3) [Cap, invisible] at (0,-1.5\cobheight) {};
                \selectpart[green]{(3) (2)}
            \end{tz}
            \longxdoubleto{\nu}
            \begin{tz}
                \node (1) [Cup, top] at (0,0) {};
                \node (2) [Cap, bot] at (0,-1.5\cobheight) {};
                \node (3) [Cup] at (0,-1.5\cobheight) {};
                \selectpart[green]{(1-center) (2-center)}
            \end{tz}
            \longxdoubleto{\mu}
            \begin{tz}
                \node[Cyl, top, bot, height scale=1.5] (X) at (0,0) {};
                \node[Cup] at (X.bottom) {};
            \end{tz}
            &
            \begin{tz}
                \node (1) [Cap, bot] at (0,-1.5\cobheight) {};
                \node (2) [Cup, invisible] at (0,0) {};
                \node (3) [Cap, invisible] at (0,0) {};
                \selectpart[green]{(3) (2)}
            \end{tz}
            \longxdoubleto{\nu}
            \begin{tz}
                \node (1) [Cap, bot] at (0,-1.5\cobheight) {};
                \node (2) [Cap, bot] at (0,0) {};
                \node (3) [Cup] at (0,0) {};
                \selectpart[green]{(1-center) (2-center)}
            \end{tz}
            \longxdoubleto{\mu}
            \begin{tz}
                \node[Cyl, bot, height scale=1.5] (X) at (0,0) {};
                \node[Cap, bot] at (X.top) {};
            \end{tz}
        \\ \nonumber
            \begin{tz}
                \node [Pants, top, bot] (A) at (0,0) {};
                \node[Cyl, bot, anchor=top, height scale=2] (B) at (A.leftleg) {};
                \node[Cyl, bot, anchor=top, height scale=2] (C) at (A.rightleg) {};
                \selectpart[green]{(A-leftleg) (A-rightleg) (B-bottom) (C-bottom)}
            \end{tz}
            \longxdoubleto{\eta}
            \begin{tz}
                \node [Pants, bot, top] (A) at (0,0) {};
                \node [Copants, bot, anchor=leftleg] (B) at (A.leftleg) {};
                \node [Pants, anchor=belt, bot] at (B.belt) {};
                \selectpart [green] {(A-leftleg) (A-belt) (A-rightleg) (B-belt)};
            \end{tz}
            \longxdoubleto{\epsilon}
            \begin{tz}
                \node [Pants, bot] (A) at (0,0) {};
                \node [Cyl, tall, bot, anchor=bot, top] at (A.belt) {};
            \end{tz}
            &
            \begin{tz}
                \node [Copants, bot] (A) at (0,0) {};
                \node[Cyl, anchor=bottom, height scale=2, top, bot] (B) at (A.leftleg) {};
                \node[Cyl, anchor=bottom, height scale=2, top, bot] (C) at (A.rightleg) {};
                \selectpart[green]{(C-bottom) (C-top) (B-bottom) (B-top)}
            \end{tz}
            \longxdoubleto{\eta}
            \begin{tz}
                \node [Copants, bot] (A) at (0,0) {};
                \node [Pants, anchor=leftleg, bot] (B) at (A.leftleg) {};
                \node [Copants, anchor=belt, top, bot] (C) at (B.belt) {};
                \selectpart [green]{(B-leftleg) (C-belt) (B-rightleg) (A-belt)};
            \end{tz}
            \longxdoubleto{\epsilon}
            \begin{tz}
                \node [Copants, top, bot] (A) at (0,0) {};
                \node [Cyl, tall, bot, anchor=top] at (A.belt) {};
            \end{tz}
        \end{calign}
        \item (Rigidity) $\phi_1^{-1},\phi_2^{-1}$ are inverses of the Frobeniusators, i.e. $\phi_1\phi_1^{-1} = \id$, $\phi_1^{-1}\phi_1 = \id$, $\phi_2\phi_2^{-1} = \id$, $\phi_2^{-1}\phi_2 = \id$ where $\phi_1,\phi_2$ are the composites given by
        \begin{align*}
            \phi_1 \quad&:=\quad
            \begin{tz}
                \node[Pants, top, bot] (A) at (0,0) {};
                \node[Cyl, bot, anchor=top] (B) at (A.leftleg) {};
                \node[Copants, bot, anchor=leftleg] (C) at (A.rightleg) {};
                \node[Cyl, top, bot, anchor=bottom] (D) at (C.rightleg) {}; 
                \selectpart[green, inner sep=1pt] {(A-belt) (D-top)};
            \end{tz}
            \longxdoubleto{\eta}
            \begin{tz}
                \node[Copants, top, wide, bot] (F) at (0,0) {};
                \node[Pants, bot, wide, anchor=belt] (G) at (F.belt) {};
                \node[Pants, bot, anchor=belt] (A) at (G.leftleg) {};
                \node[Cyl, bot, anchor=top] (B) at (A.leftleg) {};
                \node[Copants, bot, anchor=leftleg] (C) at (A.rightleg) {};
                \node[Cyl, bot, anchor=bottom] (X) at (C.rightleg) {}; 
                \selectpart[green] {(F-belt) (A-leftleg) (X-bottom)};
            \end{tz}
            \longxdoubleto{\alpha}
            \begin{tz}
                \node[Copants, top, wide, bot] (F) at (0,0) {};
                \node[Pants, bot, wide, anchor=belt] (G) at (F.belt) {};
                \node[Pants, bot, anchor=belt] (A) at (G.rightleg) {};
                \node[Cyl, tall, bot, anchor=top] (B) at (G.leftleg) {};
                \node[Copants, bot, anchor=leftleg] (C) at (A.leftleg) {};
                \selectpart[green] {(G-rightleg) (A-leftleg) (A-rightleg) (C-belt)};
            \end{tz}
            \longxdoubleto{\epsilon}
            \begin{tz}
                \node[Copants, top, bot] (A) at (0,0) {};
                \node[Pants, bot, anchor=belt] (B) at (A.belt) {};
            \end{tz}
        \\
            \phi_2 \quad&:=\quad
            \begin{tz}
                \node[Pants, top, bot] (A) at (0,0) {};
                \node[Cyl, bot, anchor=top] (B) at (A.rightleg) {};
                \node[Copants, bot, anchor=rightleg] (C) at (A.leftleg) {};
                \node[Cyl, top, bot, anchor=bottom] (D) at (C.leftleg) {}; 
                \selectpart[inner sep=1pt, green] {(D-top) (A-belt)};
            \end{tz}
            \longxdoubleto{\eta}
            \begin{tz}
                \node[Copants, top, wide, bot] (F) at (0,0) {};
                \node[Pants, bot, wide, anchor=belt] (G) at (F.belt) {};
                \node[Pants, bot, anchor=belt] (A) at (G.rightleg) {};
                \node[Cyl, bot, anchor=top] (B) at (A.rightleg) {};
                \node[Copants, bot, anchor=rightleg] (C) at (A.leftleg) {};
                \node[Cyl, bot, anchor=bottom] (X) at (C.leftleg) {}; 
                \selectpart[green] {(F-belt) (A-rightleg) (X-bottom)};
            \end{tz}
            \longxdoubleto{\alpha^{-1}}
            \begin{tz}
                \node[Copants, top, wide, bot] (F) at (0,0) {};
                \node[Pants, bot, wide, anchor=belt] (G) at (F.belt) {};
                \node[Pants, bot, anchor=belt] (A) at (G.leftleg) {};
                \node[Cyl, tall, bot, anchor=top] (B) at (G.rightleg) {};
                \node[Copants, bot, anchor=rightleg] (C) at (A.rightleg) {};
                \selectpart[green] {(G-leftleg) (A-leftleg) (A-rightleg) (C-belt)};
            \end{tz}
            \longxdoubleto{\epsilon}
            \begin{tz}
                \node[Copants, top, bot] (A) at (0,0) {};
                \node[Pants, bot, anchor=belt] (B) at (A.belt) {};
            \end{tz}
        \end{align*}
        \item (Ribbon) The twist $\theta$ satisfies
        \begin{equation*}
            \begin{tz}[xscale=1.4, yscale=2]
                \node (1) at (0,0)
                {
                $\begin{tikzpicture}
                        \node[Pants, bot] (A) at (0,0) {};
                        \node[Cap, bot] at (A.belt) {};
                        \selectpart[green, inner sep=1pt]{(A-leftleg)};
                \end{tikzpicture}$
                };
                \node (2) at (1,0)
                {
                $\begin{tikzpicture}
                        \node[Pants, bot] (A) at (0,0) {};
                        \node[Cap, bot] at (A.belt) {};
                \end{tikzpicture}$
                };
                \begin{scope}[double arrow scope]
                    \draw (1) -- node[above] {$\theta$} (2);
                \end{scope}
            \end{tz}
            \quad = \quad
            \begin{tz}[xscale=1.4, yscale=2]
                \node (1) at (0,0)
                {
                $\begin{tikzpicture}
                        \node[Pants, bot] (A) at (0,0) {};
                        \node[Cap, bot] at (A.belt) {};
                        \selectpart[green, inner sep=1pt]{(A-rightleg)};
                \end{tikzpicture}$
                };
                \node (2) at (1,0)
                {
                $\begin{tikzpicture}
                        \node[Pants, bot] (A) at (0,0) {};
                        \node[Cap, bot] at (A.belt) {};
                \end{tikzpicture}$
                };
                \begin{scope}[double arrow scope]
                    \draw (1) -- node[above] {$\theta$} (2);
                \end{scope}
            \end{tz} 
        \end{equation*}
    \end{itemize}
\end{defn}

The ribbon presentation is named as such as its linear representations (in the sense of Definition \ref{defn:rep}) correspond exactly to ribbon linear categories.

\begin{defn}
    The \emph{pivotal presentation} $\mathcal{P}$ is the 2-extension of the ribbon presentation $\mathcal{R}$ with 
    \begin{itemize} \smallbordisms
        \item Additional generating 2-morphisms:
        \begin{align*}
            \begin{tz} 
                \node[Pants, bot] (A) at (0,0) {};
                \node[Copants, top, bot, anchor=belt] at (A.belt) {};
            \end{tz}
            &\longxdoubleto{\eta ^\dag}
            \begin{tz} 
                \node[Cyl, tall, top, bot] (A) at (0,0) {};
                \node[Cyl, tall, top, bot] (B) at (2*\cobwidth, 0) {};
            \end{tz}
            &
            \begin{tz} 
                \node[Cyl, top, bot, tall] (A) at (0,0) {};
            \end{tz}
            &\longxdoubleto{\epsilon ^\dag}
            \begin{tz} 
                \node[Pants, top, bot] (A) at (0,0) {};
                \node[Copants, bot, anchor=leftleg] at (A.leftleg) {};
            \end{tz}
        \\
            \begin{tz}
                \node[Cap, bot] (A) at (0,0) {};
                \node[Cup] at (0,0) {};
            \end{tz}
            &\longxdoubleto{\nu ^\dag}{}
            \begin{tz}
                \draw[green] (0,0) rectangle (0.6, -0.6);  
            \end{tz}
            &
            \begin{tz}
                \node[Cyl, top, bot, tall] (A) at (0,0) {};
            \end{tz}
            &\longxdoubleto{\mu ^\dag}
            \begin{tz}
                \node[Cup, top] (A) at (0,0) {};
                \node[Cap, bot] (B) at (0,-2*\cobheight) {};
            \end{tz}
        \end{align*}
    \end{itemize}
    subject to the following additional conditions:
    \begin{itemize} \smallbordisms
        \item (Additional rigidity) $\phi_1^{-1},\phi_2^{-1}$ are equal to the following compositions:
        \begin{align*}
            \phi_1^{-1} \quad&=\quad
            \begin{tz}
                \node[Copants, top, bot] (A) at (0,0) {};
                \node[Pants, bot, anchor=belt] (B) at (A.belt) {};
                \selectpart[green, inner sep=1pt] {(B-rightleg)};
            \end{tz} 
            \longxdoubleto{\epsilon^\dagger}
            \begin{tz}
                \node[Copants, top, wide, bot] (F) at (0,0) {};
                \node[Pants, bot, wide, anchor=belt] (G) at (F.belt) {};
                \node[Pants, bot, anchor=belt] (A) at (G.rightleg) {};
                \node[Cyl, tall, bot, anchor=top] (B) at (G.leftleg) {};
                \node[Copants, bot, anchor=leftleg] (C) at (A.leftleg) {};
                \selectpart[green]{(F-belt) (G-leftleg) (A-rightleg)};
            \end{tz}
            \longxdoubleto{\alpha^{-1}}
            \begin{tz}
                \node[Copants, top, wide, bot] (F) at (0,0) {};
                \node[Pants, bot, wide, anchor=belt] (G) at (F.belt) {};
                \node[Pants, bot, anchor=belt] (A) at (G.leftleg) {};
                \node[Cyl, bot, anchor=top] (B) at (A.leftleg) {};
                \node[Copants, bot, anchor=leftleg] (C) at (A.rightleg) {};
                \node[Cyl, bot, anchor=bottom] (X) at (C.rightleg) {}; 
                \selectpart[green]{(F-leftleg) (F-rightleg) (G-leftleg) (G-rightleg)};
            \end{tz}
            \longxdoubleto{\eta^\dagger}
            \begin{tz}
                \node[Pants, top, bot] (A) at (0,0) {};
                \node[Cyl, bot, anchor=top] (B) at (A.leftleg) {};
                \node[Copants, bot, anchor=leftleg] (C) at (A.rightleg) {};
                \node[Cyl, top, bot, anchor=bottom] (D) at (C.rightleg) {}; 
            \end{tz}
        \\
            \phi_2^{-1} \quad&=\quad
            \begin{tz}
                \node[Copants, top, bot] (A) at (0,0) {};
                \node[Pants, bot, anchor=belt] (B) at (A.belt) {};
                \selectpart[green, inner sep=1pt] {(B-leftleg)};
            \end{tz} 
            \longxdoubleto{\epsilon^\dagger}
            \begin{tz}
                \node[Copants, top, wide, bot] (F) at (0,0) {};
                \node[Pants, bot, wide, anchor=belt] (G) at (F.belt) {};
                \node[Pants, bot, anchor=belt] (A) at (G.leftleg) {};
                \node[Cyl, tall, bot, anchor=top] (B) at (G.rightleg) {};
                \node[Copants, bot, anchor=rightleg] (C) at (A.rightleg) {};
                \selectpart[green]{(F-belt) (A-leftleg) (G-rightleg)};
            \end{tz}
            \longxdoubleto{\alpha}
            \begin{tz}
                \node[Copants, top, wide, bot] (F) at (0,0) {};
                \node[Pants, bot, wide, anchor=belt] (G) at (F.belt) {};
                \node[Pants, bot, anchor=belt] (A) at (G.rightleg) {};
                \node[Cyl, bot, anchor=top] (B) at (A.rightleg) {};
                \node[Copants, bot, anchor=rightleg] (C) at (A.leftleg) {};
                \node[Cyl, top, bot, anchor=bottom] (X) at (C.leftleg) {}; 
                \selectpart[green]{(F-leftleg) (F-rightleg) (G-leftleg) (G-rightleg)};
            \end{tz}
            \longxdoubleto{\eta^\dagger}
            \begin{tz}
                \node[Pants, top, bot] (A) at (0,0) {};
                \node[Cyl, bot, anchor=top] (B) at (A.rightleg) {};
                \node[Copants, bot, anchor=rightleg] (C) at (A.leftleg) {};
                \node[Cyl, top, bot, anchor=bottom] (D) at (C.leftleg) {}; 
            \end{tz}
        \end{align*}
        \item (Additional adjunction) $\eta^\dag,\epsilon^\dag$ witness $\tikztinycopants \dashv \tikztinypants$ while $\nu^\dag,\mu^\dag$ witness $\tikztinycap \dashv \tikztinycup$, i.e. the following are the identity:
        \begin{calign} \nonumber
            \begin{tz}
                \node[Cyl, top, bot, height scale=1.5] (X) at (0,0) {};
                \node[Cup] at (X.bottom) {};
                \selectpart[green]{(X-top) (X-bottom)}
            \end{tz}
            \longxdoubleto{\mu^\dag}
            \begin{tz}
                \node (1) [Cup, top] at (0,0) {};
                \node (2) [Cap, bot] at (0,-1.5\cobheight) {};
                \node (3) [Cup] at (0,-1.5\cobheight) {};
                \selectpart[green]{(3) (2)}
            \end{tz}
            \longxdoubleto{\nu^\dag}
            \begin{tz}
                \node (1) [Cup, top] at (0,0) {};
                \node (2) [Cup, invisible] at (0,-1.5\cobheight) {};
                \node (3) [Cap, invisible] at (0,-1.5\cobheight) {};
            \end{tz}
            &
            \begin{tz}
                \node[Cyl, bot, height scale=1.5] (X) at (0,0) {};
                \node[Cap, bot] at (X.top) {};
                \selectpart[green]{(X-top) (X-bottom)}
            \end{tz}
            \longxdoubleto{\mu^\dag}
            \begin{tz}
                    \node (1) [Cap, bot] at (0,-1.5\cobheight) {};
                    \node (2) [Cap, bot] at (0,0) {};
                    \node (3) [Cup] at (0,0) {};
                    \selectpart[green]{(3) (2)}
            \end{tz}
            \longxdoubleto{\nu^\dag}
            \begin{tz}
                    \node (1) [Cap, bot] at (0,-1.5\cobheight) {};
                    \node (2) [Cup, invisible] at (0,0) {};
                    \node (3) [Cap, invisible] at (0,0) {};
            \end{tz}
        \\ \nonumber
            \begin{tz}
                \node [Pants, bot] (A) at (0,0) {};
                \node (X) [Cyl, tall, bot, anchor=bot, top] at (A.belt) {};
                \selectpart[green]{(X-top) (X-bottom)}
            \end{tz}
            \longxdoubleto{\epsilon^\dag}
            \begin{tz}
                \node [Pants, bot, top] (A) at (0,0) {};
                \node [Copants, bot, anchor=leftleg] (B) at (A.leftleg) {};
                \node [Pants, anchor=belt, bot] (C) at (B.belt) {};
                \selectpart [green] {(A-leftleg) (A-rightleg) (C-leftleg)};
            \end{tz}
            \longxdoubleto{\eta^\dag}
            \begin{tz}
                \node [Pants, top, bot] (A) at (0,0) {};
                \node[Cyl, bot, anchor=top, height scale=2] (B) at (A.leftleg) {};
                \node[Cyl, bot, anchor=top, height scale=2] (C) at (A.rightleg) {};
            \end{tz}
            &
            \begin{tz}
                \node [Copants, top, bot] (A) at (0,0) {};
                \node (X) [Cyl, tall, bot, anchor=top] at (A.belt) {};
                \selectpart[green]{(A-belt) (X-bottom)}
            \end{tz}
            \longxdoubleto{\epsilon^\dag}
            \begin{tz}
                \node [Copants, bot] (A) at (0,0) {};
                \node [Pants, bot, anchor=leftleg] (B) at (A.leftleg) {};
                \node [Copants, anchor=belt, top, bot] (C) at (B.belt) {};
                \selectpart [green]{(C-leftleg) (C-rightleg) (B-rightleg) (B-leftleg)};
            \end{tz}
            \longxdoubleto{\eta^\dag}
            \begin{tz}
                \node [Copants, bot] (A) at (0,0) {};
                \node[Cyl, anchor=bottom, height scale=2, top, bot] (B) at (A.leftleg) {};
                \node[Cyl, anchor=bottom, height scale=2, top, bot] (C) at (A.rightleg) {};
            \end{tz}
        \end{calign}
        \item (Pivotality) The following equation as well as its rotation by $180^\circ$ about the $z$-axis hold:
        \begin{equation*}
            \begin{tz}
                \node[Cap] (A) at (0,0) {};
                \node[Cup] (B) at (0,0) {};
                \node[Cobordism Bottom End 3D] (AA) at (0,0) {};              
                \selectpart[green, inner sep=1pt] {(AA)};
            \end{tz}
            \longxdoubleto{\epsilon ^\dagger}
            \begin{tz}
                \node [Pants, bot] (A) at (0,0) {};
                \node [Cap, bot] at (A.belt) {};
                \node [Copants, bot, anchor=leftleg] (B) at (A.leftleg) {};
                \node [Cup] at (B.belt) {};
                \selectpart[green, inner sep=1pt] {(A-rightleg)};
            \end{tz}
            \longxdoubleto{\mu ^\dagger}
            \begin{tz}
                \node [Pants, bot] (A) at (0,0) {};
                \node [Cap, bot] at (A.belt) {};
                \node [Cyl, bot, height scale=1.5, anchor=top] (B) at (A.leftleg) {};
                \node [Cup] (C) at (A.rightleg) {};
                \node [Copants, bot, anchor=leftleg] (D) at (B.bottom) {};
                \node [Cap] (E) at (D.rightleg) {};
                \node [Cobordism Bottom End 3D] (B) at (D.rightleg) {};
                \node [Cup] at (D.belt) {};
                \selectpart[green] {(A-rightleg) (B)};
            \end{tz}
            \longxdoubleto{\mu}
            \begin{tz}
                \node [Pants, bot] (A) at (0,0) {};
                \node[Cap, bot] (X) at (A.belt) {};
                \node[Copants, bot, anchor=leftleg] (B) at (A.leftleg) {};
                \node[Cup] at (B.belt) {};
                \selectpart[green] {(X-center) (A-leftleg) (A-rightleg) (B-belt)};
            \end{tz}
            \longxdoubleto{\epsilon}
            \begin{tz}
                \node [Cap, bot] at (0,0) {};
                \node[Cup] at (0,0) {};
            \end{tz}
            \quad = \quad \id
        \end{equation*}
    \end{itemize}
\end{defn}
\begin{rk}
    The pivotal presentation $\mathcal{P}$ is referred to in \cite[Definition 3]{BDSV2} as the ``ribbon presentation''; we give it another name to avoid conflicting notation.
\end{rk}

From $\mathcal{P}$, adding one more relation yields a symmetric monoidal bicategory whose representations correspond to modular tensor categories with a choice of square root of the global dimension.

\begin{defn}[{\cite[Definition 3.9]{BDSV4}}] \label{defn:pres-m}
    The \emph{modular presentation} $\mathcal{M}$ is the pivotal presentation $\mathcal{P}$ with the extra relation:
    \begin{itemize} \smallbordisms
        \item (Modularity) The following equation as well as its rotation by $180^\circ$ about the $z$-axis hold:
        \begin{equation*}
            \begin{tz}[xscale=2, yscale=1.5]
                \node (1) at (0,-0.5)
                {$\begin{tikzpicture}
                    \node[Cyl, top, bot, tall] (A) at (0,0) {};
                \end{tikzpicture}$};

                \node (2) at (1,0)
                {$\begin{tikzpicture}
                    \node[Pants, top, bot] (A) at (0,0) {};
                    \node[Copants, bot, anchor=leftleg] (B) at (A.leftleg) {};
                    \selectpart[green, inner sep=1pt]{(A-leftleg)};
                    \selectpart[red, inner sep=1pt] {(A-rightleg)};
                \end{tikzpicture}$};

                \node (3) at (2,0)
                {$\begin{tikzpicture}
                    \node[Pants, top, bot] (A) at (0,0) {};
                    \node[Copants, bot, anchor=leftleg] (B) at (A.leftleg) {};
                \end{tikzpicture}$};

                \node (4) at (3,-0.5)
                {$\begin{tikzpicture}
                    \node[Cyl, top, bot, tall] (A) at (0,0) {};
                \end{tikzpicture}$};

                \node(5) at (1.5, -1)
                {$\begin{tikzpicture}
                    \node[Cup, top, bot] (A) at (0,0) {};
                    \node[Cap, bot] (B) at (0, -1.5*\cobheight) {};
                \end{tikzpicture}$};
                    
                \begin{scope}[double arrow scope]
                    \draw (1) -- node [above] {$\epsilon^\dagger$} (2);
                    \draw[] ([xshift=0pt] 2.0) to node [above, inner sep=1pt] {${\color{green}\theta}, \color{red}{\theta^{-1}}$} (3);
                    \draw (3) -- node [above] {$\epsilon$} (4);
                    \draw (1) -- node [below left] {$\mu^\dagger$} (5);
                    \draw (5) -- node [below right] {$\mu$} (4);
                \end{scope}
            \end{tz}
        \end{equation*}
    \end{itemize}
\end{defn}

This is a presentation of the componentwise variant of the signature bordism bicategory, which is defined later in Definition \ref{defn:bordcsig}.

\begin{thm}[{\cite{pres-bord,BDSV3}}]\label{conj:pres-csig}
    There exists an equivalence of symmetric monoidal bicategories $|-|_\mathcal{M}: \mathbf{F}(\mathcal{M}) \rightarrow \Bord{csig}$ between the symmetric monoidal bicategory generated by the modular presentation and the componentwise signature bordism bicategory.
\end{thm}
\begin{rk}
    This is stated as \cite[Theorem 3.10]{BDSV4}, but the proof of this is forthcoming. It is shown in \cite{BDSV3} that this follows from Theorem \ref{conj:pres-or}, which is to be the main result of \cite{pres-bord}. Alternatively, one may also deduce this from Theorem \ref{conj:pres-or} via similar arguments to those in \textsection\ref{subsection:constr-modg}.
\end{rk}

\begin{defn}[{\cite[Definition 3.13]{BDSV4}}] \label{defn:pres-n}
    The \emph{global modular presentation} $\mathcal{N}$ is a 2-extension of the modular presentation $\mathcal{M}$ with
    \begin{itemize} \smallbordisms
        \item Additional invertible generating 2-morphism:
        \begin{equation*}
            \begin{tz}
                \draw[green] (0,0) rectangle (\cobwidth, \cobwidth);
            \end{tz}
            \rightleftdoublearrow{\xi}{\xi^{-1}}
            \begin{tz}
                \draw[green] (0,0) rectangle (\cobwidth, \cobwidth);
            \end{tz}
        \end{equation*}
    \end{itemize}
    subject to 
    \begin{itemize} \smallbordisms
        \item (Inverses) $\xi\xi^{-1} = \id = \xi^{-1}\xi$.
        \item (Global anomaly) $\xi$ satisfies
        \begin{equation*}
            \begin{tz}
                \node[Cap, bot] (A) at (0,0) {};
                \node[Cup] (B) at (0,0) {};
                \draw[green] (1*\cobwidth, -\cupheight) rectangle +(\cobwidth, 2*\cupheight);
            \end{tz}
            \longxdoubleto{\xi}
            \begin{tz}
                \node[Cap, bot] (A) at (0,0) {};
                \node[Cup] at (0,0) {};
                \draw[green] (1*\cobwidth, -\cupheight) rectangle +(\cobwidth, 2*\cupheight);
            \end{tz}
            \quad = \quad
            \begin{tz}
                \node[Cap, bot] (A) at (0,0) {};
                \node[Cup] (B) at (0,0) {};
                \selectpart[green, inner sep=1pt] {(A-center)};
            \end{tz}
            \longxdoubleto{\epsilon ^\dag}
            \begin{tz}
                \node[Cap, bot] (A) at (0,0) {};
                \node[Pants, bot, anchor=belt] (B) at (A) {};
                \node[Copants, bot, anchor=leftleg] (C) at (B.leftleg) {};
                \node[Cup] (D) at (C.belt) {};
                \selectpart[green, inner sep=1pt] {(B-rightleg)};
            \end{tz}
            \longxdoubleto{\theta}
            \begin{tz}
                \node[Cap, bot] (A) at (0,0) {};
                \node[Pants, bot, anchor=belt] (B) at (A) {};
                \node[Copants, bot, anchor=leftleg] (C) at (B.leftleg) {};
                \node[Cup] (D) at (C.belt) {};
                \selectpart[green] {(B-rightleg) (B-leftleg) (A-center) (C-belt)};
            \end{tz}
            \longxdoubleto{\epsilon}
            \begin{tz}
                \node[Cap, bot] (A) at (0,0) {};
                \node[Cup] (B) at (0,0) {};
            \end{tz}
        \end{equation*}
    \end{itemize}
\end{defn}
\begin{rk}
    This presentation $\mathcal{N}$ is written as $\mathcal{N}_1$ in \cite{BDSV4}. 
\end{rk}

\begin{thm}[{\cite{pres-bord,BDSV3}}]\label{conj:pres-sig}
    There exists an equivalence of symmetric monoidal bicategories $|-|_\mathcal{N}: \mathbf{F}(\mathcal{N}) \rightarrow \Bord{sig}$ between the symmetric monoidal bicategory generated by the global modular presentation and the signature bordism bicategory.
\end{thm}
\begin{rk}
    This is stated as \cite[Theorem 3.14]{BDSV4}, but the proof of this is forthcoming. It is shown in \cite{BDSV3} that this follows from Theorem \ref{conj:pres-or}, which is to be the main result of \cite{pres-bord}. Alternatively, one may also deduce this from Theorem \ref{conj:pres-or} via similar arguments to those in \textsection\ref{subsection:constr-modg}.
\end{rk}

\begin{defn}[{\cite[Definition 3.11]{BDSV4}}] \label{defn:pres-o}
    The \emph{anomaly-free modular presentation} $\mathcal{O}$ is a 2-extension of $\mathcal{N}$ with the extra relation:
    \begin{itemize}
        \item (Anomaly-freeness) $\xi = \id$.
    \end{itemize}
\end{defn}

\begin{thm}[{\cite{BDSV2,haioun,filippos-thesis,pres-bord}}]\label{conj:pres-or}
    There exists an equivalence of symmetric monoidal bicategories $|-|_\mathcal{O}: \mathbf{F}(\mathcal{O}) \rightarrow \Bord{or}$ between the symmetric monoidal bicategory generated by the anomaly-free modular presentation and the oriented bordism bicategory.
\end{thm}
\begin{rk}
    This is stated as \cite[Theorem 3.12]{BDSV4}, and can be shown by combining the intended main result of \cite{BDSV1} with the arguments of \cite{BDSV2}. The former result will be proven in upcoming work of Bartlett, Douglas and Sytilidis \cite{pres-bord} by building upon the main results of \cite{haioun,filippos-thesis}. The other theorems of this flavour (Theorems \ref{conj:pres-csig}, \ref{conj:pres-sig}, \ref{conj:pres-csig/2} and \ref{conj:pres-sig/2}) can then be deduced from this theorem.
\end{rk}

Now, we construct modified versions of $\mathcal{M},\mathcal{N}$ which are better suited to the half-signature extension that we are working with.

\begin{defn} \label{defn:pres-mprime}
    The \emph{extended modular presentation} $\mathcal{M}'$ is a 2-extension of the pivotal presentation $\mathcal{P}$ with
    \begin{itemize} \smallbordisms
        \item Additional generating 2-morphism:
        \begin{equation*}
            \begin{tz}
                \node[Cyl, top, bot] (A) at (0,0) {};
            \end{tz}
            \rightleftdoublearrow{z}{z^{-1}}
            \begin{tz}
                \node[Cyl, top, bot] (A) at (0,0) {};
            \end{tz}
        \end{equation*}
    \end{itemize}
    subject to the following additional relations:
    \begin{itemize} \smallbordisms
        \item (Inverses) $zz^{-1} = \id = z^{-1}z$.
        \item (Centrality) $z$ satisfies 
        \begin{align*}
            \begin{tz}
                \node[Pants, top, bot] (A) at (0,0) {};
                \selectpart[green, inner sep=1pt] {(A-belt)};
            \end{tz}
            \longxdoubleto{z}
            \begin{tz}
                \node[Pants, top, bot] (A) at (0,0) {};
            \end{tz}
            \quad &= \quad
            \begin{tz}
                \node[Pants, top, bot] (A) at (0,0) {};
                \selectpart[green, inner sep=1pt] {(A-leftleg)};
            \end{tz}
            \longxdoubleto{z}
            \begin{tz}
                \node[Pants, top, bot] (A) at (0,0) {};
            \end{tz}
            \quad = \quad
            \begin{tz}
                \node[Pants, top, bot] (A) at (0,0) {};
                \selectpart[green, inner sep=1pt] {(A-rightleg)};
            \end{tz}
            \longxdoubleto{z}
            \begin{tz}
                \node[Pants, top, bot] (A) at (0,0) {};
            \end{tz}
        \\
            \begin{tz}
                \node[Copants, top, bot] (A) at (0,0) {};
                \selectpart[green, inner sep=1pt] {(A-belt)};
            \end{tz}
            \longxdoubleto{z}
            \begin{tz}
                \node[Copants, top, bot] (A) at (0,0) {};
            \end{tz}
            \quad &= \quad
            \begin{tz}
                \node[Copants, top, bot] (A) at (0,0) {};
                \selectpart[green, inner sep=1pt] {(A-leftleg)};
            \end{tz}
            \longxdoubleto{z}
            \begin{tz}
                \node[Copants, top, bot] (A) at (0,0) {};
            \end{tz}
            \quad = \quad
            \begin{tz}
                \node[Copants, top, bot] (A) at (0,0) {};
                \selectpart[green, inner sep=1pt] {(A-rightleg)};
            \end{tz}
            \longxdoubleto{z}
            \begin{tz}
                \node[Copants, top, bot] (A) at (0,0) {};
            \end{tz}
        \end{align*}
        \item (Extended modularity) The following equation as well as its rotation by $180^\circ$ about the $z$-axis hold:
        \begin{equation*}
            \begin{tz}[xscale=2, yscale=1.5]
                \node (1) at (0,-0.5)
                {$\begin{tikzpicture}
                    \node[Cyl, top, bot, tall] (A) at (0,0) {};
                \end{tikzpicture}$};

                \node (2) at (1,0)
                {$\begin{tikzpicture}
                    \node[Pants, top, bot] (A) at (0,0) {};
                    \node[Copants, bot, anchor=leftleg] (B) at (A.leftleg) {};
                    \selectpart[green, inner sep=1pt]{(A-leftleg)};
                    \selectpart[red, inner sep=1pt] {(A-rightleg)};
                \end{tikzpicture}$};

                \node (3) at (2,0)
                {$\begin{tikzpicture}
                    \node[Pants, top, bot] (A) at (0,0) {};
                    \node[Copants, bot, anchor=leftleg] (B) at (A.leftleg) {};
                \end{tikzpicture}$};

                \node (4) at (3,-0.5)
                {$\begin{tikzpicture}
                    \node[Cyl, top, bot, tall] (A) at (0,0) {};
                \end{tikzpicture}$};

                \node(5) at (2, -1)
                {$\begin{tikzpicture}
                    \node[Cup, top, bot] (A) at (0,0) {};
                    \node[Cap, bot] (B) at (0, -1.5*\cobheight) {};
                \end{tikzpicture}$};

                \node(6) at (1, -1)
                {$\begin{tikzpicture}
                    \node[Cyl, top, bot, tall] (A) at (0,0) {};
                \end{tikzpicture}$};
                    
                \begin{scope}[double arrow scope]
                    \draw (1) -- node [above] {$\epsilon^\dagger$} (2);
                    \draw[] ([xshift=0pt] 2.0) to node [above, inner sep=1pt] {${\color{green}\theta}, \color{red}{\theta^{-1}}$} (3);
                    \draw (3) -- node [above] {$\epsilon$} (4);
                    \draw (1) -- node [below left] {$z^{-1}$} (6);
                    \draw (5) -- node [below right] {$\mu$} (4);
                    \draw (6) -- node [below] {$\mu^\dag$} (5);
                \end{scope}
            \end{tz}
        \end{equation*}
    \end{itemize}
\end{defn}

As noted in \cite[Corollary 45]{BDSV2}, the centrality relation allows for $z$ to be moved around within each connected component, and hence $z$ can be commuted past all of the other generators of $\mathcal{M'}$. Thus, $z$ is central with respect to both horizontal and vertical composition.

The half-signature analogues to $\mathcal{M},\mathcal{N}$ are given as follows:
\begin{defn} \label{defn:pres-fg}
    The \emph{half-signature presentation} $\mathcal{H}$ is the extended modular presentation $\mathcal{M}'$ with one extra relation:
    \begin{itemize} \smallbordisms
        \item (Anomaly-freeness)
        \begin{equation*}
            \begin{tz}
                \node[Cap, bot] (A) at (0,0) {};
                \node[Cup] (B) at (0,0) {};
            \end{tz}
            \longxdoubleto{\id}
            \begin{tz}
                \node[Cap, bot] (A) at (0,0) {};
                \node[Cup] at (0,0) {};
            \end{tz}
            \quad = \quad
            \begin{tz}
                \node[Cap, bot] (A) at (0,0) {};
                \node[Cup] (B) at (0,0) {};
                \selectpart[green, inner sep=1pt] {(A-center)};
            \end{tz}
            \longxdoubleto{\epsilon ^\dag}
            \begin{tz}
                \node[Cap, bot] (A) at (0,0) {};
                \node[Pants, bot, anchor=belt] (B) at (A) {};
                \node[Copants, bot, anchor=leftleg] (C) at (B.leftleg) {};
                \node[Cup] (D) at (C.belt) {};
                \selectpart[green, inner sep=1pt] {(B-rightleg)};
            \end{tz}
            \longxdoubleto{\theta}
            \begin{tz}
                \node[Cap, bot] (A) at (0,0) {};
                \node[Pants, bot, anchor=belt] (B) at (A) {};
                \node[Copants, bot, anchor=leftleg] (C) at (B.leftleg) {};
                \node[Cup] (D) at (C.belt) {};
                \selectpart[green] {(B-rightleg) (B-leftleg) (A-center) (C-belt)};
            \end{tz}
            \longxdoubleto{\epsilon}
            \begin{tz}
                \node[Cap, bot] (A) at (0,0) {};
                \node[Cup] (B) at (0,0) {};
            \end{tz}
        \end{equation*}
    \end{itemize}
\end{defn}

This is a presentation of the componentwise variant of the half-signature bordism bi\-category defined in Definition \ref{defn:bordcsig}.

\begin{thm}\label{conj:pres-csig/2}
    There exists an equivalence of symmetric monoidal bicategories $|-|_\mathcal{H}: \mathbf{F}(\mathcal{H}) \rightarrow \Bord{csig/2}$ between the symmetric monoidal bicategory generated by the half-signature presentation and the componentwise half-signature bordism bicategory.
\end{thm}
\begin{rk}
    The proof of this theorem is in \textsection\ref{subsection:constr-modg}, where we construct the symmetric monoidal functor $|-|_\mathcal{H}$ and deduce from Theorem \ref{conj:pres-or} that $|-|_\mathcal{H}$ is an equivalence of symmetric monoidal bicategories.
\end{rk}

\begin{defn}\label{defn:pres-g}
    The \emph{global half-signature presentation} $\mathcal{G}$ is the 2-extension of $\mathcal{H}$ with 
    \begin{itemize} \smallbordisms
        \item Additional invertible generating 2-morphism:
        \begin{equation*}
            \begin{tz}
                \draw[green] (0,0) rectangle (\cobwidth, \cobwidth);
            \end{tz}
            \rightleftdoublearrow{\zeta}{\zeta^{-1}}
            \begin{tz}
                \draw[green] (0,0) rectangle (\cobwidth, \cobwidth);
            \end{tz}
        \end{equation*}
    \end{itemize}
    subject to 
    \begin{itemize} \smallbordisms
        \item (Inverses) $\zeta\zeta^{-1} = \id = \zeta^{-1}\zeta$.
        \item (Global $z$) $\zeta$ satisfies 
        \begin{equation*}
            \begin{tz}
                \node[Cyl, top, bot] (A) at (0,0) {};
                \draw[green] (1*\cobwidth, -\cupheight) rectangle +(\cobwidth, 2*\cupheight);
            \end{tz}
            \longxdoubleto{\zeta}
            \begin{tz}
                \node[Cyl, top, bot] (A) at (0,0) {};
                \draw[green] (1*\cobwidth, -\cupheight) rectangle +(\cobwidth, 2*\cupheight);
            \end{tz}
            \quad = \quad
            \begin{tz}
                \node[Cyl, top, bot] (A) at (0,0) {};
            \end{tz}
            \longxdoubleto{z}
            \begin{tz}
                \node[Cyl, top, bot] (A) at (0,0) {};
            \end{tz}
        \end{equation*}
    \end{itemize}
\end{defn}

\begin{thm}\label{conj:pres-sig/2}
    There exists an equivalence of symmetric monoidal bicategories $|-|_\mathcal{G}: \mathbf{F}(\mathcal{G}) \rightarrow \Bord{sig/2}$ between the symmetric monoidal bicategory generated by the global half-signature presentation and the half-signature bordism bicategory.
\end{thm}
\begin{rk}
    The proof of this theorem is in \textsection\ref{subsection:constr-modg}, where we construct the symmetric monoidal functor $|-|_\mathcal{G}$ and deduce from Theorem \ref{conj:pres-or} that $|-|_\mathcal{G}$ is an equivalence of symmetric monoidal bicategories.
\end{rk}

By construction, the symmetric monoidal bicategories presented by the various presentations defined above readily admit symmetric monoidal functors between them, resulting in the following commuting diagram:
\begin{equation} \label{eq:fs}
        \begin{tikzcd}
        &&& {\mathbf{F}(\mathcal{H})} & {\mathbf{F}(\mathcal{G})} & \\
        {\mathbf{F}(\mathcal{R})} & {\mathbf{F}(\mathcal{P})} & {\mathbf{F}(\mathcal{M}')} &&& {\mathbf{F}(\mathcal{O})} \\
        &&& {\mathbf{F}(\mathcal{M})} & {\mathbf{F}(\mathcal{N})}
        \arrow[from=1-4, to=1-5]
        \arrow[from=1-5, to=2-6]
        \arrow[from=2-1, to=2-2]
        \arrow[from=2-2, to=2-3]
        \arrow[from=2-2, to=3-4]
        \arrow[from=2-3, to=1-4]
        \arrow[from=2-3, to=3-4]
        \arrow[from=3-4, to=3-5]
        \arrow[from=3-5, to=2-6]
    \end{tikzcd}
\end{equation}
Most of these arise naturally from taking a 2-extension, i.e. sending each generator to its respective counterpart. We describe the remaining ones:
\begin{itemize}
    \item The functor $\mathbf{F}(\mathcal{M'}) \rightarrow \mathbf{F}(\mathcal{M})$ is given by sending the central element $z$ to the identity. Then, the extended modularity relation in $\mathcal{M'}$ reduces to the modularity relation in $\mathcal{M}$. Alternatively, we may think of $\mathcal{M}$ as a 2-extension of $\mathcal{M}'$ with the added relation that $z=\id$.
    \item The functor $\mathbf{F}(\mathcal{G}) \rightarrow \mathbf{F}(\mathcal{O})$ is given by sending $z$ and $\zeta$ to the identity. The anomaly-freeness relation in $\mathcal{G}$ then follows from anomaly-freeness in $\mathcal{O}$. Alternatively, we may think of $\mathcal{O}$ as a 2-extension of $\mathcal{G}$ with the added relations that $z=\id,\zeta=\id$.
\end{itemize}
In essence, downwards arrows correspond to strengthening the extended modularity relation to the modularity relation, while upwards arrows correspond to imposing anomaly-freeness.

Under the equivalences given by the Theorems \ref{conj:pres-csig}, \ref{conj:pres-sig}, \ref{conj:pres-or}, \ref{conj:pres-csig/2} and \ref{conj:pres-sig/2}, we may rewrite the rightmost five terms in (\ref{eq:fs}) as the following quotients:
\begin{equation*}
        \begin{tikzcd}
        {\Bord{csig/2}} & {\Bord{sig/2}} & \\
        && {\Bord{or}} \\
        {\Bord{csig}} & {\Bord{sig}}
        \arrow[from=1-1, to=1-2]
        \arrow[from=1-2, to=2-3]
        \arrow[from=3-1, to=3-2]
        \arrow[from=3-2, to=2-3]
    \end{tikzcd}
\end{equation*}
Here, the leftmost two arrows are given by summing the signature terms, and the rightmost two arrows are given by forgetting the signature terms.

Even though $\Bord{sig/2}$ admits inclusions to $\Bord{sig}$ as described in \textsection\ref{subsection:aut-bordsig}, none of these are compatible with the rest of the functors in (\ref{eq:fs}). The functor $|-|_\mathcal{N}$ defined in \cite{BDSV3} sends each generating 2-morphism to some bordism with signature term $0$. However, depending on the choice of inclusion of $\Bord{sig/2}$, one of the sets $\set{\eta,\epsilon,\nu,\mu}$ and $\set{\eta^\dag,\epsilon^\dag,\nu^\dag,\mu^\dag}$ will not be sent by $|-|_\mathcal{N}$ to 2-morphisms in the subextension $\Bord{sig/2}$. In our construction of $|-|_\mathcal{G}$ in Proposition \ref{prop:fg-sig/2}, the latter set will be sent to bordisms with signature term $\pm 1$.

\subsection{An equivalence $\mathbf{F}(\mathcal{G}) \rightarrow \Bord{sig/2}$} \label{subsection:constr-modg}

In this subsection, we shall deduce Theorem \ref{conj:pres-sig/2} from Theorem \ref{conj:pres-or} by constructing an equivalence of symmetric monoidal bicategories $|-|_\mathcal{G}: \mathbf{F}(\mathcal{G}) \rightarrow \Bord{sig/2}$. This will allow us to apply the results in the next section on the classification of linear representations of $\mathbf{F}(\mathcal{G})$ to obtain the classification of linear representations of $\Bord{sig/2}$.

The proof of Theorem \ref{conj:pres-csig/2} is very similar, and so we provide a brief sketch of this at the end of this subsection. Additionally, the methods here provide an alternative way to \cite{BDSV3} to deduce Theorems \ref{conj:pres-csig} and \ref{conj:pres-sig} from Theorem \ref{conj:pres-or}. We refer the reader to the remark following Construction \ref{constr:modg}.

Before defining $|-|_\mathcal{G}$, we first review the construction of the symmetric monoidal functor $|-|_\mathcal{O}: \mathbf{F}(\mathcal{O}) \rightarrow \Bord{or}$ described in \cite{BDSV2}. Recall that in order to define $|-|_\mathcal{O}$, it suffices to specify the images of each of the generating objects, 1-morphisms and 2-morphisms, and then to check that each of the relations are satisfied.
\begin{constr}[{\cite[Definition 3]{BDSV2}}] \label{constr:geom-o}
    the symmetric monoidal functor $|-|_\mathcal{O}: \mathbf{F}(\mathcal{O}) \rightarrow \Bord{or,exp}$ is defined as follows.

    At the object and 1-morphism level, the generators are sent to the corresponding object/1-morphism as suggested by the pictures.
    \begin{itemize}
        \item The generating object $\tikztinycirc$ is sent to the pair $(\mathbb{S},\mathbb{D})$ comprising the standard circle and the standard bounding disc.
        \item The generating 1-morphisms $\tikztinypants,\tikztinycopants,\tikztinycup,\tikztinycap$ are sent to $(\Sigma,H_\Sigma)$, where $\Sigma$ is the standard pants, copants, cup and cap respectively, and $H_\Sigma = D^3$, filling the ``insides'' of the depicted surfaces. Viewing these $D^3$s as being embedded in $\mathbb{R}^3$ in the obvious way, we may orient them according to the right-handed orientation on $\mathbb{R}^3$, which in turn induces an orientation on the respective surfaces.
    \end{itemize}
    At the level of 2-morphisms, the images of the generators are as given below.
    \begin{itemize}
        \item $\alpha,\lambda,\rho,\phi_1,\phi_2$ and their inverses are sent to the mapping cylinders of the obvious diffeomorphisms relative to the boundary as suggested by the pictured surfaces.
        \item $\theta$ is sent to a mapping cylinder of the right-handed Dehn twist, while $\beta$ is sent to the mapping cylinder of a left-handed half-twist; see \cite[(19)]{BDSV2}.
        \item $\nu$ implements to $0$-handle addition while its reverse $\nu^\dag$ implements $3$-handle addition.
        \item $\mu$ implements the $1$-handle addition linking the two components while its reverse $\mu^\dag$ implements $2$-handle addition along the equator of the cylinder.
        \item $\eta$ implements the $1$-handle addition linking the two components while its reverse $\eta^\dag$ implements the $2$-handle addition along the longitudinal circle separating the surface laterally.
        \item $\epsilon^\dag$ implements the $1$ handle addition linking the front and the back of the cylinder while its reverse $\epsilon$ implements $2$-handle addition along the longitude of the torus with two boundary components.
        \item $\xi$ is sent to the identity 2-morphism.
    \end{itemize}

    One may check that these assignments indeed satisfy the relations of $\mathcal{O}$:
    \begin{itemize}
        \item All relations involving only invertible generators (the monoidal, balanced and ribbon relations and the various relations encoding inverses) clearly hold at the level of diffeomorphisms, and so hold for their mapping cylinders.
        \item The (additional) adjunction and (additional) rigidity relations correspond to handle cancellation. Likewise, in the pivotality relation (and its $180^\circ$ rotation), $\mu^\dag\epsilon^\dag$ is by handle cancellation in fact the mapping cylinder of a diffeomorphism, and $\epsilon\mu$ is the mapping cylinder of the inverse diffeomorphism.
        \item The modularity and anomaly-freeness relations are addressed in \cite[Remarks 5 \& 8]{BDSV2}. The global anomaly relation follows from the anomaly-freeness relation.
    \end{itemize}
\end{constr}

We will now build upon this $|-|_\mathcal{O}$ to define a symmetric monoidal functor $|-|_\mathcal{G}: \mathbf{F}(\mathcal{G}) \rightarrow \Bord{sig/2}$. We shall first state the images of the generating object, 1-morphisms and 2-morphisms of $\mathcal{G}$, and then check that this assignment defines a symmetric monoidal functor. Here, recall that $\Bord{sig/2,out}$ has the same objects and 1-morphisms as $\Bord{or,exp}$, and its 2-morphisms are pairs consisting of a 2-morphism in $\Bord{or,exp}$ and a signature term $n\in\mathbb{Z}$ satisfying a parity condition as in Definition \ref{defn:sig/2}.

\begin{constr}\label{constr:modg}
    Define $|-|_\mathcal{G}: \mathbf{F}(\mathcal{G}) \rightarrow \Bord{sig/2,out}$ on the generating object, 1-morphisms and 2-morphisms as follows:
    \begin{itemize}
        \item Generating object and 1-morphisms: same as $|-|_\mathcal{O}$.
        \item Generating 2-morphisms: For each generating 2-morphisms $\gamma$ of the pivotal presentation $\mathcal{P}$, set $|\gamma|_\mathcal{G} := (|\gamma|_\mathcal{O},0)$. For the remaining generating 2-morphisms, set
            \begin{align*}
                |\eta^\dag|_\mathcal{G} := (|\eta^\dag|_\mathcal{O}, 1) &\quad\quad |\epsilon^\dag|_\mathcal{G} := (|\epsilon^\dag|_\mathcal{O}, -1) \\
                |\nu^\dag|_\mathcal{G} := (|\nu^\dag|_\mathcal{O}, -1) &\quad\quad |\mu^\dag|_\mathcal{G} := (|\mu^\dag|_\mathcal{O}, 1) \\
                |z|_\mathcal{G} := \pa{\id_{\tikztinycyl}, 2} &\quad\quad |\zeta| _\mathcal{G} := (\emptyset, 2).
            \end{align*}
    \end{itemize}
\end{constr}
\begin{rk}
    In \cite{BDSV3}, there is a similar functor $|-|_\mathcal{N}: \mathbf{F}(\mathcal{N})\rightarrow \Bord{sig}$ sending $\xi$ to $(\id_{\emptyset},1)$ and each other generator $\gamma$ to $(|\gamma|_\mathcal{O},0)$. This may be checked to satisfy the relations of $\mathcal{N}$ in a similar way to the proof of Proposition \ref{prop:fg-sig/2} below, which then allows us to deduce from Theorem \ref{conj:pres-or} that $|-|_\mathcal{N}$ is an equivalence of symmetric monoidal bicategories via the same argument as Proposition \ref{prop:g-equiv}. We may construct a functor $\mathbf{F}(\mathcal{G}) \rightarrow \mathbf{F}(\mathcal{N})$ compatible with $|-|_\mathcal{G}, |-|_\mathcal{N}$ and the inclusion $\Bord{sig/2,out} \rightarrow \Bord{sig}$ as follows: send each generator of $\mathcal{G}$ that comes from $\mathcal{P}$ to the corresponding generator in $\mathcal{N}$, and let $\eta^\dag\mapsto \xi\eta^\dag, \epsilon^\dag\mapsto \xi^{-1}\epsilon^\dag, \nu^\dag\mapsto \xi^{-1}\nu^\dag, \mu^\dag\mapsto \xi\mu^\dag$ and $z\mapsto \xi^2 \id_{\tikztinycyl},\zeta\mapsto \xi^2$. This, however, is \textit{not} compatible with the functors in (\ref{eq:fs}). Similarly, the inclusion corresponding to $\Bord{sig/2,in}\hookrightarrow \Bord{sig}$ would instead require the images of $\eta,\epsilon,\nu,\mu$ to be adjusted by a $\xi^\pm$ term, which is again incompatible with the functors in (\ref{eq:fs}).
\end{rk}

\begin{prop}\label{prop:fg-sig/2}
    Construction \ref{constr:modg} defines a symmetric monoidal functor $|-|_\mathcal{G}: \mathbf{F}(\mathcal{G}) \rightarrow \Bord{sig/2,out}$ such that the following diagram commutes:
    \[\begin{tikzcd}
        {\mathbf{F}(\mathcal{G})} && {\Bord{sig/2,out}} \\
        \\
        {\mathbf{F}(\mathcal{O})} && {\Bord{or,exp}}
        \arrow["{|-|_\mathcal{G}}", from=1-1, to=1-3]
        \arrow[from=1-1, to=3-1]
        \arrow["q", from=1-3, to=3-3]
        \arrow["{|-|_\mathcal{O}}", from=3-1, to=3-3]
    \end{tikzcd}\]
    Here, the left arrow is given by the relevant functor in (\ref{eq:fs}); it sends $z$ and $\zeta$ to identity 2-morphisms.
\end{prop}
\begin{proof}
    Note that $|-|_\mathcal{G}$ as given in Construction \ref{constr:modg} agrees with $|-|_\mathcal{O}$ on the generating object and 1-morphisms, and for each generator $\gamma$ of $\mathcal{G}$, $|\gamma|_\mathcal{G}$ is sent to a lift of $|\gamma|_\mathcal{O}$ in $\Bord{sig/2,out}$. It thus suffices to verify that this $|-|_\mathcal{G}$ indeed defines a symmetric monoidal functor, and then the compatibility with $|-|_\mathcal{O}$ will be automatic by construction.

    This will be checked over a series of lemmas. In Lemma \ref{lem:modg-parity}, we check that Construction \ref{constr:modg} indeed sends the generating 2-morphisms to 2-morphisms of $\Bord{sig/2,out}$, i.e. that the parity condition on the signature term $n$ is satisfied. Then, we shall check that this indeed satisfies the relations of $\mathcal{G}$. By the functoriality of $|-|_\mathcal{O}$, it suffices to check the signature term for each relation. This is done in Lemmas \ref{lem:modg-r}, \ref{lem:modg-p} and \ref{lem:modg-g}.
\end{proof}

\begin{lem} \label{lem:modg-parity}
    The assignments in Construction \ref{constr:modg} send the generating 2-morphisms to 2-morphisms in $\Bord{sig/2,out}$.
\end{lem}
\begin{proof}
    We check that the condition on the signature $n \equiv m(M) \pmod{2}$ is satisfied. Note that for every generator except for $\zeta$, the underlying bordism $M$ has $\ol{M} = S^3$.
    
    In the case of the generators in $\mathcal{P}$ as well as $z$, we also have $\ol{\parout M} = S^2$, and so $m(M)$ as defined in Definition \ref{defn:sig/2} is $0$.
    
    For the remaining non-invertible generators, $\ol{\parout M}$ is $S^2 \sqcup S^2$ in the case of $\eta^\dag$ and $\mu^\dag$, $S^1\times S^1$ in the case of $\epsilon^\dag$, and $\emptyset$ in the case of $\nu^\dag$. In all of these cases, $m(M)$ is odd.
    
    Finally, in the case of $\zeta$, we have $M = \emptyset$, and so $m(M) = 0$. Indeed, these values of $m(M)$ match the parities of the corresponding signature terms.
\end{proof}

\begin{lem} \label{lem:modg-r}
    The assignments in Construction \ref{constr:modg} satisfy the relations of $\mathcal{R}$.
\end{lem}
\begin{proof}
    For most of these relations, all intermediate surfaces have genus $0$, and so the $c\ssig$ term arising from Wall's invariant is $0$. This leaves us with the adjunction relations for $\eta,\epsilon$ and the rigidity relations.

    In both of the adjunction relations, note that $|\eta|_\mathcal{O} \setminus \mathring{H}_{\parout |\eta|_\mathcal{O}}$ defines the same Lagrangian subspace as $H_{\parout |\eta|_\mathcal{O}}$, and so by Lemma \ref{lem:coinc-zero}, we again have $c\ssig(|\epsilon|_\mathcal{O},|\eta|_\mathcal{O})=0$.
    
    For the rigidity relations, the same holds true for $\eta$ and the composites $\alpha^\pm \eta$.
\end{proof}
    
\begin{lem} \label{lem:modg-p}
    The assignments in Construction \ref{constr:modg} satisfy the additional relations of $\mathcal{P}$.
\end{lem}
\begin{proof}
    Firstly, the additional rigidity and adjunction relations are satisfied by the same reasoning as the versions without the daggers; see the proof of Lemma \ref{lem:modg-r}.
    
    For the pivotality relation, note that $c\ssig(|\epsilon|_\mathcal{O},|\mu|_\mathcal{O}) = 0$ by Lemma \ref{lem:coinc-zero} as the Lagrangian subspace defined by $|\mu|_\mathcal{O} \setminus \mathring{H}_{\parout |\mu|_\mathcal{O}}$ is the same as the one defined by $H_{\parout |\mu|_\mathcal{O}}$. Similarly, $c\ssig(|\mu^\dag|_\mathcal{O},|\epsilon^\dag|_\mathcal{O}) = 0$. Finally, $c\ssig(|\epsilon\mu|_\mathcal{O},|\mu^\dag\epsilon^\dag|_\mathcal{O}) = 0$ as the central surface has genus $0$. Hence, the overall signature term for the composition in the pivotality relation is $0$.
\end{proof}

\begin{lem} \label{lem:modg-g}
    The assignments in Construction \ref{constr:modg} satisfy the remaining relations of $\mathcal{G}$.
\end{lem}
\begin{proof}
    The centrality relation for $z$ is clearly satisfied, as is the global $z$ relation linking $\zeta$ and z. It remains to check that our assignment satisfies the extended modularity and anomaly-freeness relations (and their $180^\circ$ rotations about the $z$-axis).

    In the lower route of the extended modularity relation, there is no contribution from Wall's invariant as all the surfaces have genus $0$. Thus, the overall signature term is $-2+1+0 = -1$. To calculate the signature of the upper route, we may use the fact that $\tilde{c}\ssig = 0$ and take the horizontal composition of these with $\id_{\tikztinycup}$ and $\id_{\tikztinycap}$. By (\ref{eq:cocyc-2}) and (\ref{eq:middle-four}), the $c\ssig$ terms computed in this way remain the same. Then, the middle 2-morphism is now the identity 2-morphism, and thus the overall contribution from Wall's invariant is $0$ by Lemma \ref{lem:coinc-zero} as $\epsilon^\dag$ and $\epsilon$ both correspond to 2-handle addition along the same curve. The overall signature of the upper route is thus $-1+0+0=-1$, which agrees with the lower route. The argument for the $180^\circ$ rotation of the extended modularity relation about the $z$-axis is identical.

    Lastly, we check the anomaly-freeness relation. Note that the Dehn twist $\theta$ satisfies the conditions for Lemma \ref{lem:c-norm}. Thus, it suffices to determine $c\ssig(|\epsilon|_\mathcal{O},|\theta\epsilon^\dag|_\mathcal{O})$.
    Let $V,A,B,C$ be the vector spaces as defined in (\ref{eq:csig-def}). We pick a basis of $V$ given by the meridian (the curve along which $\theta$ is Dehn twisting), oriented in the right-handed way, and the longitude, oriented anticlockwise. Under this identification of $V$ with $\mathbb{R}^2$, the intersection form is given by $$\left\langle \binom{a}{b}, \binom{c}{d} \right\rangle = ad-bc.$$ Now, $A,B,C$ are spanned by $\binom01, \binom10, \binom11$ respectively; these satisfy the linear relation $$\binom01 + \binom10 + \binom{-1}{-1} = 0.$$ Then as $$\left\langle \binom01, \binom10 \right\rangle = -1 < 0,$$ we have $c\ssig(|\epsilon|_\mathcal{O},|\theta\epsilon^\dag|_\mathcal{O}) = -\sigma(V;A,B,C) = 1$.

    Overall, the total signature of this composition is $(-1+0+0) + 1 = 0$, and thus the anomaly-freeness relation is satisfied. The argument for the $180^\circ$ rotation of the anomaly-freeness relation about the $z$-axis is identical.
\end{proof}

Finally, we may use Proposition \ref{prop:fg-sig/2} to deduce Theorem \ref{conj:pres-sig/2} from Theorem \ref{conj:pres-or}. 

\begin{prop}\label{prop:g-equiv}
    $|-|_\mathcal{G}: \mathbf{F}(\mathcal{G})\rightarrow\Bord{sig/2,out}$ is an equivalence of symmetric monoidal bicategories.
\end{prop}
\begin{proof}
    It is clear that $|-|_\mathcal{G}$ is essentially surjective and essentially full. It remains to show that it is fully faithful, i.e. bijective on 2-morphisms.

    Let $f,g$ be any two 1-morphisms in $\mathbf{F}(\mathcal{G})$ with the same source and target. We abuse notation to let $f,g$ denote the corresponding 1-morphisms in $\mathbf{F}(\mathcal{O})$. Then as $|-|_\mathcal{O}$ is fully faithful, it induces a bijection between $\Hom_{\mathbf{F}(\mathcal{O})}(f,g)$ and $\Hom_{\Bord{or,exp}}(|f|_\mathcal{O},|g|_\mathcal{O})$.

    Recall that $q$ induces a function $\Hom_{\Bord{sig/2,out}}(|f|_\mathcal{O},|g|_\mathcal{O})\rightarrow \Hom_{\Bord{or,exp}}(|f|_\mathcal{O},|g|_\mathcal{O})$, sending a $\mathbb{Z}$-family of 2-morphisms in $\Bord{sig/2,out}$ to each 2-morphism in $\Bord{or}$. We want to obtain an analogous statement for the functor $\mathbf{F}(\mathcal{G})\rightarrow\mathbf{F}(\mathcal{O})$. To do this, we study the preimage of each 2-morphism of $\mathbf{F}(\mathcal{O})$ under this functor. Recall that this sends the central generators $z,\zeta$ to identity 2-morphisms and all other generators of $\mathcal{G}$ to their counterparts in $\mathcal{O}$. As $z$ and $\zeta$ are central and $z$ can be written in terms of $\zeta$ via the global $z$ relation, the preimage of each $\psi$ is exactly the set of all 2-morphisms of the form $\zeta^n \psi$, where $n$ is an integer. Note that $\zeta$ is torsion-free as $|\zeta|_\mathcal{G}$ is, so every preimage is a $\mathbb{Z}$-family of 2-morphisms.

    At the level of 2-morphisms, the diagram in Proposition \ref{prop:fg-sig/2} is now the following commuting diagram of sets:
    \[\begin{tikzcd}
        {\Hom_{\mathbf{F}(\mathcal{G})}(f,g)} && {\Hom_{\Bord{sig/2,out}}(|f|_\mathcal{G},|g|_\mathcal{G})} \\
        \\
        {\Hom_{\mathbf{F}(\mathcal{O})}(f,g)} && {\Hom_{\Bord{or,exp}}(|f|_\mathcal{O},|g|_\mathcal{O})}
        \arrow[from=1-1, to=1-3]
        \arrow[from=1-1, to=3-1]
        \arrow[from=1-3, to=3-3]
        \arrow[from=3-1, to=3-3]
    \end{tikzcd}\]
    The bottom arrow is a bijection of sets. For both vertical arrows, the preimage of each element is a $\mathbb{Z}$-torsor. It thus suffices to show that the top arrow is compatible with these $\mathbb{Z}$-actions. Indeed, $\mathbb{Z}$ acts by multiplication by $\zeta$ on the left and by shifting the signature by $2$ on the right, while by definition, we have $|\zeta|_\mathcal{G} = (\emptyset,2)$.
\end{proof}

This concludes the proof of Theorem \ref{conj:pres-sig/2}.

The proof of Theorem \ref{conj:pres-csig/2} (and similarly, Theorem \ref{conj:pres-sig}, mutandis mutatis) is analogous. The equivalence $|-|_\mathcal{H}$ is defined exactly as in Construction \ref{constr:modg} (but forgetting $\zeta$); checking that this defines a symmetric monoidal functor compatible with $|-|_\mathcal{O}$ is the same as the proof of Proposition \ref{prop:fg-sig/2}. We may then follow the argument of the proof of Proposition \ref{prop:g-equiv}; the only change is that the preimage of a 2-morphism with $c$ connected components is now a $\mathbb{Z}^c$-torsor. This acts by composing the relevant component with $z$ on one side and shifting the signature of the relevant component by $2$ on the other side; by construction, these are compatible under $|-|_\mathcal{H}$.

\section{Classifying representations of $\mathbf{F}(\mathcal{H})$ and $\mathbf{F}(\mathcal{G})$}
\label{section:rep-mtc}

The main results of \cite{BDSV4} include a classification of linear representations of $\mathbf{F}(\mathcal{M})$ and $\mathbf{F}(\mathcal{N})$ in terms of modular tensor categories and certain choices of square root. The aim of this subsection to prove an analogous result for $\mathbf{F}(\mathcal{H})$ and $\mathbf{F}(\mathcal{G})$; linear representations of these are classified in a similar way to those of $\mathbf{F}(\mathcal{M})$ and $\mathbf{F}(\mathcal{N})$, but without any choice of square root. Then, combining this with the results of the previous section, we will obtain the classification of linear representations of $\Bord{sig/2}$.

\subsection{Review of linear categorical concepts} Before we proceed, we first briefly review the relevant linear categorical concepts. This serves the purpose of fixing notation and terminology.

The main concept in this section is that of a linear representation of a symmetric monoidal bicategory.
\begin{defn}[{\cite[Definition 2.7]{BDSV4}}] \label{defn:rep}
    $\Vect$ is the symmetric monoidal bicategory with
    \begin{itemize}
        \item Objects: Cauchy-complete $k$-linear categories
        \item 1-morphisms: $k$-linear functors
        \item 2-morphisms: Natural transformations
    \end{itemize}
    The monoidal structure is given by the Cauchy completion of the enriched tensor product.

    A \emph{linear representation} of a symmetric monoidal bicategory $\mathbf{C}$ is a symmetric monoidal functor $\mathbf{C} \rightarrow \Vect$.
\end{defn}
\begin{rk}
    As laid out in \cite[Appendix A]{BDSV4}, there are other bicategorical analogues of $\vect$, but our results here apply for these as well.
\end{rk}

A related target bicategory is Kapranov--Voevodsky's \cite{KV} bicategory of 2-vector spaces $\KV$.
\begin{defn}
    $\KV$ is the symmetric monoidal bicategory with 
    \begin{itemize}
        \item Objects: positive integers.
        \item 1-morphisms: matrices of finite-dimensional $k$-vector spaces.
        \item 2-morphisms: matrices of $k$-linear maps.
    \end{itemize}
    Composition of 1-morphisms and horizontal composition of 2-morphisms is given by matrix multiplication, with $\oplus$ and $\otimes$ playing the role of the usual addition and multiplication. Vertical composition of 2-morphisms is given by the entrywise composition.

    This has a symmetric monoidal structure given by multiplication on objects, and on the 1- and 2-morphisms, the pairwise tensor products of entries.
\end{defn}

This is equivalent to the bicategory of finite semisimple linear categories, with the integer $s$ corresponding to the linear category $\vect^{\oplus s}$. Thus, $\KV$ admits a fully faithful functor into $\Vect$. We will see in the next subsection that the linear representations of bicategories in question all factor through $\KV$.

The aim of this section is to classify linear representations of $\mathbf{F}(\mathcal{H})$ and $\mathbf{F}(\mathcal{G})$ in terms of modular tensor categories. These are semisimple ribbon linear categories whose braiding satisfies a certain non-degeneracy condition. We provide brief definitions here, and refer the reader to \cite[\textsection 3.1]{partB} for more details.
\begin{defn}[{\cite[Definition 8.10.1]{EGNO}}]
    A \emph{ribbon category} $\mathcal{C}$ is a rigid balanced braided monoidal category whose twist is compatible with the duals.
\end{defn}
\begin{defn}\label{defn:ss-ribbon}
    A \emph{semisimple ribbon category} is a ribbon linear category which is linearly equivalent to $\vect^{\oplus s}$ for some $s$ and whose unit $\mathbbm{1}$ is simple.

    For such a category, pick representatives $\mathbbm{1}=X_1,\ldots,X_s$ of the isomorphism classes of simple objects, and let $\beta_{ij}:X_i\otimes X_j \rightarrow X_j\otimes X_i$ be the braiding for each $i,j$. The \emph{$S$-matrix} is the $s\times s$ matrix whose $(i,j)$-entry is given by $S_{ij} := \tr(\beta_{ij}\beta_{ji})$.

    A \emph{modular tensor category} is a semisimple ribbon category whose $S$-matrix is invertible.
\end{defn}

We now list some useful invariants of modular tensor categories. Again, pick representatives $\mathbbm{1}=X_1,\ldots,X_s$ of the isomorphism classes of simple objects.
\begin{defn}\label{defn:cat-consts}
    For a semisimple ribbon category, its \emph{global dimension} is given by $D := \sum_{i=1}^s \dim(X_s)^2$ and its \emph{Gauss sums} are $p_\pm := \sum_{i=1}^s \theta_{X_s}^{\pm 1}\dim(X_s)^2$ where $\theta_i$ is the twist of $X_i$. The \emph{anomaly} is the ratio $\frac{p_+}{p_-}$.
\end{defn}
Moreover, we record a useful relation between the global dimension and Gauss sums:
\begin{lem}[{\cite[Corollary 3.1.10]{BK}}]\label{lem:dpp}
    In a modular tensor category, we have $D = p_+p_-$.
\end{lem}
\begin{rk}
    In \cite{BDSV4}, this is instead taken to be the definition of global dimension.
\end{rk}

\subsection{Representations of $\mathbf{F}(\mathcal{H})$ and $\mathbf{F}(\mathcal{G})$}

We now present a modification of the arguments of \cite{BDSV4} which classifies the linear representations of $\mathbf{F}(\mathcal{H})$ and $\mathbf{F}(\mathcal{G})$ in terms of modular tensor category.

First, we note that \cite[\textsection 3]{BDSV2} and \cite[\textsection 4, \textsection 5.1-5.2]{BDSV4} do not use the modularity relation of $\mathcal{M}$, and so the results there apply as well to $\mathcal{P}$ and any of the presentations 2-extending it. We summarise the relevant results below.

Let $Z:\mathbf{F}(\mathcal{R}) \rightarrow \Vect$ be a symmetric monoidal functor. Then, by \cite[Lemma 3.4]{BDSV4}, the images of $\tikztinycup,\tikztinypants$ as well as $\alpha,\rho,\lambda,\beta,\theta$ under $Z$ endow $Z(\tikztinycirc)$ with the structure of a linear balanced braided monoidal category. Moreover, by \cite[Theorem 4.9]{BDSV4}, $Z(\tikztinycirc)$ is rigid with a compatible twist, and so is a ribbon linear category. Conversely, the data of this ribbon category $Z(\tikztinycirc)$ determines exactly the images under $Z$ of the generating object, 1-morphisms and 2-morphisms of $\mathcal{R}$. We refer the reader to \cite[Propositions 4.3-4.7]{BDSV4} for an explicit description of these images, and to \cite[Proposition 5.7]{partB} for an equivalent reformulation.

Now, let $Z:\mathbf{F}(\mathcal{P}) \rightarrow \Vect$ be a symmetric monoidal functor. The proofs in \cite[Appendix A]{BDSV4} apply for any 2-extension of $\mathcal{P}$, and so by \cite[Corollary A.23]{BDSV4}, $Z$ factors through $\KV$. (In particular, this would still hold true if we replaced $\Vect$ with one of the other bicategorical analogues of $\vect$ in \cite[Appendix A]{BDSV4}.) A consequence \cite[Lemma 5.3]{BDSV4} of this is that the ribbon category $Z(\tikztinycirc)$ is a finite direct sum of semisimple ribbon categories in the sense of Definition \ref{defn:ss-ribbon}. We may thus restrict ourselves to the case where $Z(\tikztinycirc)$ is linearly equivalent to $\vect^{\oplus s}$ for some $s$ and the unit of $Z(\tikztinycirc)$ is simple.

In this case, the images of the remaining generating 2-morphisms $\eta^\dag,\epsilon^\dag,\nu^\dag,\mu^\dag$ of $\mathcal{P}$ may be determined up to a nonzero scalar $p$.  We refer the reader to \cite[Propositions 5.4-5.6]{BDSV4} for explicit formulas for these images, and to \cite[Proposition 5.8]{partB} (with each instance of $p_+$ replaced by $p$) for an equivalent reformulation. In Lemma \ref{lem:z-value}, when considering symmetric monoidal functors $\mathbf{F}(\mathcal{H}) \rightarrow \Vect$, the anomaly-freeness relation will allow for us to solve for $p$.

Now, we proceed as in \cite{BDSV4}, but replacing the modularity relation in $\mathcal{M}$ with the extended modularity relation in $\mathcal{M}'$. Consider a functor $Z:\mathbf{F}(\mathcal{M}')\rightarrow \Vect$, so it factors through $\KV$. Then, let us consider the following composite maps:
\begin{defn}[{\cite[Definition 21]{BDSV2}}] \label{defn:dehn-twists} \smallbordisms
    Let $\II,A,\III$ denote the following compositions:
    \begin{align*} 
    \II \quad&:=\quad
    \begin{aligned}  \begin{tikzpicture}[xscale=2]
    \node (1) at (0,0)
    {
    $\begin{tikzpicture}
        \node[Copants, top, bot] (A) at (0,0) {};
        \node[Pants, bot, anchor=belt] (B) at (A.belt) {};
    \end{tikzpicture}$
    };
    \node (2) at (1,0)
    {
    $\begin{tikzpicture}
        \node[Pants, top, bot] (A) at (0,0) {};
        \node[Cyl, bot, anchor=top] (B) at (A.leftleg) {};
        \node[Copants, bot, anchor=leftleg] (C) at (A.rightleg) {};
        \node[Cyl, bot, anchor=bottom, top] (D) at (C.rightleg) {};
        \selectpart[green, inner sep=1pt] {(A-rightleg)};
    \end{tikzpicture}$
    };
    \node (3) at (2,0)
    {
    $\begin{tikzpicture}
        \node[Pants, bot, top] (A) at (0,0) {};
        \node[Cyl, bot, anchor=top] (B) at (A.leftleg) {};
        \node[Copants, bot, anchor=leftleg] (C) at (A.rightleg) {};
        \node[Cyl, bot, anchor=bottom, top] (D) at (C.rightleg) {};
    \end{tikzpicture}$
    };
    \node (4) at (3,0)
    {
    $\begin{tikzpicture}
        \node[Copants, top, bot] (A) at (0,0) {};
        \node[Pants, bot, anchor=belt] (B) at (A.belt) {};
    \end{tikzpicture}$
    };
    \begin{scope}[double arrow scope]
        \draw (1) --  node[above]{$\phi_1^{-1}$} (2);
        \draw (2) --  node[above]{$\theta$} (3);
        \draw (3) --  node[above]{$\phi_1$} (4);
    \end{scope} \end{tikzpicture} \end{aligned}
    \\
    A \quad&:=\quad
    \begin{tz}[xscale=2, yscale=2]
    \node (1) at (0,0)
    {
    $\begin{tikzpicture}
        \node[Pants, top, bot] (A) at (0,0) {};
        \node[Copants, bot, anchor=leftleg] (B) at (A.leftleg) {};
        \selectpart[green, inner sep=1pt] {(B-belt)};
    \end{tikzpicture}$
    };
    \node [inner sep=0pt] (2) at (1,0)
    {
    $\begin{tikzpicture}
        \node[Pants, top, bot] (A) at (0,0) {};
        \node[Copants, bot, anchor=leftleg] (B) at (A.leftleg) {};
        \node[Pants, bot, anchor=belt] (C) at (B.belt) {};
        \node[Copants, bot, anchor=leftleg] (D) at (C.leftleg) {};
        \selectpart[green] {(A-leftleg) (A-rightleg) (C-leftleg) (C-rightleg)};
    \end{tikzpicture}$
    };
    \node [inner sep=0pt] (3) at (2,0)
    {
    $\begin{tikzpicture}
        \node[Pants, top, bot] (A) at (0,0) {};
        \node[Copants, bot, anchor=leftleg] (B) at (A.leftleg) {};
        \node[Pants, bot, anchor=belt] (C) at (B.belt) {};
        \node[Copants, bot, anchor=leftleg] (D) at (C.leftleg) {};
        \selectpart[green] {(A-belt) (A-leftleg) (A-rightleg) (B-belt)};
    \end{tikzpicture}$
    };
    \node (4) at (3,0)
    {
    $\begin{tikzpicture}
        \node[Pants, top, bot] (A) at (0,0) {};
        \node[Copants, bot, anchor=leftleg] (B) at (A.leftleg) {};
    \end{tikzpicture}$
    };
    \begin{scope}[double arrow scope]
        \draw (1) -- node[above] {$\epsilon^\dagger$} (2);
        \draw (2) -- node[above] {$\II^{-1}$} (3);
        \draw (3) -- node[above] {$\epsilon$} (4);
    \end{scope}
    \end{tz}
    \\
    \III \quad&:=\quad
    \begin{tz}[xscale=2, yscale=2]
    \node (1) at (0,0)
    {
    $\begin{tikzpicture}
            \node[Pants, top, bot] (A) at (0,0) {};
            \node[Copants, bot, anchor=leftleg] (B) at (A.leftleg) {};
            \selectpart[green, inner sep=1pt] {(A-leftleg)};
    \end{tikzpicture}$
    };
    \node (2) at (1,0)
    {
    $\begin{tikzpicture}
            \node[Pants, top, bot] (A) at (0,0) {};
            \node[Copants, bot, anchor=leftleg] (B) at (A.leftleg) {};
    \end{tikzpicture}$
    };
    \node (3) at (2,0)
    {
    $\begin{tikzpicture}
            \node[Pants, top, bot] (A) at (0,0) {};
            \node[Copants, bot, anchor=leftleg] (B) at (A.leftleg) {};
            \selectpart[green, inner sep=1pt] {(A-leftleg)};
    \end{tikzpicture}$
    };
    \node (4) at (3,0)
    {
    $\begin{tikzpicture}
            \node[Pants, top, bot] (A) at (0,0) {};
            \node[Copants, bot, anchor=leftleg] (B) at (A.leftleg) {};
    \end{tikzpicture}$
    };
    \begin{scope}[double arrow scope]
        \draw (1) -- node[above] {$\theta$} (2);
        \draw (2) -- node[above] {$A$} (3);
        \draw (3) -- node[above] {$\theta$} (4);
    \end{scope}
    \end{tz}
    \end{align*}
\end{defn}
\begin{rk}
    Under $|-|_\mathcal{O}$, $\II$ and $A$ are sent to the mapping cylinders of certain Dehn twists; see \cite[Figure 1]{BDSV2} and the accompanying text.
\end{rk}

Note that $\II$ is clearly invertible, as it is the composition of invertible 2-morphisms. One may show that $A$, and hence $\III$, is invertible as well:
\begin{prop}
    The composites $A,\III$ are invertible.
\end{prop}
\begin{proof}
    Let $A^\dag$ be the composite given by replacing $\II^{-1}$ in the definition of $A$ with $\II$. In \cite[Proposition 25]{BDSV2}, it is shown that in $\mathbf{F}(\mathcal{M})$, we have $A\circ A^\dag = A^\dag\circ A = \id$. Following the same proof and replacing all instances of the modularity relation with the extended modularity relation, we instead obtain $A\circ A^\dag = A^\dag\circ A = z^{-1}$ in $\mathbf{F}(\mathcal{M}')$ (where $z^{-1}$ may be taken to act anywhere on the torus with two boundary components). Since $z$ is invertible, we nonetheless obtain that $A$ is invertible. 

    Then, $\III$ is invertible as it is the composition of $A$ with invertible 2-morphisms.
\end{proof}

By computing the action of $Z(\III)$ on a torus (viewed as the composition of a cup, copants, pants and cap), we may show that the braiding of $Z(\tikztinycirc)$ is non-degenerate, i.e. that it is a modular tensor category.

\begin{prop}[{\cite[Proposition 5.14]{BDSV4}}]
    $Z$ sends the torus to an $s$-dimensional vector space. On this vector space, $Z(\III)$ acts as $\frac1p$ times the $S$-matrix of $Z(\tikztinycirc)$.
\end{prop}
\begin{cor} \label{cor:mprime-mtc}
    Given a symmetric monoidal functor $Z: \mathbf{F}(\mathcal{M}') \rightarrow \Vect$ such that the unit of $Z(\tikztinycirc)$ is simple, the ribbon category $Z(\tikztinycirc)$ is a modular tensor category.
\end{cor}
\begin{proof}
    As $\III$ is invertible and $p$ is non-zero, the $S$-matrix of $Z(\tikztinycirc)$ is invertible.
\end{proof}

Now, we want to compute the value of $p$, which as of yet is still indeterminate. This will arise from working with $\mathcal{H}$ and using the extra anomaly-freeness relation. Before that, we shall first perform some computations in $\mathbf{F}(\mathcal{M}')$. In analogy to \cite[Definition 40]{BDSV2}, we let the \emph{anomaly} $x$ denote the composition 
\begin{equation*} \smallbordisms
    x\quad:=\quad
    \begin{tz}
        \node [Cyl, top, bot, tall] (A) at (0,0) {};
    \end{tz}
    \longxdoubleto{\epsilon ^\dag}
    \begin{tz}
            \node[Pants, top, bot] (A) at (0,0) {};
            \node[Copants, bot, anchor=leftleg] (B) at (A.leftleg) {};
            \selectpart[inner sep=1pt, green] {(A-rightleg)};
    \end{tz}
    \longxdoubleto{\theta}
    \begin{tz}
            \node[Pants, top, bot] (A) at (0,0) {};
            \node[Copants, bot, anchor=leftleg] (B) at (A.leftleg) {};
    \end{tz}
    \longxdoubleto{\epsilon}
    \begin{tz}
        \node [Cyl, top, bot, tall] (A) at (0,0) {};
    \end{tz}
\end{equation*}
in $\mathbf{F}(\mathcal{M}')$ (and correspondingly, its 2-extensions).

\begin{lem}
    The anomaly $x$ is central with respect to the horizontal and vertical composition of 2-morphisms.
\end{lem}
\begin{proof}
    We follow the proof of the corresponding result in \cite{BDSV2}. Note that the proof of \cite[Lemma 44]{BDSV2} only uses relations from $\mathcal{P}$, and so $x$ satisfies the centrality relation that $z$ satisfies in Definition \ref{defn:pres-mprime}. We conclude as in \cite[Corollary 45]{BDSV2}.
\end{proof}

Let $x^\dag$ denote $x$ but with the $\theta$ in the composition replaced by $\theta^{-1}$. In $\mathbf{F}(\mathcal{M})$, this was the inverse of $x$. In $\mathbf{F}(\mathcal{M}')$, however, this is now off by a $z$ term.

\begin{lem}
    In $\mathbf{F}(\mathcal{M}')$, $x^\dag x = z^{-1} = xx^\dag$.
\end{lem}
\begin{proof}
    We follow the proof of \cite[Lemma 46]{BDSV2}. The only change necessary is to apply the extended modularity relation instead of the modularity relation in the penultimate step, which introduces an extra factor of $z^{-1}$. We thus have $xx^\dag = z^{-1}$, and by the centrality of $x$, the other equality holds as well.
\end{proof}

On the other hand, there is an explicit formula for $Z(x),Z(x^\dag)$ given by directly computing the composition of the images of the relevant generating 2-morphisms whose composition defines $x,x^\dag$.
\begin{lem}[{\cite[Lemma 5.18]{BDSV4}}] \label{lem:x-action}
    Let $p_\pm$ denote the Gauss sums of the modular tensor category $Z(\tikztinycirc)$, as defined in Definition \ref{defn:cat-consts}. Let $p$ be the scalar on which the images of $\eta^\dag,\epsilon^\dag,\nu^\dag,\mu^\dag$ depend. Then, $Z(x)$ acts as multiplication by $\frac{p_+}p$ while $Z(x^\dag)$ acts as multiplication by $\frac{p_-}p$.
\end{lem}

Now, we work in $\mathbf{F}(\mathcal{H})$, so we can use the anomaly-freeness relation.
\begin{lem} \label{lem:z-value}
    Let $Z: \mathbf{F}(\mathcal{H}) \rightarrow \Vect$. Then we have $p=p_+$, and $Z(z)$ acts as multiplication by the anomaly $\frac{p_+}{p_-}$.
\end{lem}
\begin{proof}
    By the anomaly-freeness relation, we have $x = \id$, and so $x^\dag = z^{-1}$. Comparing these to Lemma \ref{lem:x-action}, we obtain the desired results.
\end{proof}
\begin{rk}
    Here we see the difference between $\mathcal{M}$ and $\mathcal{H}$: for the former, the analogous arguments in \cite[Lemma 5.19]{BDSV4} obtained a value of $p^2$ in terms of $p_\pm$, which determined $p$ up to a choice of sign. In contrast, we have determined $p$ exactly, without any sign ambiguity.
\end{rk}
\begin{rk}
    In this context, the ``anomaly-freeness'' relation is a bit of a misnomer; it is named as such because in going from $\mathcal{M}$ to $\mathcal{O}$ in \cite{BDSV4}, this relation forces the anomaly to be $1$. In our case, the anomaly-freeness relation merely fixes the value of $p$ and the image of $z$ under $Z$, but we nonetheless keep the name for consistency.
\end{rk}

Finally, we show that representations of $\mathbf{F}(\mathcal{H})$ indeed correspond to finite direct sums of modular tensor categories.
\begin{thm} \label{thm:rep-fh}
    Linear representations of $\mathbf{F}(\mathcal{H})$ are classified by finite direct sums of modular tensor categories.
\end{thm}
\begin{proof}
    Given some $Z: \mathbf{F}(\mathcal{H}) \rightarrow \Vect$, the ribbon category $Z(\tikztinycirc)$ is a finite direct sum of modular tensor categories, by Corollary \ref{cor:mprime-mtc}. As $\mathcal{H}$ extends $\mathcal{R}$, any equivalence $Z \simeq Z'$ would define an equivalence of ribbon linear categories $Z(\tikztinycirc) \simeq Z'(\tikztinycirc)$. Hence, the assignment $Z \mapsto Z(\tikztinycirc)$ from equivalence classes of representations to equivalence classes of finite direct sums of modular tensor categories is well-defined.

    Conversely, given a modular tensor category $\mathcal{C}$, we may define a symmetric monoidal functor $Z_{\mathcal{C}}: \mathbf{F}(\mathcal{H}) \rightarrow \KV \rightarrow \Vect$ as in \cite[Proposition 6.1]{BDSV4}: let $Z_\mathcal{C}$ send $\tikztinycirc$ to the object of $\KV$ given by the number of isomorphism classes of simple objects of $\mathcal{C}$ and 1-morphisms to matrices of spaces of labels of their associated ribbon graphs. Then, on the generating 2-morphisms, define $Z_\mathcal{C}$ as in \cite[Propositions 4.3-4.7]{BDSV4} and \cite[Propositions 5.4-5.6]{BDSV4}, with the scalar $p$ in the latter group being $p=p_+$. For the final generating 2-morphism $z$, let $Z_{\mathcal{C}}(z)$ be multiplication by $\frac{p_+}{p_-}$. It suffices to verify that these assignments satisfy the relations of $\mathcal{H}$. These are all directly addressed in the proof of \cite[Proposition 6.1]{BDSV4} except for the extended modularity relation. For this, modifying the computations for the modularity relation in the proof of \cite[Proposition 6.1]{BDSV4} along with an application of Lemma \ref{lem:dpp} gives us that the top path in the extended modularity relation corresponds to multiplication by $$\frac1{p_+} \delta_{i1} p_+p_- = \delta_{i1} p_-$$ while the bottom path corresponds to multiplication by $$\frac{p_-}{p_+} \delta_{i1} p_+ = \delta_{i1} p_-$$ and so the extended modularity relation indeed holds. In the general case, we may extend this construction by taking direct sums.

    Moreover, by following the proof of \cite[Proposition 6.2]{BDSV4}, a braided monoidal equivalence $\mathcal{C} \simeq \mathcal{C}'$ preserving the twist induces an equivalence of representations $Z_{\mathcal{C}} \simeq Z_{\mathcal{C}'}$. Thus, the assignment $\mathcal{C} \mapsto Z_{\mathcal{C}}$ is indeed a well-defined map from the set of equivalence classes of finite direct sums of modular tensor categories to the set of equivalence classes of representations of $\mathbf{F}(\mathcal{H})$.

    By construction, for any $Z: \mathbf{F}(\mathcal{H}) \rightarrow \Vect$, we have $Z \simeq Z_{Z(\tikztinycirc)}$. In particular, here we use the uniqueness of the value $p$ as proven in Lemma \ref{lem:z-value}. Conversely, for any modular tensor category $\mathcal{C}$, the ribbon category $Z_\mathcal{C}(\tikztinycirc)$ is exactly $\mathcal{C}$. Hence, the two maps above indeed define a bijection.
\end{proof}
\begin{rk}
    The proof of \cite[Theorem 1]{BDSV4} shows that representations of $\mathbf{F}(\mathcal{M})$ are classified by finite direct sums of modular categories with a choice of square root of the anomaly of each summand. By replacing $\mathcal{M}$ with $\mathcal{H}$, we have essentially removed the choice of square roots.
\end{rk}

Likewise, \cite[Theorem 3]{BDSV4} shows that representations of $\mathbf{F}(\mathcal{N})$ are classified by finite direct sums of modular tensor categories whose anomalies are equal along with a choice of square root of this anomaly. The following result for $\mathbf{F}(\mathcal{G})$ may be thought of as a version of this without the choice of square root.
\begin{thm} \label{thm:rep-fg}
    Linear representations of $\mathbf{F}(\mathcal{G})$ are classified by finite direct sums of modular tensor categories whose anomalies are equal.
\end{thm}
\begin{proof}
    The proof is analogous to that of \cite[Theorem 3]{BDSV4}. Given a linear representation of $\mathbf{F}(\mathcal{G})$, $\zeta$ must be sent to multiplication by some scalar $a$. The global $z$ relation along with Lemma \ref{lem:z-value} implies that $a = \frac{p_+}{p_-}$ for each direct summand of $Z(\tikztinycirc)$. Conversely, given any modular tensor category $\mathcal{C}$, sending $\zeta$ to $\frac{p_+}{p_-}$ extends the $Z_\mathcal{C}$ defined in Theorem~\ref{thm:rep-fh} to a symmetric monoidal functor $Z_\mathcal{C}:\mathbf{F}(\mathcal{G}) \rightarrow \Vect$.
\end{proof}
\begin{rk}
    The simple representations of $\mathbf{F}(\mathcal{H})$ and $\mathbf{F}(\mathcal{G})$ are the same: both correspond to modular tensor categories. The added condition for the representations of $\mathbf{F}(\mathcal{G})$ is merely a restriction of which simple representations we may take direct sums of.
\end{rk}

By combining Theorems \ref{thm:rep-fh} and \ref{thm:rep-fg} with Theorems \ref{conj:pres-csig/2} and \ref{conj:pres-sig/2} respectively, we obtain our main results.
\begin{thm}\label{thm:main-csig/2}
    Linear representations of $\Bord{csig/2}$ are classified by finite direct sums of modular tensor categories.
\end{thm} 
\begin{proof}
    This follows from Theorems \ref{conj:pres-csig/2} and \ref{thm:rep-fh}.
\end{proof}
\begin{thm} \label{thm:main-sig/2}
    Linear representations of $\Bord{sig/2}$ are classified by finite direct sums of modular tensor categories whose anomalies are equal.
\end{thm}
\begin{proof}
    This follows from Theorems \ref{conj:pres-sig/2} and \ref{thm:rep-fg}.
\end{proof}

\appendix

\section{Extensions of equivalent bicategories} \label{section:equiv-ext}

Throughout this paper, we have used various models for the oriented bordism bicategory and its extensions. It is, however, not clear a priori that equivalences of bicategories and compatible with the group of extensions constructed in Proposition \ref{prop:ext-gp}. The aim of this appendix is to show that the groups of extensions are invariant under equivalences of bicategories. In particular, Theorems \ref{thm:exists-sig/2} and \ref{thm:indiv} do not depend on the choice of model of $\Bord{or}$.

We will follow the notation of \cite[\textsection 4]{JY} throughout this appendix.

\subsection{Pullbacks of extensions}

Let $(F,F^2,F^0): \mathbf{C} \rightarrow \mathbf{D}$ be a pseudofunctor between bicategories, i.e. a lax functor for which all $F_{g,f}^2$ and $F_x^0$ are invertible. For each abelian group $A$, we shall define a group homomorphism $F^*:\Ext(\mathbf{D},A) \rightarrow \Ext(\mathbf{C},A)$ via a pullback construction.

\begin{defn} \label{defn:pb-ext}
    Let $\hatD$ be an extension of $\mathbf{D}$ by an abelian group $A$ and let $(F,F^2,F^0): \mathbf{C} \rightarrow \mathbf{D}$ be a pseudofunctor. The \emph{pullback} $\hatC = F^*\hatD$ is an extension of $\mathbf{C}$ by $A$ defined as follows: for each 2-morphism $\beta$ of $\mathbf{D}$, choose a lift $\widehat{\beta}$ to $\hatD$ such that each identity 2-morphism lifts to an identity 2-morphism. Then, let $\hatC$ be the bicategory with 
    \begin{itemize}
        \item Objects and 1-morphisms: same as $\mathbf{C}$.
        \item 2-morphisms: pairs $(\alpha,\beta)$ of 2-morphisms in $\mathbf{C},\hatD$ respectively such that $F\alpha = q\beta$.
    \end{itemize}
    Vertical composition of 2-morphisms is done entrywise, while the horizontal composition of $(\alpha,\beta)$ and $(\alpha',\beta')$ where $\alpha:f\rightarrow g$ and $\alpha':f'\rightarrow g'$ is given by $$\pa{\alpha'\star\alpha, \widehat{F_{a',b'}^2}(\beta'\star\beta)\pa{\widehat{F_{a,b}^2}}^{-1}}.$$

    Given composable 1-morphisms $f,g,h$, their associator in $\hatC$ is given by $$\pa{\alpha_{h,g,f},\widehat{F_{h,gf}^2}\pa{\id_{Fh}\star\widehat{F_{g,f}^2}}\mathcal{A}_{Fh,Fg,Ff}\pa{\widehat{F_{h,g}^2}\star\id_{Ff}}^{-1}\pa{\widehat{F_{hg,f}^2}}^{-1}}$$ where $\mathcal{A}_{Fh,Fg,Ff}$ denotes the associator of $Ff,Fg,Fh$ in $\hatD$.

    For each 1-morphism $f:x\rightarrow y$, its identity 2-morphism is given by $(\id_f,\id_{Ff})$ and its left and right unitors are respectively given by $$\pa{\lambda_f, \mathcal{L}_{Ff}\pa{\widehat{F_y^0}\star\id_{Ff}}^{-1}\pa{\widehat{F_{\id_y,f}^2}}^{-1}}\qquad \text{and} \qquad \pa{\rho_f,\mathcal{R}_{Ff}\pa{\id_{Ff}\star\widehat{F_x^0}}^{-1}\pa{\widehat{F_{f,\id_x}^2}}^{-1}}$$ where $\mathcal{L}_{Ff},\mathcal{R}_{Ff}$ denote the left and right unitors of $Ff$ in $\hatD$.
\end{defn}
One may verify the unity (triangle) and pentagon axioms directly: the terms arising from the $F^2,F^0$ cancel out, and so these hold by the unity and pentagon axioms on $\mathbf{C},\hatD$.

While this construction may a priori depend on the choice of lift of each 2-morphism of $\hatD$, one may show that a different choice of lift would yield an equivalent extension:
\begin{lem}
    Consider another choice of lift given by $a_\beta \cdot \widehat{\beta}$ (where $a_\beta\in A$) for each 2-morphism $\beta$. The pullback constructed in this way is equivalent to $\hatC$.
\end{lem}
\begin{proof}
    Let $\hatC'$ be the pullback constructed from this new choice of lift. We may construct an equivalence $(G,G^2,G^0):\hatC'\rightarrow\hatC$ directly. Indeed, for each pair of composable 1-morphisms $f,g$, let $$G_{f,g}^2 = a_{F_{g,f}^2} \cdot \id_{gf}$$ and for each object $x$, let $$G_x^0 = a_{F_x^0} \cdot \id_{\id_x}.$$ Then letting $G$ send each $(\alpha,\beta)$ to $(\alpha,\beta)$, we have a pseudofunctor $(G,G^2,G^0):\hatC'\rightarrow\hatC$. This is clearly has an inverse, and is thus an equivalence of extensions of $\mathbf{C}$.
\end{proof}

We shall now describe this construction in terms of cocycles. Given a lift $\widehat{\beta}$ of each 2-morphism $\beta$ of $\mathbf{D}$ to $\hatD$, we may obtain a cocycle $\mathbf{c} = (c,\tilde{c},a,\ell,r)$ for the extension $\hatD$. Correspondingly, there is a choice of lift of each 2-morphism $\alpha$ of $\mathbf{C}$ to the 2-morphism $\widehat{\alpha} := (\alpha,\widehat{F\alpha})$ of $\hatC$. We shall compute the cocycle $F^*\mathbf{c} = (F^*c,F^*\tilde{c},F^*a,F^*\ell,F^*r)$ for this choice of lifts.

Firstly, by definition of vertical composition in $\hatC$, we have
\begin{equation} \label{eq:pb-first}
    F^*c(\beta,\alpha) = c(F\beta,F\alpha).
\end{equation}
Next, from the definition of horizontal composition in $\hatC$, for horizontally composable 2-morphisms $\alpha,\beta$ in $\mathbf{C}$ we have $\widehat{F(\beta\star\alpha)}\widehat{F_{a,b}^2} = \widehat{F_{a',b'}^2}\pa{\widehat{F\beta}\star \widehat{F\alpha}}$, and so we have
\begin{equation}
    F^*\tilde{c}(\beta,\alpha) = \tilde{c}(F\beta,F\alpha) + c(F^2_{a',b'}, F\beta\star F\alpha) - c(F(\beta\star\alpha), F^2_{a,b}).
\end{equation}
where $\alpha:a\rightarrow a'$, $\beta: b\rightarrow b'$.

Simiarly, the definition of associators in $\hatC$ gives
\begin{equation}
    \begin{aligned}
        F^*a(h,g,f) &= a(Fh,Fg,Ff) + c(F^2_{h,gf}, \id_{Fh}\star F^2_{g,f}, \alpha_{Fh,Fg,Ff}) + \tilde{c}(\id_{Fh}, F^2_{g,f}) \\
        &\quad- c(F\alpha_{h,g,f}, F^2_{hg,f}, F^2_{h,g}\star\id_{Ff}) - \tilde{c}(F^2_{h,g},\id_{Ff})
    \end{aligned}
\end{equation}
for any composable 1-morphisms $f,g,h$ while the definitions of the left and right unitors give
\begin{equation}
    F^*\ell(f) = \ell(Ff) - c(F\lambda_f, F^2_{\id_y,f}, F^0_y\star\id_{Ff}) - \tilde{c}(F^0_y,\id_{Ff})
\end{equation}
and
\begin{equation} \label{eq:pb-last}
    F^*r(f) = r(Ff) - c(F\rho_f, F^2_{f,\id_x}, \id_{Ff}\star F^0_x) - c'(\id_{Ff}, F^0_x)
\end{equation}
for each 1-morphism $f:x\rightarrow y$.

One may easily check that a different choice of cocycle $\mathbf{c}$ (i.e. adding a coboundary) will result in a $F^*\mathbf{c}$ which differs from the original by a coboundary. In other words, pullbacks of equivalent extensions are equivalent, and so the pseudofunctor $(F,F^2,F^0):\mathbf{C}\rightarrow\mathbf{D}$ induces a homomorphism $F^*:\Ext(\mathbf{D},A) \rightarrow \Ext(\mathbf{C},A)$.

Moreover, if $(G, G^2, G^0): \mathbf{D} \rightarrow \mathbf{E}$ is also a pseudofunctor, then following the definition of the composite pseudofunctor $(GF,(GF)^2,(GF)^0)$ given in \cite[Definition 4.1.26]{JY}, one may check directly from Definition \ref{defn:pb-ext} that $(GF)^* = F^*G^*$.

\subsection{Strong transformations induce equivalences of extensions}

The computations here involve many diagrams of 2-morphisms which commute up to an action of an element $a$ of $A$. These will look like
\[\begin{tikzcd}
	\bullet & |[text=blue]| a & \bullet
	\arrow[curve={height=30pt}, from=1-1, to=1-3]
	\arrow[curve={height=-30pt}, from=1-1, to=1-3]
\end{tikzcd}\]
where each arrow may be replaced by a composition of arrows. The arrows represent 2-morphisms; the label $a$ of the face is the element of $A$ which, upon acting on the anti-clockwise route, yields the clockwise route. Faces which commute are left blank.

For example, the data of vertical composition may be written as
\[\begin{tikzcd}[cramped]
	\bullet && \bullet \\
	& |[text=blue]| {c(\beta,\alpha)} \\
	&& \bullet
	\arrow["{\widehat\alpha}", from=1-1, to=1-3]
	\arrow["{\widehat{\beta\alpha}}"', curve={height=30pt}, from=1-1, to=3-3]
	\arrow["{\widehat\beta}", from=1-3, to=3-3]
\end{tikzcd}\]

These may be combined in the obvious ways:
\begin{lem}\label{lem:merging}
    Suppose we have diagrams of the form:
    \[\begin{tikzcd}[cramped]
        \bullet && |[text=blue]| b && \bullet && \bullet & |[text=blue]| c & \bullet \\
        & \bullet & |[text=blue]| a & \bullet &&&& \bullet & |[text=blue]| d & \bullet
        \arrow[curve={height=-30pt}, from=1-1, to=1-5]
        \arrow[from=1-1, to=2-2]
        \arrow[curve={height=-30pt}, from=1-7, to=1-9]
        \arrow[from=1-7, to=2-8]
        \arrow[from=1-9, to=2-10]
        \arrow[curve={height=30pt}, from=2-2, to=2-4]
        \arrow[curve={height=-18pt}, from=2-2, to=2-4]
        \arrow[from=2-4, to=1-5]
        \arrow[from=2-8, to=1-9]
        \arrow[curve={height=30pt}, from=2-8, to=2-10]
    \end{tikzcd}\]
    Then, their outer faces may be written respectively as
    \[\begin{tikzcd}[cramped]
        \bullet && \bullet && \bullet && \bullet \\
        & |[text=blue]| {a+b} &&&& |[text=blue]| {c+d} \\
        \bullet && \bullet && \bullet && \bullet
        \arrow[from=1-1, to=1-3]
        \arrow[from=1-1, to=3-1]
        \arrow[from=1-5, to=1-7]
        \arrow[from=1-5, to=3-5]
        \arrow[from=1-7, to=3-7]
        \arrow[from=3-1, to=3-3]
        \arrow[from=3-3, to=1-3]
        \arrow[from=3-5, to=3-7]
    \end{tikzcd}\]
\end{lem}
\begin{proof}
    In both cases, compare the outer routes with the unique route through the middle edge.
\end{proof}

In essence, as long as two faces may be merged such that the new merged face consists of one single clockwise route and one single anticlockwise route, we may do so, adding the values of the faces. We will be using this method of merging faces to prove the following computational result:

\begin{prop}\label{prop:equiv-cocyc}
    Let $(F,F^2,F^0), (G,G^2,G^0)$ be pseudofunctors $\mathbf{C} \rightarrow \mathbf{D}$ such that there exists a strong transformation $\theta: F\rightarrow G$. Then $F^*, G^*: \Ext(\mathbf{D},A) \rightarrow \Ext(\mathbf{C},A)$ are equal as homomorphisms.
\end{prop}
\begin{proof}
    Let $\mathbf{c} = (c,\tilde{c},a,\ell,r)$ be a cocycle for $\mathbf{D}$. Recall that this corresponds to an extension $\hatD$ as well as a choice of section, i.e. some 2-morphism $\widehat{\alpha}$ in $\hatD$ for each 2-morphism $\alpha$ in $\mathbf{D}$.

    For a 2-morphism $\alpha: f \rightarrow g$, where $f,g: x \rightarrow y$, let $m(\alpha)$ be such that we have
    \[\begin{tikzcd}[cramped]
        {Gf\theta_x} && {Gg\theta_x} \\
        & |[text=blue]| {m(\alpha)} \\
        {\theta_yFf} && {\theta_yFg}
        \arrow["{\widehat{G\alpha}\star\id}", from=1-1, to=1-3]
        \arrow["{\theta_f}"', from=1-1, to=3-1]
        \arrow["{\theta_g}", from=1-3, to=3-3]
        \arrow["{\id\star\widehat{F\alpha}}"', from=3-1, to=3-3]
    \end{tikzcd}\]
    Explicitly, this is
    \begin{equation*}
        m(\alpha) = c(\theta_g, G\alpha \star \id_{\theta_x}) + \tilde{c}(G\alpha, \id_{\theta_x}) - c(\id_{\theta_y}\star F\alpha, \theta_f) - \tilde{c}(F\alpha, \id_{\theta_y}).
    \end{equation*}
    We may write down formulas for $n,k$ (defined below) in a similar way, but these are more cumbersome, and so are omitted.

    For 1-morphisms $f: x\rightarrow y$, $g: y\rightarrow z$, let $n(g,f)$ such that
    \[\begin{tikzcd}[cramped, row sep=large]
        && {G(gf)\theta_x} \\
        {(GgGf)\theta_x} && |[text=blue]| {n(g,f)} && {\theta_zF(gf)} \\
        {Gg(Gf\theta_x)} & {Gg(\theta_yFf)} & {(Gg\theta_y)Ff} & {(\theta_zFg)Ff} & {\theta_z(FgFf)}
        \arrow["{\id\star\widehat\theta_f}", from=3-1, to=3-2]
        \arrow["{\mathcal{A}^{-1}}", from=3-2, to=3-3]
        \arrow["{\widehat\theta_g\star\id}", from=3-3, to=3-4]
        \arrow["{\mathcal{A}}", from=3-4, to=3-5]
        \arrow["{\id\star\widehat{F^2_{g,f}}}", from=3-5, to=2-5]
        \arrow["{\mathcal{A}}", from=2-1, to=3-1]
        \arrow["{\widehat{G^2_{g,f}}\star\id}"', from=2-1, to=1-3]
        \arrow["{\widehat\theta_{gf}}"', from=1-3, to=2-5]
    \end{tikzcd}\]
    The $\mathcal{A}$ denote the relevant associators in $\hatD$. For example, the leftmost arrow is the lift $(\alpha_{Gg,Gf,\theta_x}, a(Gg,Gf,\theta_x))$ of the associator $\alpha_{Gg,Gf,\theta_x}$ in $\mathbf{D}$. The diagram commutes upon taking the quotient to $\mathbf{D}$, by lax naturality.

    Finally, for each object $x$, let $k(x)$ be such that
    \[\begin{tikzcd}[cramped,column sep=tiny, row sep=large]
        && {G\id_x\theta_x} \\
        {\id_{Gx}\theta_x} && |[text=blue]| {k(x)} && {\theta_xF\id_x} \\
        & {\theta_x} && {\theta_x\id_{Fx}}
        \arrow["{\theta_{\id_x}}", from=1-3, to=2-5]
        \arrow["{\widehat{G^0_x}\star\id}", from=2-1, to=1-3]
        \arrow["{\mathcal{L}_{\theta_x}}"', from=2-1, to=3-2]
        \arrow["{\mathcal{R}_{\theta_x}^{-1}}"', from=3-2, to=3-4]
        \arrow["{\id\star\widehat{F^0_x}}", from=3-4, to=2-5]
    \end{tikzcd}\]
    Here, $\mathcal{L}_{\theta_x}$ and $\mathcal{R}_{\theta_x}$ denote the left and right unitors of $\theta_x$ in $\hatD$. The diagram commutes upon taking the quotient to $\mathbf{D}$, by lax unity.

    We shall show that $F^*\mathbf{c} - G^*\mathbf{c}$ is equal to the coboundary of $(m,n,k)$. This is checked componentwise in Lemmas \ref{lem:equiv-c}, \ref{lem:equiv-cprime}, \ref{lem:equiv-a} and \ref{lem:equiv-lr}.
\end{proof}

\begin{lem}\label{lem:equiv-c}
    For vertically composable 2-morphisms $\alpha: f\rightarrow g$, $\beta: g\rightarrow h$ in $\hatC$, where $f,g,h:x \rightarrow y$, we have
    \begin{equation} \label{eq:equiv-c}
        c(F\beta, F\alpha) - c(G\beta, G\alpha) = m(\beta\alpha) - m(\alpha) - m(\beta).
    \end{equation}
\end{lem}
\begin{proof}
    Consider the following diagram in $\hatD$: 

    \[\begin{tikzcd}[cramped, column sep=tiny, row sep=small]
        {Gf\circ\theta_x} &&&&& {} & {\theta_y\circ Ff} \\
        &&& |[text=blue]| {-m(\alpha)} \\
        & |[text=blue]| {c(G\beta,G\alpha)} & {Gg\circ\theta_x} && {\theta_y\circ Fg} & |[text=blue]| {-c(F\beta,F\alpha)} \\
        &&& |[text=blue]| {-m(\beta)} \\
        {Gh\circ\theta_x} &&&&&& {\theta_y\circ Fh}
        \arrow["{\widehat{\theta}_f}", from=1-1, to=1-7]
        \arrow["{\widehat{G\alpha}\star\id}", from=1-1, to=3-3]
        \arrow["{\widehat{G(\beta\alpha)}\star\id}"', from=1-1, to=5-1]
        \arrow["{\id\star\widehat{F\alpha}}"', from=1-7, to=3-5]
        \arrow["{\id\star\widehat{F(\beta\alpha)}}", from=1-7, to=5-7]
        \arrow["{\widehat{\theta}_g}", from=3-3, to=3-5]
        \arrow["{\widehat{G\beta}\star\id}", from=3-3, to=5-1]
        \arrow["{\id\star\widehat{F\beta}}"', from=3-5, to=5-7]
        \arrow["{\widehat{\theta}_h}", from=5-1, to=5-7]
    \end{tikzcd}\]
    The left face is
    \begin{equation}\label{eq:first-middle-four}
        \tilde{c}(G\alpha,\id_{\theta_x}) + \tilde{c}(G\beta,\id_{\theta_x}) + c(G\beta\star\id_{\theta_x},G\alpha\star\id_{\theta_x}) - \tilde{c}(G(\beta\alpha),\id_{\theta_x}) = c(G\beta,G\alpha)
    \end{equation}
    by an application of (\ref{eq:middle-four}). A similar computation holds for the right face.

    The outer face contributes a term of $-m(\beta\alpha)$, so applying Lemma \ref{lem:merging} repeatedly (for example, starting from the top face and merging in a clockwise fashion), we have
    \begin{equation*}
        -m(\beta\alpha) = c(G\beta,G\alpha) - c(F\beta,F\alpha) - m(\alpha) - m(\beta)
    \end{equation*}
    which rearranges to (\ref{eq:equiv-c}).
\end{proof}

\begin{lem}\label{lem:equiv-cprime}
    For horizontally composable 2-morphisms $\alpha:f\rightarrow f', \beta:g\rightarrow g'$, where  $f,f':x\rightarrow y$ and $g,g':y\rightarrow z$, we have
    \begin{equation} \label{eq:equiv-cprime}
        F^*\tilde{c}(\beta,\alpha) - G^*\tilde{c}(\beta,\alpha) = m(\beta\star\alpha) - m(\alpha) - m(\beta) - n(g',f') + n(g,f). 
    \end{equation}
\end{lem}
\begin{proof}
    We have the following diagram in $\hatD$:

    \[\begin{tikzcd}[cramped,column sep=tiny, row sep=small]
        {(Gg'Gf')\theta_x} &&&&&& {Gg'(Gf'\theta_x)} \\
        \\
        & |[text=blue]| {-G^*\tilde{c}(\beta,\alpha)} & {(Gg Gf)\theta_x} && {Gg(Gf\theta_x)} & |[text=blue]| {m(\alpha)} \\
        {G(g'f')\theta_x} &&&&&& {Gg'(\theta_yFf')} \\
        && {G(gf)\theta_x} && {Gg(\theta_yFf)} \\
        & |[text=blue]| {-m(\beta\star\alpha)} && |[text=blue]| {-n(g,f)} \\
        && {\theta_zF(gf)} && {(Gg\theta_y)Ff} \\
        {\theta_zF(g'f')} &&&&&& {(Gg'\theta_y)Ff'} \\
        & |[text=blue]| {F^*\tilde{c}(\beta,\alpha)} & {\theta_z(FgFf)} && {(\theta_zFg)Ff} & |[text=blue]| {m(\beta)} \\
        \\
        {\theta_z(Fg'Ff')} &&&&&& {(\theta_zFg')Ff'}
        \arrow["{\mathcal{A}}", from=1-1, to=1-7]
        \arrow["{\widehat{G^2_{f',g'}}\star\id}"', from=1-1, to=4-1]
        \arrow["{\id\star\widehat\theta_f'}", from=1-7, to=4-7]
        \arrow["{(\widehat{G\beta}\star\widehat{G\alpha})\star\id}"', from=3-3, to=1-1]
        \arrow["{\mathcal{A}}", from=3-3, to=3-5]
        \arrow["{\widehat{G^2_{f,g}}\star\id}"', from=3-3, to=5-3]
        \arrow["{\widehat{G\beta}\star(\widehat{G\alpha}\star\id)}", from=3-5, to=1-7]
        \arrow["{\id\star\widehat\theta_f}", from=3-5, to=5-5]
        \arrow["{\widehat\theta_{g'f'}}"', from=4-1, to=8-1]
        \arrow["{\mathcal{A}^{-1}}", from=4-7, to=8-7]
        \arrow["{\widehat{G(\beta\star\alpha)}\star\id}", from=5-3, to=4-1]
        \arrow["{\widehat\theta_{gf}}"', from=5-3, to=7-3]
        \arrow["{\widehat{G\beta}\star(\id\star\widehat{F\alpha})}"', from=5-5, to=4-7]
        \arrow["{\mathcal{A}^{-1}}", from=5-5, to=7-5]
        \arrow["{\id\star\widehat{F(\beta\star\alpha)}}"', from=7-3, to=8-1]
        \arrow["{(\widehat{G\beta}\star\id)\star\widehat{F\alpha}}", from=7-5, to=8-7]
        \arrow["{\widehat\theta_g\star\id}", from=7-5, to=9-5]
        \arrow["{\widehat\theta_g\star\id}", from=8-7, to=11-7]
        \arrow["{\id\star\widehat{F^2_{f,g}}}", from=9-3, to=7-3]
        \arrow["{\id\star(\widehat{F\beta}\star\widehat{F\alpha})}", from=9-3, to=11-1]
        \arrow["{\mathcal{A}}", from=9-5, to=9-3]
        \arrow["{(\id\star\widehat{F\beta})\star\widehat{F\alpha}}"', from=9-5, to=11-7]
        \arrow["{\id\star\widehat{F^2_{f',g'}}}", from=11-1, to=8-1]
        \arrow["{\mathcal{A}}", from=11-7, to=11-1]
    \end{tikzcd}\]
    The squares whose faces are unlabelled commute in $\hatD$, by naturality of associators. The top-left and bottom-left squares require an application of (\ref{eq:middle-four}), similar to the derivation of (\ref{eq:first-middle-four}) in the previous lemma.

    We may apply Lemma \ref{lem:merging} repeatedly, starting from the top face and working clockwise to the bottom-left face, then merging the central face, and finally the remaining two faces on the left. As the outer rectangle contributes a $-n(g',f')$ term, we obtain 
    \begin{equation*}
        -n(g',f') = -n(g,f) + m(\alpha) + m(\beta) - m(\beta\star\alpha) + F^*\tilde{c}(\beta,\alpha) - G^*\tilde{c}(\beta,\alpha),
    \end{equation*}
    whence (\ref{eq:equiv-cprime}) follows.
\end{proof}

\begin{lem}\label{lem:equiv-a}
    Let $f:w\rightarrow x, g:x\rightarrow y, h:y\rightarrow z$ be 1-morphisms. Then,
    \begin{equation}\label{eq:equiv-a}
        F^*a(h,g,f) - G^*a(h,g,f) = n(h,g) + n(hg,f) - n(g,f) - n(h,gf) + m(\alpha_{h,g,f})
    \end{equation}
\end{lem}

\begin{proof}
    The two diagrams on the next page have the same outer border. Turning the first one inside-out (and reversing its face labels), we may combine these to obtain a larger diagram whose outer face is the top pentagon of the first diagram.
    
    \[\begin{tikzcd}[cramped,column sep=tiny,row sep=2.25em]
        &&& {((\theta_zFh)Fg)Ff} \\
        {(\theta_z(FhFg))Ff} && {\theta_z((FhFg)Ff)} && {\theta_z(Fh(FgFf))} && {(\theta_zFh)(FgFf)} \\
        {(\theta_zF(hg))Ff} && {\theta_z(F(hg)Ff)} & |[text=blue]| {F^*a(h,g,f)} & {\theta_z(FhF(gf))} && {(\theta_zFh)F(gf)} \\
        {(G(hg)\theta_x)Ff} && {\theta_zF((hg)f)} && {\theta_zF(h(gf))} && {(Gh\theta_y)F(gf)} \\
        \\
        {G(hg)(\theta_xFf)} && {G((hg)f)\theta_w} && {G(h(gf))\theta_w} && {Gh(\theta_yF(gf))} \\
        {G(hg)(Gf\theta_w)} && {(G(hg)Gf)\theta_w} & |[text=blue]| {-G^*a(h,g,f)} & {(GhG(gf))\theta_w} && {Gh(G(gf)\theta_w)} \\
        {(GhGg)(Gf\theta_w)} && {((GhGg)Gf)\theta_w} && {(Gh(GgGf))\theta_w} && {Gh((GgGf)\theta_w)} \\
        &&& {Gh(Gg(Gf\theta_w))}
        \arrow["{\mathcal{A}\star\id}"', from=1-4, to=2-1]
        \arrow["{\mathcal{A}}", from=1-4, to=2-7]
        \arrow["{\mathcal{A}}"', from=2-1, to=2-3]
        \arrow["{(\id\star\widehat{F^2_{h,g}})\star\id}", from=2-1, to=3-1]
        \arrow["{\id\star\mathcal{A}}", from=2-3, to=2-5]
        \arrow["{\id\star(\widehat{F^2_{h,g}}\star\id)}", from=2-3, to=3-3]
        \arrow["{\id\star(\id\star\widehat{F^2_{g,f}})}", from=2-5, to=3-5]
        \arrow["{\mathcal{A}}", from=2-7, to=2-5]
        \arrow["{\id\star\widehat{F^2_{g,f}}}", from=2-7, to=3-7]
        \arrow["{\mathcal{A}}"', from=3-1, to=3-3]
        \arrow["{\id\star\widehat{F^2_{hg,f}}}"', from=3-3, to=4-3]
        \arrow["{\id\star\widehat{F^2_{h,gf}}}", from=3-5, to=4-5]
        \arrow["{\mathcal{A}}", from=3-7, to=3-5]
        \arrow["{\widehat\theta_{hg}\star\id}", from=4-1, to=3-1]
        \arrow["{-n(hg,f)}"{description, style = {font = \normalsize}}, draw=none, from=4-1, to=6-3, color=blue]
        \arrow["{\id\star\widehat{F\alpha_{h,g,f}}}", from=4-3, to=4-5]
        \arrow["{-m(\alpha_{h,g,f})}"{description, style = {font = \normalsize}}, draw=none, from=4-3, to=6-5, color=blue]
        \arrow["{n(h,gf)}"{description, style = {font = \normalsize}}, draw=none, from=4-5, to=6-7, color=blue]
        \arrow["{\widehat\theta_h\star\id}"', from=4-7, to=3-7]
        \arrow["{\mathcal{A}^{-1}}", from=6-1, to=4-1]
        \arrow["{\widehat\theta_{hg,f}}"', from=6-3, to=4-3]
        \arrow["{\widehat{G\alpha_{h,g,f}}\star\id}"', from=6-3, to=6-5]
        \arrow["{\widehat\theta_{h,gf}}", from=6-5, to=4-5]
        \arrow["{\mathcal{A}^{-1}}"', from=6-7, to=4-7]
        \arrow["{\id\star\widehat\theta_f}", from=7-1, to=6-1]
        \arrow["{\widehat{G^2_{hg,f}}\star\id}", from=7-3, to=6-3]
        \arrow["{\mathcal{A}}"', from=7-3, to=7-1]
        \arrow["{\widehat{G^2_{h,gf}}\star\id}"', from=7-5, to=6-5]
        \arrow["{\mathcal{A}}", from=7-5, to=7-7]
        \arrow["{\id\star\widehat\theta_{gf}}"', from=7-7, to=6-7]
        \arrow["{\widehat{G^2_{h,g}}\star\id}", from=8-1, to=7-1]
        \arrow["{\mathcal{A}}"', from=8-1, to=9-4]
        \arrow["{(\widehat{G^2_{h,g}}\star\id)\star\id}", from=8-3, to=7-3]
        \arrow["{\mathcal{A}}"', from=8-3, to=8-1]
        \arrow["{\mathcal{A}\star\id}"', from=8-3, to=8-5]
        \arrow["{(\id\star\widehat{G^2_{g,f}})\star\id}"', from=8-5, to=7-5]
        \arrow["{\mathcal{A}}", from=8-5, to=8-7]
        \arrow["{\id\star(\widehat{G^2_{g,f}}\star\id)}"', from=8-7, to=7-7]
        \arrow["{\id\star\mathcal{A}}", from=8-7, to=9-4]
    \end{tikzcd}\]

    \[\begin{tikzcd}[cramped,column sep=tiny,row sep=2.25em]
        &&& {((\theta_zFh)Fg)Ff} \\
        {(\theta_z(FhFg))Ff} && |[text=blue]| {n(h,g)} && {((Gh\theta_y)Fg)Ff} && {(\theta_zFh)(FgFf)} \\
        {(\theta_zF(hg))Ff} && {(Gh(\theta_yFg))Ff} && {(Gh\theta_y)(FgFf)} && {(\theta_zFh)F(gf)} \\
        {(G(hg)\theta_x)Ff} && {(Gh(Gg\theta_x))Ff} && {Gh(\theta_y(FgFf))} && {(Gh\theta_y)F(gf)} \\
        {G(hg)(\theta_xFf)} && {((GhGg)\theta_x)Ff} && {Gh((\theta_yFg)Ff)} && {Gh(\theta_yF(gf))} \\
        {G(hg)(Gf\theta_w)} && {(GhGg)(\theta_xFf)} && {Gh((Gg\theta_x)Ff)} && {Gh(G(gf)\theta_w)} \\
        {(GhGg)(Gf\theta_w)} && {Gh(Gg(\theta_xFf))} && |[text=blue]| {-n(g,f)} && {Gh((GgGf)\theta_w)} \\
        &&& {Gh(Gg(Gf\theta_w))}
        \arrow["{\mathcal{A}\star\id}"', from=1-4, to=2-1]
        \arrow["{\mathcal{A}}", from=1-4, to=2-7]
        \arrow["{(\id\star\widehat{F^2_{h,g}})\star\id}", from=2-1, to=3-1]
        \arrow["{(\widehat\theta_h\star\id)\star\id}"{pos=0.8}, from=2-5, to=1-4]
        \arrow["{\mathcal{A}}"', from=2-5, to=3-5]
        \arrow["{\id\star\widehat{F^2_{g,f}}}", from=2-7, to=3-7]
        \arrow["{\mathcal{A}^{-1}\star\id}"', from=3-3, to=2-5]
        \arrow["{\mathcal{A}}", from=3-3, to=5-5]
        \arrow["{\widehat\theta_h\star\id}", from=3-5, to=2-7]
        \arrow["{\id\star\widehat{F^2_{g,f}}}", from=3-5, to=4-7]
        \arrow["{\widehat\theta_{hg}\star\id}", from=4-1, to=3-1]
        \arrow["{(\id\star\widehat\theta_g)\star\id}", from=4-3, to=3-3]
        \arrow["{\mathcal{A}}"', from=4-3, to=6-5]
        \arrow["{\mathcal{A}^{-1}}", from=4-5, to=3-5]
        \arrow["{\id\star(\id\star\widehat{F^2_{g,f}})}"{description}, from=4-5, to=5-7]
        \arrow["{\widehat{\theta}_h\star\id}"', from=4-7, to=3-7]
        \arrow["{\mathcal{A}^{-1}}", from=5-1, to=4-1]
        \arrow["{(\widehat{G^2_{h,g}}\star\id)\star\id}"{description}, from=5-3, to=4-1]
        \arrow["{\mathcal{A}\star\id}"', from=5-3, to=4-3]
        \arrow["{\id\star\mathcal{A}}", from=5-5, to=4-5]
        \arrow["{\mathcal{A}^{-1}}"', from=5-7, to=4-7]
        \arrow["{\id\star\widehat\theta_f}", from=6-1, to=5-1]
        \arrow["{\widehat{G^2_{h,g}}\star\id}", from=6-3, to=5-1]
        \arrow["{\mathcal{A}^{-1}}"', from=6-3, to=5-3]
        \arrow["{\mathcal{A}}", from=6-3, to=7-3]
        \arrow["{\id\star(\widehat\theta_g\star\id)}"', from=6-5, to=5-5]
        \arrow["{\id\star\widehat\theta_{gf}}"', from=6-7, to=5-7]
        \arrow["{\widehat{G^2_{h,g}}\star\id}", from=7-1, to=6-1]
        \arrow["{\id\star\widehat\theta_f}"', from=7-1, to=6-3]
        \arrow["{\mathcal{A}}"', from=7-1, to=8-4]
        \arrow["{\id\star\mathcal{A}^{-1}}", from=7-3, to=6-5]
        \arrow["{\id\star(\widehat{G^2_{g,f}}\star\id)}", from=7-7, to=6-7]
        \arrow["{\id\star\mathcal{A}}", from=7-7, to=8-4]
        \arrow["{\id\star(\id\star\widehat\theta_f)}"'{pos=0}, from=8-4, to=7-3]
    \end{tikzcd}\]
    In the second diagram, the $n(h,g)$ and $-n(g,f)$ faces require the use of (\ref{eq:middle-four}), as in (\ref{eq:first-middle-four}). We may repeatedly apply Lemma \ref{lem:merging} by merging the faces in the following order:
    \begin{itemize}
        \item In the first diagram, merge the bottom two faces of each of the three columns.
        \item Merge the three merged faces into one.
        \item Merge this to the octagonal face at the bottom-right of the second diagram.
        \item Merge this to the bottom pentagon of the first diagram.
        \item Merge this to the bottom-left quadrilateral of the second diagram, then merge repeatedly clockwise up to and including the octagonal face in the top-left.
        \item Merge this to the three central faces of the second diagram, from bottom to top.
        \item Merge this to the two remaining faces in the top-right of the second diagram.
        \item Merge this to the top-left and top-right squares of the first diagram, and then to the final hexagonal face.
    \end{itemize}
    Thus, we have merged all the faces except for the outer face (the top pentagon in the first diagram).

    Therefore, we conclude that
    \begin{equation*}
        F^*a(h,g,f) - G^*a(h,g,f) - n(hg,f) + n(h,gf) - m(\alpha_{h,g,f}) = n(h,g) - n(g,f),
    \end{equation*}
    which rearranges to (\ref{eq:equiv-a}).
\end{proof}

\begin{lem}\label{lem:equiv-lr}
    Let $f:x\rightarrow y$ be a 1-morphism in $\mathbf{C}$. Then
    \begin{equation}\label{eq:equiv-l}
        F^*\ell(f) - G^*\ell(f) = k(y) + n(\id_y,f) - m(\lambda_f)
    \end{equation}
    and
    \begin{equation}\label{eq:equiv-r}
        F^*r(f) - G^*r(f) = k(x) + n(f,\id_x) - m(\rho_f).
    \end{equation}
\end{lem}
\begin{proof}
    Consider the following diagram in $\hatD$:
    \[\begin{tikzcd}[row sep=huge]
        {(\theta_yF\id_y)Ff} & {(G\id_y\theta_y)Ff} & {G\id_y(\theta_yFf)} & {G\id_y(Gf\theta_x)} \\
        {(\theta_y\id_{Fy})Ff} & {(\id_{Gy}\theta_y)Ff} & {\id_{Gy}(\theta_yFf)} & {\id_{Gy}(Gf\theta_x)} \\
        {\theta_y(\id_{Fy}Ff)} & {\theta_yFf} & {Gf\theta_x} & {(\id_{Gy}Gf)\theta_x} \\
        {\theta_y(F\id_yFf)} & {\theta_yF(\id_yf)} & {G(\id_yf)\theta_x} & {(G\id_yGf)\theta_x}
        \arrow["{\mathcal{A}}"', shift right=7, curve={height=50pt}, from=1-1, to=4-1]
        \arrow["{\theta_{\id_y}\star\id}"', from=1-2, to=1-1]
        \arrow["{\mathcal{A}^{-1}}"', from=1-3, to=1-2]
        \arrow["{\id\star\theta_f}"', from=1-4, to=1-3]
        \arrow["{(\id\star\widehat{F^0_y})\star\id}"{description}, from=2-1, to=1-1]
        \arrow["{-k(y)}"{description, style = {font = \normalsize}}, draw=none, from=2-1, to=1-2, color=blue]
        \arrow["{\mathcal{A}^{-1}}", from=3-1, to=2-1]
        \arrow["{(\widehat{G^0_y}\star\id)\star\id}"{description}, from=2-2, to=1-2]
        \arrow["{\mathcal{L}_{\theta_y}\star\id}"{description}, from=2-2, to=3-2]
        \arrow["{\widehat{G^0_y}\star\id}"{description}, from=2-3, to=1-3]
        \arrow["{\mathcal{A}^{-1}}"', from=2-3, to=2-2]
        \arrow["{\mathcal{L}_{\theta_yFf}}", from=2-3, to=3-2]
        \arrow["{\widehat{G^0_y}\star\id}"', from=2-4, to=1-4]
        \arrow["{\id\star\theta_f}"', from=2-4, to=2-3]
        \arrow["{\mathcal{L}_{Gf\theta_x}}"', from=2-4, to=3-3]
        \arrow["{\id\star\mathcal{L}_{Ff}}", from=3-1, to=3-2]
        \arrow["{\id\star(\widehat{F^0_y}\star\id)}"{description}, from=3-1, to=4-1]
        \arrow["{\mathcal{R}^{-1}_{\theta_y}\star\id}"{description}, from=3-2, to=2-1]
        \arrow["{\theta_f}"', from=3-3, to=3-2]
        \arrow["{\mathcal{A}}"', from=3-4, to=2-4]
        \arrow["{\mathcal{L}_{Gf}\star\id}"', from=3-4, to=3-3]
        \arrow["{(\widehat{G^0_y}\star\id)\star\id}"{description}, from=3-4, to=4-4]
        \arrow["{F^*\ell(f)}"{description, style = {font = \normalsize}}, draw=none, from=4-1, to=3-2, color=blue]
        \arrow["{\id\star\widehat{F^2_{\id_y,f}}}"', from=4-1, to=4-2]
        \arrow["{\id\star\widehat{F\lambda_f}}"{description}, from=4-2, to=3-2]
        \arrow["{m(\lambda_f)}"{description, style = {font = \normalsize}}, draw=none, from=4-2, to=3-3, color=blue]
        \arrow["{\widehat{G\lambda_f}\star\id}"{description}, from=4-3, to=3-3]
        \arrow["{-G^*\ell(f)}"{description, style = {font = \normalsize}}, draw=none, from=4-3, to=3-4, color=blue]
        \arrow["{\theta_{\id_yf}}", from=4-3, to=4-2]
        \arrow["{\mathcal{A}}"', shift right=7, curve={height=50pt}, from=4-4, to=1-4]
        \arrow["{\widehat{G^2_{\id_y,f}}\star\id}", from=4-4, to=4-3]
    \end{tikzcd}\]
    The $-k(y)$, $F^*\ell(f)$ and $-G^*\ell(f)$ faces require (\ref{eq:middle-four}).

    Starting from the top-left face, merge the middle triangular face, then proceed clockwise and merge the remaining faces. Then noting that the outer face contributes a term of $n(\id_y,f)$, Lemma \ref{lem:merging} gives
    \begin{equation*}
        n(\id_y,f) = F^*\ell(f) - G^*\ell(f) + m(\lambda_f) - k(y),
    \end{equation*}
    which rearranges to (\ref{eq:equiv-l}).

    The computation for (\ref{eq:equiv-r}) is similar.
\end{proof}

This completes the proof of Proposition \ref{prop:equiv-cocyc}. Finally, we deduce:

\begin{thm}\label{thm:equiv-ext}
    Let $F: \mathbf{C} \rightarrow \mathbf{D}$ be an equivalence of bicategories. Then $F^*: \Ext(\mathbf{D}, A) \rightarrow \Ext(\mathbf{C}, A)$ is an isomorphism.
\end{thm}
\begin{proof}
    There exists $G: \mathbf{D} \rightarrow \mathbf{C}$ such that $FG \simeq \id$ and $GF \simeq \id$. Thus, by Proposition \ref{prop:equiv-cocyc}, $G^*$ and $F^*$ are inverses to each other.
\end{proof}

\printbibliography

@inbook{KV,
	title = {2-categories and Zamolodchikov tetrahedra equations},
	booktitle = {Algebraic groups and their generalizations: quantum and infinite-dimensional methods (University Park, PA, 1991)},
	series = {Proc. Sympos. Pure Math.},
	volume = {56},
	year = {1994},
	pages = {177{\textendash}259},
	publisher = {Amer. Math. Soc., Providence, RI},
	organization = {Amer. Math. Soc., Providence, RI},
	author = {Kapranov, M. M. and Voevodsky, V. A.}
    }

@article{MR,
    author = {Gregor Masbaum and Justin D. Roberts},
    journal = {Mathematische Annalen},
    number = {1},
    pages = {131--150},
    title = {On central extensions of mapping class groups},
    volume = {302},
    year = {1995},
    doi = {10.1007/BF01444490}
    }

@misc{csp-phd,
    title={The classification of two-dimensional extended topological field theories}, 
    author={Christopher J. Schommer-Pries},
    year={2009},
    eprint={1112.1000},
    archivePrefix={arXiv},
    primaryClass={math.AT},
    note={Revised 2014 version of PhD thesis, Department of Mathematics, University of California, Berkeley}
    }

@misc{csp-inv-tft,
    title={Invertible Topological Field Theories}, 
    author={Christopher Schommer-Pries},
    year={2017},
    eprint={1712.08029},
    archivePrefix={arXiv},
    primaryClass={math.AT},
    }

@article{Rohlin,
    author = {Vladimir A. Rohlin},
    title = {New results in the theory of four-dimensional manifolds},
    journal = {Doklady Akad. Nauk SSSR (N.S.)},
    volume = {84},
    year = {1952},
    pages = {221--224},
    }

@article{Harer,
    author = {John Harer},
    year = {1983},
    pages = {221--239},
    title = {The second homology group of the mapping class group of an orientable surface},
    volume = {72},
    journal = {Inventiones Mathematicae},
    doi = {10.1007/BF01389321}
    }

@article{korkstip,
    title={The second homology groups of mapping class groups of orientable surfaces},
    volume={134},
    DOI={10.1017/S0305004102006461},
    number={3},
    journal={Mathematical Proceedings of the Cambridge Philosophical Society},
    author={Mustafa Korkmaz and András I. Stipsicz},
    year={2003},
    pages={479--489}
    }

@article{Meyer,
    author = {Werner Meyer},
    journal = {Mathematische Annalen},
    pages = {239--264},
    title = {Die Signatur von Flächenbündeln},
    volume = {201},
    year = {1973},
    doi = {10.1007/BF01427946}
    }

@misc{Walker,
    author = {Kevin Walker},
    note = {Preliminary Version \#2},
    title = {On Witten's 3-manifold invariants},
    year = {1991},
    url = {https://canyon23.net/math/1991TQFTNotes.pdf}
    }

@article{Wall,
    author = {C. T. C. Wall},
    journal = {Inventiones Mathematicae},
    pages = {269--274},
    title = {Non-additivity of the signature.},
    volume = {7},
    year = {1969}
    }

@book{LV,
    title={The Weil representation, Maslov index and Theta series}, 
    author={Gérard Lion and Michèle Vergne},
    isbn={9780817630072},
    year={1980},
    publisher={Progress in Mathematics},
    volume={6}
    }

@article{CLM,
    author = {Cappell, Sylvain E. and Lee, Ronnie and Miller, Edward Y.},
    title = {On the Maslov index},
    journal = {Communications on Pure and Applied Mathematics},
    volume = {47},
    number = {2},
    pages = {121-186},
    doi = {10.1002/cpa.3160470202},
    year = {1994}
    }

@article{Gilmer,
    author = {Patrick M. Gilmer},
    year = {2001},
    title = {Integrality for TQFTs},
    volume = {125},
    journal = {Duke Mathematical Journal},
    doi = {10.1215/S0012-7094-04-12527-8}
    }

@article{GM,
    title = {Maslov index, Lagrangians, mapping class groups and TQFT},
    author = {Patrick M. Gilmer and Gregor Masbaum},
    pages = {1067--1106},
    volume = {25},
    number = {5},
    journal = {Forum Mathematicum},
    doi = {doi:10.1515/form.2011.143},
    year = {2013}
    }

@article{Atiyah-TQFT,
    author = {Michael F. Atiyah},
    journal = {Publications Mathématiques de l’Institut des Hautes Scientifiques},
    pages = {175--186},
    title = {Topological quantum field theories},
    volume = {68},
    year = {1988},
    doi = {doi:10.1007/BF02698547}
    }

@inproceedings{Segal,
    author = {Segal, Graeme},
    title = {The definition of conformal field theory},
    booktitle = {Symposium on Topology, Geometry and Quantum Field Theory (Segalfest)},
    pages = {421--575},
    month = {6},
    year = {2002}
    }

@book{BK,
    title = {Lectures on Tensor Categories and Modular Functors},
    author = {Bojko Bakalov and Kirillov, Jr., Alexander},
    publisher = {American Mathematical Society},
    year = {2001},
    volume = {21}
    }

@article{BHMV,
    title = {Three-manifold invariants derived from the Kauffman bracket},
    journal = {Topology},
    volume = {31},
    number = {4},
    pages = {685--699},
    year = {1992},
    issn = {0040-9383},
    doi = {doi:10.1016/0040-9383(92)90002-Y},
    author = {C. Blanchet and N. Habegger and G. Masbaum and P. Vogel}
    }

@misc{BDSV1,
    author = {Bruce Bartlett and Christopher L. Douglas and Christopher J. Schommer-Pries and Jamie Vicary},
    note = {Unpublished manuscript},
    title = {A finite presentation of the 3-dimensional bordism bicategory},
    }

@misc{BDSV2,
      title={Extended 3-dimensional bordism as the theory of modular objects}, 
      author={Bruce Bartlett and Christopher L. Douglas and Christopher J. Schommer-Pries and Jamie Vicary},
      year={2014},
      eprint={1411.0945},
      archivePrefix={arXiv},
      primaryClass={math.GT},
    }

@misc{BDSV3,
    author = {Bruce Bartlett and Christopher L. Douglas and Christopher J. Schommer-Pries and Jamie Vicary},
    note = {Unpublished manuscript},
    title = {Extensions of symmetric monoidal bicategories},
    }

@misc{BDSV4,
    title={Modular categories as representations of the 3-dimensional bordism 2-category}, 
    author={Bruce Bartlett and Christopher L. Douglas and Christopher J. Schommer-Pries and Jamie Vicary},
    year={2015},
    eprint={1509.06811},
    archivePrefix={arXiv},
    primaryClass={math.AT},
    }

@phdthesis{dr-thesis,
    title={Construction of extended topological quantum field theories},
    author={De Renzi, Marco},
    year={2017},
    school={Université Sorbonne Paris Cité},
    shorthand = {DeR17}
    }

@phdthesis{filippos-thesis,
    title = {On surgery presentations of bordism bicategories},
    author = {Sytilidis, Filippos Ilarion},
    year = {2025},
    school = {University of Oxford},
    url = {https://sites.google.com/view/fsytilidis/phd-thesis}
    }

@article{AS-Novikov,
    doi = {doi:10.2307/1970717},
    author = {Michael F. Atiyah and Isadore M. Singer},
    journal = {Annals of Mathematics},
    number = {3},
    pages = {546--604},
    title = {The Index of Elliptic Operators: III},
    volume = {87},
    year = {1968}
    }

@article{cohom-cat,
    title = {Cohomology of small categories},
    journal = {Journal of Pure and Applied Algebra},
    volume = {38},
    number = {2},
    pages = {187--211},
    year = {1985},
    issn = {0022-4049},
    doi = {10.1016/0022-4049(85)90008-8},
    author = {Hans-Joachim Baues and Günther Wirsching}
    }

@book{JY,
    title={2-dimensional categories}, 
    author={Niles Johnson and Donald Yau},
    isbn={9780198871385},
    year={2021},
    publisher={Oxford University Press},
    }

@misc{FST,
    title={Fully local Reshetikhin-Turaev theories}, 
    author={Daniel S. Freed and Claudia I. Scheimbauer and Constantin Teleman},
    year={2026},
    eprint={2601.05518v1},
    archivePrefix={arXiv},
    primaryClass={math.QA},
    }

@book{EGNO,
    title={Tensor Categories},
    author={Etingof, Pavel and Gelaki, Shlomo and Nikshych, Dmitri and Ostrik, Victor},
    volume={205},
    year={2015},
    publisher={American Mathematical Society},
    series={Mathematical Surveys and Monographs}
    }

@book{Kock, 
    series={London Mathematical Society Student Texts},
    title={Frobenius Algebras and 2-D Topological Quantum Field Theories},
    publisher={Cambridge University Press},
    author={Kock, Joachim},
    year={2003},
    collection={London Mathematical Society Student Texts}}

@inbook{Lawrence,
    author = {R. J. Lawrence},
    year = {1993},
    title = {Triangulations, Categories and Extended Topological Field Theories},
    booktitle = {Quantum Topology},
    chapter = {},
    pages = {191-208},
    doi = {10.1142/9789812796387_0011},
    }

@article{Mumford,
    title = {Abelian quotients of the Teichmüller modular group},
    journal = { Journal d’Analyse Mathématique},
    volume = {18},
    pages = {227-244},
    year = {1967},
    doi = {10.1007/BF02798046},
    author = {David Mumford},
    }

@article{Powell,
    doi = {10.2307/2043120},
    author = {Jerome Powell},
    journal = {Proceedings of the American Mathematical Society},
    number = {3},
    pages = {347--350},
    publisher = {American Mathematical Society},
    title = {Two Theorems on the Mapping Class Group of a Surface},
    volume = {68},
    year = {1978}
    }

@misc{Korkmaz,
      title={Low-dimensional homology groups of mapping class groups: a survey}, 
      author={Mustafa Korkmaz},
      year={2003},
      eprint={math/0307111},
      archivePrefix={arXiv},
      primaryClass={math.GT}
}

@article{Witten,
    author = {Witten, Edward},
    title = {Quantum Field Theory and the Jones Polynomial},
    doi = {10.1007/BF01217730},
    journal = {Commun. Math. Phys.},
    volume = {121},
    pages = {351--399},
    year = {1989}
    }

@article{Atiyah-2framings,
    title = {On framings of 3-manifolds},
    journal = {Topology},
    volume = {29},
    number = {1},
    pages = {1-7},
    year = {1990},
    issn = {0040-9383},
    doi = {10.1016/0040-9383(90)90021-B},
    author = {Michael Atiyah}
    }

@article{Abrams,
    author = {Abrams, Lowell},
    title = {Two-dimensional topological quantum field theories and Frobenius algebras},
    doi = {10.1142/S0218216596000333},
    journal = {Journal of Knot Theory and Its Ramifications},
    volume = {5},
    pages = {569--587},
    year = {1996}
    }

@article{Juhasz,
    title = {Defining and classifying TQFTs via surgery},
    journal = {Quantum Topology},
    volume = {9},
    year = {2018},
    number = {2},
    pages = {229--321},
    doi = {10.4171/QT/108},
    author = {Juhász, András}
    }

@article{Turaev-Maslov,
    title = {First symplectic Chern class and Maslov indices},
    journal = {Journal of Mathematical Sciences},
    volume = {37},
    year = {1987},
    pages = {1115--1127},
    doi = {10.1007/BF01086635},
    author = {Turaev, V. G.}
    }

@misc{haioun,
    title={Defining extended TQFTs via handle attachments}, 
    author={Benjamin Haïoun},
    year={2024},
    eprint={2412.14649},
    archivePrefix={arXiv},
    primaryClass={math.GT},
    }

@misc{pres-bord,
    title={A finite presentation of the three dimensional bordism bicategory},
    author={Bruce Bartlett and Christopher Douglas and Filippos Sytilidis},
    year={2026},
    note={In preparation}
}

@misc{partB,
    title={All once-extended 3D TQFTs are Reshetikhin--Turaev theories},
    author={Glen Lim},
    year={2026}
}
\end{document}